\documentclass[12pt]{amsart}

\usepackage{amscd,latexsym,amsthm,amsfonts,amssymb,amsmath,amsxtra}
\usepackage[mathscr]{eucal}
\usepackage{graphicx}
\usepackage{hyperref}
\renewcommand\theequation{\thesection.\arabic{equation}}

\newcommand{\BA}{{\mathbb {A}}}

\newcommand{\BC}{{\mathbb {C}}}

\newcommand{\BR}{{\mathbb {R}}}

\newcommand{\CA}{{\mathcal {A}}}

\newcommand{\CC}{{\mathcal {C}}}
\renewcommand{\CD}{{\mathcal {D}}}
\newcommand{\CE}{{\mathcal {E}}}

\newcommand{\CN}{{\mathcal {N}}}

\newcommand{\CS}{{\mathcal {S}}}

\newcommand{\cusp}{{\mathrm{cusp}}}

\newcommand{\disc}{{\mathrm{disc}}}

\newcommand{\GL}{{\mathrm{GL}}}

\newcommand{\Ind}{{\mathrm{Ind}}}

\newcommand{\Res}{{\mathrm{Res}}}

\newcommand{\Sym}{{\mathrm{Sym}}}

\newcommand{\Sp}{{\mathrm{Sp}}}

\newcommand{\tr}{{\mathrm{tr}}}

\newcommand{\wt}{\widetilde}

\newcommand{\ol}{\overline}
\newcommand{\ul}{\underline}

\newcommand{\bs}{\backslash}

\def\bks{{\backslash}}

\def\diag{{\rm diag}}

\def\lam{{\lambda}}

\newtheorem{thm}{Theorem}[section]

\newtheorem{lem}[thm]{Lemma}
\newtheorem{prop}[thm]{Proposition}
\newtheorem {conj}[thm]{Conjecture}

\newtheorem {ques/conj}[thm]{Question/Conjecture}

\newtheorem{rmk}[thm]{Remark}

\makeatletter

\newcommand{\Rmnum}[1]{\expandafter\@slowromancap\romannumeral #1@}
\makeatother

\begin{document}
\renewcommand{\theequation}{\arabic{equation}}
\numberwithin{equation}{section}

\title[E\MakeLowercase{xtension of }G\MakeLowercase{inzburg}-J\MakeLowercase{iang}-S\MakeLowercase{oudry correspondence}]
{O\MakeLowercase{n extension of }G\MakeLowercase{inzburg}-J\MakeLowercase{iang}-S\MakeLowercase{oudry correspondence to certain automorphic forms on} $Sp_{4mn}(\BA)$ \MakeLowercase{and} $\wt{Sp}_{4mn \pm 2n}(\BA)$}

\author{B\MakeLowercase{aiying} L\MakeLowercase{iu}}
\address{Department of Mathematics\\
University of Utah\\
155 S 1400 E Room 233, Salt Lake City, UT 84112-0090, USA.}
\email{liuxx969@umn.edu}


\date{August, 2013}



\begin{abstract}
Let $F$ be a number field, and $\BA=\BA_F$.
In this paper, first, we provide a family of global Arthur parameters confirming all parts of a general conjecture on the relation between the structure of Fourier coefficients and
the structure of global Arthur parameters, given by Jiang in 2012. Then we extend a correspondence between certain automorphic forms on $Sp_{4n}(\BA)$ and $\wt{Sp}_{2n}(\BA)$, given by Ginzburg, Jiang and Soudry in 2012, to certain automorphic forms on $Sp_{4mn}(\BA)$ and $\wt{Sp}_{4mn \pm 2n}(\BA)$, using the same idea of considering compositions of automorphic descent maps.
\end{abstract}

\maketitle

\section{Introduction}

Let $F$ be a number field, and $\BA=\BA_F$.
Fourier coefficients play important roles in the study automorphic forms.
For example, a basic and fundamental result in the theory of automorphic forms
for $GL_n(\BA)$ is that cuspidal automorphic forms are globally generic,
that is, have non-vanishing
Whittaker-Fourier coefficients, due to Shalika (\cite{S74}) and Piatetski-Shapiro (\cite{PS79}) independently.
This result has been extended to the discrete spectrum of $GL_n(\BA)$
in \cite{JL13a}.
As another example, for classical groups, Ginzburg, Rallis and Soudry (\cite{GRS11})
developed the theory of automorphic descent by studying certain Fourier coefficients
of special type residual representations, producing the inverse of special cases of Langlands functorial transfers from classical groups to the general linear groups.
The idea of automorphic descent has been generalized to explicit
constructions of endoscopy transfers for classical groups which can be found
in \cite{G12} and \cite{J12}.

For general connected reductive groups, there is a framework
of attaching Fourier coefficients to nilpotent orbits (for symplectic groups, see \cite{GRS03}, \cite{JL13b}). 
For classical groups $G$, nilpotent orbits are parametrized
by partitions and certain non-degenerate quadratic forms (\cite{W01}).
For any irreducible automorphic representation $\pi$ of $G(\BA)$, let 
$\mathfrak{p}^m(\pi)$ be the set of maximal partitions (under the natural ordering of partitions) providing non-vanishing 
Fourier coefficients for $\pi$ (for precise definition, see \cite{J12} and \cite{JL13b}). 

On the other hand, given a general connected reductive group $G$, a main theme in the 
theory of automorphic forms is to study the discrete spectrum, which consists of
cuspidal spectrum and residual spectrum.
For $G=GL_n$, the residual spectrum was constructed explicitly by
Moeglin and Waldspurger (\cite{MW89}). For symplectic and special orthogonal
groups, the discrete spectrum was classified by Arthur (\cite{Ar13}) up to automorphic $L^2$-packets parametrized by global Arthur parameters. 

Towards understanding the Fourier coefficients information of members in the automorphic $L^2$-packets, Jiang (\cite{J12}) made a conjecture, which relates the structure of the global Arthur parameter of an irreducible automorphic representation $\pi$ which is in the discrete spectrum, to the structure of $\mathfrak{p}^m(\pi)$. We recall the symplectic case of this conjecture as follows. Let $G_n = Sp_{2n}$, with symplectic form 
$$\begin{pmatrix}
0 & v_n\\
-v_n & 0
\end{pmatrix},$$
where $v_n$ is an $n \times n$ matrix with $1$'s on the second diagonal and $0$'s elsewhere. Fix a Borel subgroup $B=TU$, where the maximal torus $T$ consists of elements of the following form
$$
\diag(t_1,\cdots,t_n;t_n^{-1},\cdots,t_1^{-1})
$$
and the unipotent radical $U$ consists of all upper unipotent matrices in $G_n$.

The set of global Arthur parameters for the discrete spectrum of $G_n=Sp_{2n}$ is denoted by $\wt{\Psi}_2(Sp_{2n})$, the elements of which
are of the form
\begin{equation}\label{psin}
\psi:=\psi_1\boxplus\psi_2\boxplus\cdots\boxplus\psi_r,
\end{equation}
where $\psi_i$ are pairwise different simple global Arthur parameters of orthogonal type and have the form
$\psi_i=(\tau_i,b_i)$ with $\tau_i\in\CA_\cusp(a_i)$, $2n+1 = \sum_{i=1}^r a_ib_i$ (the dual group of $Sp_{2n}$ is $SO_{2n+1}(\BC)$),
and $\prod_i \omega_{\tau_i}^{b_i} = 1$ (the condition
on the central characters of the parameter), following \cite[Section 1.4]{Ar13}. More precisely, $\psi_i=(\tau_i,b_i)$
satisfies the following conditions:
if $\tau_i$ is of symplectic type (i.e., $L(s, \tau_i, \wedge^2)$ has a pole at $s=1$), then $b_i$ is even; if $\tau_i$ is of orthogonal type (i.e., $L(s, \tau_i, \Sym^2)$ has a pole at $s=1$), then $b_i$ is odd.

\begin{thm}[Theorem 1.5.2, \cite{Ar13}] For each global Arthur parameter $\psi\in\wt{\Psi}_2(Sp_{2n})$ there defines a global Arthur
packet $\wt{\Pi}_\psi$. The discrete spectrum of $Sp_{2n}(\BA)$ has the following decomposition
$$
L^2_\disc(Sp_{2n}(F)\bks Sp_{2n}(\BA))
\cong\oplus_{\psi\in\wt{\Psi}_2(Sp_{2n})}\oplus_{\pi\in\wt{\Pi}_\psi(\epsilon_\psi)}m_\psi\pi,
$$
where $\wt{\Pi}_\psi(\epsilon_\psi)$ denotes the subset of $\wt{\Pi}_\psi$ consisting of members which occur in the
discrete spectrum.
\end{thm}

As in \cite{J12}, $\wt{\Pi}_\psi(\epsilon_\psi)$ is called the automorphic $L^2$-packet attached to $\psi$. For $\psi$ of the form in \eqref{psin}, let $\underline{p}(\psi)=[(b_1)^{(a_1)}\cdots(b_r)^{(a_r)}]$.
For $\pi\in\wt{\Pi}_\psi(\epsilon_\psi)$, the structure of the global Arthur parameter $\psi$ deduces constraints on
the structure of $\frak{p}^m(\pi)$, which is given by the following conjecture.

\begin{conj}[Conjecture 4.2, \cite{J12}]\label{cubmfc}
For any $\psi\in\wt{\Psi}_2(Sp_{2n})$, let $\wt{\Pi}_{\psi}(\epsilon_\psi)$ be the automorphic $L^2$-packet attached to $\psi$.
Then the following hold.
\begin{enumerate}
\item[(1)] Any symplectic partition $\ul{p}$ of $2n$, if $\ul{p}>\eta_{{\frak{g}^\vee,\frak{g}}}(\underline{p}(\psi))$, does
not belong to $\frak{p}^m(\pi)$ for any $\pi\in\wt{\Pi}_{\psi}(\epsilon_\psi)$.
\item[(2)] For a $\pi\in\wt{\Pi}_{\psi}(\epsilon_\psi)$, any partition $\ul{p}\in\frak{p}^m(\pi)$ has the property that
$\ul{p}\leq \eta_{{\frak{g}^\vee,\frak{g}}}(\ul{p}(\psi))$.
\item[(3)] There exists at least one member $\pi\in\wt{\Pi}_{\psi}(\epsilon_\psi)$ having the property that
$\eta_{{\frak{g}^\vee,\frak{g}}}(\ul{p}(\psi))\in \frak{p}^m(\pi)$.
\end{enumerate}
Here $\eta_{{\frak{g}^\vee,\frak{g}}}$ denotes the Barbasch-Vogan duality map
$($see \cite[Definition A1]{BV85} and \cite[Section 3.5]{Ac03}$)$ from the partitions for the dual group $G^\vee(\BC)$ to
the partitions for $G$.
\end{conj}

We refer to \cite[Section 4]{J12} for more discussion on this conjecture and related topics. Part (1) of Conjecture \ref{cubmfc} is completely proved in \cite{JL13c}.
In the first part of this paper, we provide a family of global Arthur parameters confirming all parts of Conjecture \ref{cubmfc}.

Let $\tau$ be an irreducible unitary cuspidal automorphic representation of
$GL_{2n}(\BA)$, with the properties that
$L(s, \tau, \wedge^2)$ has a simple
pole at $s=1$, and 
$L(\frac{1}{2}, \tau) \neq 0$. 

By Theorem 2.3 of \cite{GJS12}, there is an irreducible representation $\wt{\pi}$
of $\wt{Sp}_{2n}(\BA)$, which is $\psi^{1}$-generic,
lifts weakly to $\tau$ with respect to $\psi$.

Let $\Delta(\tau, m)$ be a Speh representation in the discrete spectrum of 
$GL_{2mn}(\BA)$. For more information about Speh representations, we refer to \cite{MW89}, or the 
Section 1.2 of \cite{JLZ12}.

Let $P_r=M_rN_r$ be the maximal parabolic subgroup of 
$Sp_{2l}$ with Levi subgroup $M_r$ isomorphic to 
$GL_r \times Sp_{2l-2r}$. Using 
the normalization in \cite{Sh10},
the group $X_{M_{r}}^{Sp_{2l}}$ of all continues homomorphisms from 
$M_{r}(\BA)$ to $\BC^{\times}$, which is trivial 
on $M_{r}(\BA)^1$ (see \cite{MW95}), 
will be identified with $\BC$ by $s \rightarrow \lambda_s$. Let $\wt{P_r}(\BA)$ be the pre-image of $P_r(\BA)$ in $\wt{Sp}_{2l}(\BA)$.

For any $\phi \in A(N_{2mn}(\BA)M_{2mn}(F) \bs Sp_{4mn}(\BA))_{\Delta(\tau,m)}$, following \cite{L76} and \cite{MW95}, an
residual Eisenstein series
can be defined by
$$
E(\phi,s)(g)=\sum_{\gamma\in P_{2mn}(F)\bks Sp_{4mn}(F)}\lambda_s \phi(\gamma g).
$$
It converges absolutely for real part of $s$ large and has meromorphic continuation to the whole complex plane $\BC$. By \cite{JLZ12}, this Eisenstein series
has a simple pole at $\frac{m}{2}$, which is the right-most one.
Denote the representation generated by these residues at $s=\frac{m}{2}$
by $\CE_{\Delta(\tau, m)}$. 
This residual representation is square-integrable.
By Section 6.2 of \cite{JLZ12}, the global Arthur parameter of $\CE_{\Delta(\tau,m)}$ is 
$\psi=(\tau, 2m) \boxplus (1_{GL_1}, 1)$.

Our first main result can be stated as follows.

\begin{thm}\label{main1}
Assume that $F$ is any number field.
$$\mathfrak{p}^m(\CE_{\Delta(\tau, m)}) = [(2n)^{2m}].$$
\end{thm}

This theorem
was discussed by
Ginzburg in \cite{G08} with a quite sketchy argument. A fully detailed proof will be given in Section 2.

By Theorem 1.3, Proposition 6.4 and Remark 6.5 of \cite{JL13c}, Parts (1) and (2) of Conjecture \ref{cubmfc} hold for these global Arthur parameters $\psi=(\tau, 2m) \boxplus (1_{GL_1}, 1)$. Note that in this case
$$\eta_{{so_{2n+1}(\BC),sp_{2n}(\BC)}}(\ul{p}(\psi)) =\eta_{so_{2n+1}(\BC),sp_{2n}(\BC)}([(2m)^{2n}(1)])= [(2n)^{2m}].$$ 
Combining Theorem \ref{main1}, we can see that all parts of Conjecture \ref{cubmfc} are confirmed for this family of global Arthur parameters.

%


In \cite{GRS03}, for any irreducible cuspidal automorphic representation $\pi$
of symplectic groups or their double covers, Ginzburg, Rallis and Soudry found a maximal partition which has only even parts, providing non-vanishing Fourier coefficients for $\pi$. We denote this partition by $\ul{p}(\pi)$.

Next, we assume that $F$ is not totally imaginary, and consider $\ol{\CN}_{Sp_{4mn}}$, the set of 
irreducible cuspidal automorphic representations $\pi$ which are nearly equivalent to $\CE_{\Delta(\tau, m)}$ and 
$$
\ul{p}(\pi) = [(2n)^{2m-1}(2n_1)^{s_1}(2n_2)^{s_2} \cdots (2n_k)^{s_k}],
$$
with $2n \geq 2n_1 > 2n_2 > \cdots > 2n_k$, $k \geq 1$. Note that these partitions are less than or equal to $[(2n)^{2m}]$.

$\ol{\CN}_{Sp_{4mn}}$ can be naturally decomposed into a disjoint 
union of two sets $\CN_{Sp_{4mn}} \cup \CN'_{Sp_{4mn}}$,
where $\CN_{Sp_{4mn}}$ consists of elements 
having a nonzero Fourier coefficient $FJ_{\psi^{-1}_{n-1}}$ (for definition, see \cite[Section 3.2]{GRS11}),
while $\CN'_{Sp_{4mn}}$ consists of elements 
having no nonzero Fourier coefficients $FJ_{\psi^{-1}_{n-1}}$.

For any $\wt{\phi} \in A(N_{2kn}(\BA)\wt{M}_{2kn}(F) \bs \wt{Sp}_{4kn+2n}(\BA))_{\mu_{\psi} \Delta(\tau,k) \otimes \wt{\pi}}$, following \cite{L76} and \cite{MW95}, an
residual Eisenstein series
can be defined by
$$
\wt{E}(\wt{\phi},s)(g)=\sum_{\gamma\in P_{2kn}(F)\bks Sp_{4kn+2n}(F)}\lambda_s \wt{\phi}(\gamma g).
$$
It converges absolutely for real part of $s$ large and has meromorphic continuation to the whole complex plane $\BC$. By similar argument as that in \cite{JLZ12}, this Eisenstein series
has a simple pole at $\frac{k+1}{2}$, which is the right-most one.
Denote the representation generated by these residues at $s=\frac{k+1}{2}$
by $\wt{\CE}_{\Delta(\tau, k) \otimes \wt{\pi}}$. This residual representation is square-integrable.

Let $\CN'_{\wt{Sp}_{4(m-1)n+2n}}(\tau, \psi)$ be the set of irreducible genuine cuspidal
automorphic representations $\wt{\sigma}_{4(m-1)n+2n}$ of $\wt{Sp}_{4(m-1)n+2n}(\BA)$, which are nearly equivalent to
the residual representation $\wt{\CE}_{\Delta(\tau, m-1) \otimes \wt{\pi}}$,
have no nonzero Fourier coefficients $FJ_{\psi^{1}_{n-1}}$, and
$$
\ul{p}(\wt{\sigma}_{4(m-1)n+2n})=[(2n)^{2(m-1)}(2n_1)^{s_1}(2n_2)^{s_2} \cdots (2n_k)^{s_k}],
$$
with $2n \geq 2n_1 > 2n_2 > \cdots > 2n_k$, $k \geq 1$.

Let $\CN_{\wt{Sp}_{4mn+2n}}(\tau, \psi)$ be the set of irreducible genuine cuspidal
automorphic representations $\wt{\sigma}_{4mn+2n}$ of $\wt{Sp}_{4mn+2n}(\BA)$, which are nearly equivalent to
the residual representation $\wt{\CE}_{\Delta(\tau, m) \otimes \wt{\pi}}$,
have a nonzero Fourier coefficient $FJ_{\psi^{1}_{n-1}}$, and 
$$
\ul{p}(\wt{\sigma}_{4mn+2n})=[(2n)^{2m}(2n_1)^{s_1}(2n_2)^{s_2} \cdots (2n_k)^{s_k}],
$$
with $2n \geq 2n_1 > 2n_2 > \cdots > 2n_k$, $k \geq 1$.

For any $\wt{\sigma}_{4(m-1)n+2n} \in \CN'_{\wt{Sp}_{4(m-1)n+2n}}(\tau, \psi)$,
for any 
$$\wt{\phi} \in A(N_{2mn}(\BA)\wt{M}_{2mn}(F) \bs \wt{Sp}_{4mn+2n}(\BA))_{\mu_{\psi} \tau \otimes \wt{\sigma}_{4(m-1)n+2n}},$$
by similar calculation as in Pages 996-997 of \cite{GJS12},
it is easy to see that the corresponding Eisenstein series 
has a simple pole at $s=m$.
Let $\wt{\CE}_{\tau, \wt{\sigma}_{4(m-1)n+2n}}$ be the residual
representation of $\wt{Sp}_{4mn+2n}(\BA)$ generated by
the corresponding residues. This residual representation is square-integrable.

For any $\sigma_{4mn} \in \CN'_{{Sp}_{4mn}}(\tau, \psi)$, 
for any 
$${\phi} \in A(N_{2mn}(\BA){M}_{2mn}(F) \bs {Sp}_{4mn+2n}(\BA))_{\tau \otimes {\sigma}_{4mn}},$$
also by similar calculation as in Pages 996-997 of \cite{GJS12},
it is easy to see that the corresponding Eisenstein series
has a simple pole at $s=\frac{2m+1}{2}$. Let $\CE_{\tau, \sigma_{4mn}}$ be the residual
representation of $Sp_{4(m+1)n}(\BA)$ generated by
the corresponding residues. This residual representation is square-integrable.

For any $\sigma_{4mn} \in \CN_{{Sp}_{4mn}}(\tau, \psi)$,
let $\CD_{2n, \psi^{-1}}^{4mn}(\sigma_{4mn})$ be the $\psi^{-1}$-descent
of $\sigma_{4mn}$ from ${Sp}_{4mn}(\BA)$ to $\wt{Sp}_{4(m-1)n+2n}(\BA)$
(defined in Chapter 3 of \cite{GRS11}).
Note that by the tower property (see Theorem 7.10 of \cite{GRS11}),
$\CD_{2n, \psi^{-1}}^{4mn}(\sigma_{4mn})$ is cuspidal.

For any $\wt{\sigma}_{4mn+2n} \in \CN_{\wt{Sp}_{4mn+2n}}(\tau, \psi)$,
let $\CD_{2n, \psi^{1}}^{4mn+2n}(\wt{\sigma}_{4mn+2n})$ be the $\psi^{1}$-descent
of $\wt{\sigma}_{4mn+2n}$ from $\wt{Sp}_{4mn+2n}(\BA)$ to ${Sp}_{4mn}(\BA)$ (defined in Chapter 3 of \cite{GRS11}).
Note that by the tower property (see Theorem 7.10 of \cite{GRS11}),
$\CD_{2n, \psi^{1}}^{4mn+2n}(\wt{\sigma}_{4mn+2n})$ is also cuspidal.

Our second main result is that there are correspondences
between $\CN_{{Sp}_{4mn}}(\tau, \psi)$ and $\CN'_{\wt{Sp}_{4(m-1)n+2n}}(\tau, \psi)$, and between
$\CN_{\wt{Sp}_{4mn+2n}}(\tau, \psi)$ and $\CN'_{{Sp}_{4n}}(\tau, \psi)$,
as follows.

\begin{thm}\label{thm1}
Assume that $F$ is a number field which is not totally imaginary.

(1) There is a surjective map 
$$
\Psi: \CN_{{Sp}_{4mn}}(\tau, \psi) \rightarrow \CN'_{\wt{Sp}_{4(m-1)n+2n}}(\tau, \psi)
$$
$$
\sigma_{4mn} \mapsto \CD_{2n, \psi^{-1}}^{4mn}(\sigma_{4mn}).
$$

(2) If for any $\wt{\sigma}_{4(m-1)n+2n} \in \CN'_{\wt{Sp}_{4(m-1)n+2n}}(\tau, \psi)$,
$\wt{\CE}_{\tau, \wt{\sigma}_{4(m-1)n+2n}}$ is irreducible, then $\Psi$ is also injective.
\end{thm}

Note that the case of $m=1$ has already been proved by Ginzburg, Jiang and Soudry
in \cite{GJS12}. We use the same idea here to extend the correspondence to higher ranks.
Among others, one key idea is to consider compositions of automorphic descent maps.
Also note that they 
include $\CE_{\Delta(\tau, 1)}$ in the domain of the map $\Psi$. For simplicity, we just
let the domain of the map $\Psi$ consist of only irreducible
cuspidal automorphic representations, and consider the descent of $\CE_{\Delta(\tau, m)}$ separately in Section 7.

\begin{thm}\label{thm2}
Assume that $F$ is a number field which is not totally imaginary.

(1) There is a surjective map 
$$
\Psi: \CN_{\wt{Sp}_{4mn+2n}}(\tau, \psi) \rightarrow \CN'_{{Sp}_{4n}}(\tau, \psi)
$$
$$
\wt{\sigma}_{4mn+2n} \mapsto \CD_{2n, \psi^{1}}^{4mn+2n}(\wt{\sigma}_{4mn+2n}).
$$

(2) If for any $\sigma_{4mn} \in \CN'_{{Sp}_{4mn}}(\tau, \psi)$,
$\CE_{\tau, {\sigma}_{4mn}}$ is irreducible, then $\Psi$ is also injective.
\end{thm}

Due to the similarity of the proofs of Theorem \ref{thm1} and Theorem \ref{thm2},
we only give the proof for Theorem \ref{thm1}.

%
%

Theorem \ref{thm1} and Theorem \ref{thm2} together give us the following diagram
about correspondences between various sets of irreducible cuspidal automorphic representations:

\begin{figure}[h]\label{fig1}
\setlength{\unitlength}{0.2cm}
\begin{picture}(60,36)
\centering
\put(4,0){$\vdots$}
\put(0,8){$\wt{Sp}_{4mn-2n}(\BA)$}
\put(0,16){$Sp_{4mn}(\BA)$}
\put(0,24){$\wt{Sp}_{4mn+2n}(\BA)$}
\put(4,32){$\vdots$}
\put(16.2,0){$\vdots$}
\put(16,4){$\downarrow$}
\put(18,4){$\CD_{2n, \psi^{1}}^{4mn-2n}$}
\put(12,8){$\wt{\CN}_{4mn-2n}$}
\put(21,8){$\dot\bigcup$}
\put(24,8){$\wt{\CN}_{4mn-2n}'$}
\put(28,12){$\downarrow$}
\put(30,12){$\CD_{2n, \psi^{-1}}^{4mn}$}
\put(26,16){$\CN_{4mn}$}
\put(34,16){$\dot\bigcup$}
\put(38,16){$\CN_{4mn}'$}
\put(40,20){$\downarrow$}
\put(42,20){$\CD_{2n, \psi^{1}}^{4mn+2n}$}
\put(36,24){$\wt{\CN}_{4mn+2n}$}
\put(45,24){$\dot\bigcup$}
\put(48,24){$\wt{\CN}_{4mn+2n}'$}
\put(52,28){$\downarrow$}
\put(54,28){$\CD_{2n, \psi^{-1}}^{4mn+4n}$}
\put(52.2,32){$\vdots$}
\end{picture}
\end{figure}

In the above diagram, for short, we write that
$\CN_{{4mn}} := \CN_{Sp_{4mn}}$, $\CN_{{4mn}}' := \CN_{Sp_{4mn}}'$,
$\wt{\CN}_{{4mn \pm 2n}} := \CN_{\wt{Sp}_{4mn \pm 2n}}$,
$\wt{\CN}_{{4mn \pm 2n}}' := \CN_{\wt{Sp}_{4mn \pm 2n}}'$.

\begin{rmk}\label{rmk1}
In Theorem \ref{thm1} and Theorem \ref{thm2}, we assume that $F$ is a number field which is not totally imaginary,
the reason is that when $F$ is a totally imaginary number field, then
our construction will ``stop at some point, and can not go to higher levels". 
The explicit explanation of this phenomenon will appear elsewhere.
\end{rmk}

From Theorem \ref{main1}, for the residual representation 
$\CE_{\Delta(\tau, m)}$, we know that 
$\mathfrak{p}^m(\CE_{\Delta(\tau, m)}) = [(2n)^{2m}]$.
From its proof, and by Lemma 2.6 of \cite{GRS03} or Lemma 3.1 of \cite{JL13b},
we can see that it has a non-vanishing Fourier coefficient
attached to the partition $[(2n)1^{4mn-2n}]$ with respect
to the character $\psi_{[(2n)1^{4mn-2n}], -1}$. 
In Section 7, for any number field $F$,
we show that both $\CE_{\Delta(\tau, m)}$
and  $\CD^{4mn}_{2n, \psi^{-1}}(\CE_{\Delta(\tau, m)})$
are irreducible. The result can be stated as follows.

\begin{thm}\label{irre2}
Assume that $F$ is any number field.

(1) $\CD^{4mn}_{2n, \psi^{-1}}(\CE_{\Delta(\tau, m)})$ is square-integrable and is in the 
discrete spectrum.

(2)
Both $\CE_{\Delta(\tau, m)}$ and $\CD^{4mn}_{2n, \psi^{-1}}(\CE_{\Delta(\tau, m)})$
are irreducible.
\end{thm}

Note that in general, it is difficult to prove the irreducibility of certain descent representations. The case of $m=1$ of Theorem \ref{irre2} was proved in Theorem 4.1 of \cite{GJS12}, noting that by Theorem 2.5 of \cite{GJS12} $\CE_{\Delta(\tau, 1)}$ is irreducible. Also note that, the irreducibility of
$\CD^{4n}_{2n, \psi^{-1}}(\CE_{\Delta(\tau, 1)})$ actually has already been proved by Jiang and Soudry in \cite{JS03}, using different methods.

At the end of this introduction, we discuss the contents by section.
In Section 2, we will show Theorem \ref{main1},
whose proof is reduced to that of Lemma \ref{constantterm}, which 
will be given in
Section 3.
The Section 4, we will prove Part (1) of Theorem \ref{thm1},
whose proof is reduced to that of Theorem \ref{thm6}, which will be given in
Section 5. In Section 6, we completes the proof of Theorem \ref{thm1},
by proving its Part (2). In Section 7, we prove 
Theorem \ref{irre2}. We assume that $F$ is not totally imaginary only in Sections 4--6.

Acknowledgement. The author would like to take this opportunity to
express his deepest gratitude to his advisor Prof. Dihua Jiang, 
for introducing the author to the topics of Fourier coefficients of automorphic forms,
automorphic descent, and constructions of square-integrable 
automorphic representations, for sharing with the author his wonderful ideas
and insights to various problems of mathematics, for his
constant encouragement and support.
The author also would like to thank Prof. David Soudry for very helpful conversations on related topics of automorphic descent.

\section{Proof of Theorem \ref{main1}}

In this section, we prove Theorem \ref{main1}, which was discussed by
Ginzburg in \cite{G08} with a quite sketchy argument. To be complete, we give the full details here. $F$ is any number field in this section, Sections 3 and 7.

\begin{thm}[Ginzburg, Theorem 1 \cite{G08}]\label{main1part2}
$$\mathfrak{p}^m(\CE_{\Delta(\tau, m)}) = [(2n)^{2m}].$$
\end{thm}

\begin{proof}
By Theorem 1.3, Proposition 6.4 and Remark 6.5 of \cite{JL13c}, we only have to show that 
$\CE_{\Delta(\tau, m)}$ has a nonzero Fourier coefficient attached to
$[(2n)^{2m}]$.

We will prove this by induction on $m$. For $m=1$, this is proved in the book \cite{GRS11}.
Note that when $m=1$, $\CE_{\Delta(\tau, 1)}$ has a nonzero Fourier coefficient attached to 
the partition $[(2n)1^{2n}]$, and the descent to $\widetilde{Sp}_{2n}$ is generic
(see Theorem 3.1 of \cite{GRS11}).
Therefore, $\CE_{\Delta(\tau, 1)}$ has a nonzero Fourier coefficient attached to the 
composite partition $[(2n)1^{2n}] \circ [(2n)]$, which implies that 
$\CE_{\Delta(\tau, 1)}$ has a nonzero Fourier coefficient attached to the partition
$[(2n)^{2}]$ by Lemma 2.6 of \cite{GRS03}, or Lemma 3.1 of \cite{JL13b}.
For definition of composite partitions, we refer to Section 1 of \cite{GRS03}.

Now we assume that the theorem is true for the case of $m-1$,
and consider the case of $m \geq 2$.

Take any $\varphi \in \CE_{\Delta(\tau, m)}$, its Fourier coefficients
attached to $\ul{p} = [(2n)^{2m}]$ are of the following forms
\begin{equation}\label{main1equ1}
\varphi^{\psi_{\ul{p}, \ul{a}}}(g) = \int_{[V_{\ul{p},2}]}
\varphi(vg) \psi_{\ul{p}, \ul{a}}^{-1}(v)dv,
\end{equation}
where $\ul{a} = \{a_1, a_2, \ldots, a_{2m}\} \subset (F^*/ (F^*)^2)^{2m}$.
For definitions of the unipotent group $V_{\ul{p},2}$
and its character $\psi_{\ul{p}, \ul{a}}$, see Section 2 of \cite{JL13b}.

Assume that $T$ is the maximal split torus in $Sp_{4mn}$, 
consists of elements 
$$
\diag(t_1, t_2, \ldots, t_{2mn}, t_{2mn}^{-1}, \ldots, t_2^{-1}, t_1^{-1}).
$$

Let $\omega_1$ be the Weyl element of $Sp_{4mn}$, sending elements $t \in T$
to the following torus elements:
\begin{equation}\label{main1equ2}
t'=\diag(t^{(0)}, t^{(1)}, t^{(2)}, \ldots, t^{(n)}, t^{(n),*}, \ldots, t^{(2),*}, t^{(1),*}, t^{(0),*}),
\end{equation}
where 
\begin{align*}
t^{(0)} = & \diag(t_1, t_{n+1}, t_2, t_{n+2}, \ldots, t_i, t_{n+i}, 
\ldots, t_n, t_{2n})\\
t^{(j)} = & \diag(t_{2n+j}, t_{3n+j}, \ldots, t_{in+j}, \ldots, t_{(2m-1)n+j}),
\end{align*}
for $1 \leq j \leq n$.

Identify $Sp_{4(m-1)n}$ with its image in $Sp_{4mn}$ under 
the embedding $g \mapsto \diag(I_{2n}, g, I_{2n})$. Denote the restriction of $\omega_1$ to $Sp_{4(m-1)n}$
by $\omega_1'$.

Conjugating cross by $\omega_1$, the Fourier coefficient
$\varphi^{\psi_{\underline{p},\underline{a}}}$ becomes:
\begin{equation}\label{main1equ3}
\int_{[U_{\ul{p},2}]} \varphi(u \omega_1 g) \psi_{\underline{p},\underline{a}}^{\omega_1}(u)^{-1} du,
\end{equation}
where $U_{\ul{p},2} = \omega_1 V_{\ul{p},2} \omega_1^{-1}$,
and $\psi_{\underline{p},\underline{a}}^{\omega_1}(u) = \psi_{\underline{p},\underline{a}}(\omega_1^{-1}u\omega_1)$.

Now, we describe the structure of elements in $U_{\ul{p},2}$.
Any element in $U_{\underline{p},2}$
has the following form:
\begin{equation}\label{main1equ4}
u = 
\begin{pmatrix}
z_{2^n} & q_1 & q_2\\
0 & u' & q_1^*\\
0 & 0 & z_{2^n}^*
\end{pmatrix}
\begin{pmatrix}
I_{2n} & 0 & 0\\
p_1 & I_{(4m-4)n} & 0\\
0 & p_1^* & I_{2n}
\end{pmatrix},
\end{equation}
where $z_{2^n} \in V_{2^n}$, the unipotent radical of 
the parabolic $Q_{2^n}$ of $GL_{2n}$ with Levi isomorphic to 
$GL_2^{\times n}$;
$u' \in U_{[(2n)^{2m-2}],2} := \omega_1' V_{[(2n)^{2m-2}],2} \omega_1'^{-1}$;
$q_1 \in M_{2n\times (4m-4)n}$, 
$p_1 \in M_{(4m-4)n \times 2n}$, satisfy certain conditions,
which we do not specify at this moment;
$q_2 \in M_{(2n)\times (2n)}$, such that $q_2^t v_{2n} - v_{2n} q_2 = 0$,
where $v_{2n}$ is a matrix only with ones on the second diagonal.
Note that 
\begin{align*}
& \psi_{\underline{p},\underline{a}}^{\omega_1}
(\begin{pmatrix}
z_{2^n} & q_1 & q_2\\
0 & I_{(4m-4)n} & q_1^*\\
0 & 0 & z_{2^n}^*
\end{pmatrix}) \\
= & \psi(z_{2^n}(1,3)+\cdots + z_{2^n}(i,i+2) + \cdots + z_{2^n}(2n-2,2n))\\
\cdot & \psi(a_1 q_2(2n-1, 2) + a_2 q_2(2n,1)),
\end{align*}
where $a_1, a_2$ come from the 
$\ul{a} = \{a_1, a_2, \ldots, a_{2m}\}$ occurred in 
the Fourier coefficient 
$\varphi^{\psi_{\underline{p},\underline{a}}}$.

Since to show that $\mathcal{E}_{\Delta(\tau, m)}$ has a nonzero Fourier coefficient attached to the partition
$\ul{p} = [(2n)^{2m}]$, we only need to show that it
has a nonzero Fourier coefficient $\varphi^{\psi_{\underline{p},\underline{a}}}$
for some $\ul{a}$, we consider the following special type of $\ul{a}$:
$$
\ul{a} = \{1, -1, a_3, \ldots, a_{2m}\},
$$
where $a_3, \ldots, a_{2m}$ are arbitrary elements in $F^*/(F^*)^2$.

Let 
$A = 
\begin{pmatrix}
1 & -1\\
1 & 1
\end{pmatrix}$, and 
$\epsilon=\diag(A, \ldots, A; I_{(4m-4)n}; A^*, \ldots, A^*)$,
as in (2.31) of \cite{GJS12}.
Conjugating cross the integral in \eqref{main1equ3} by 
$\epsilon$, it becomes:
\begin{equation}\label{main1equ5}
\int_{[U^{\epsilon}_{\ul{p},2}]} \varphi(u \epsilon \omega_1 g) \psi_{\underline{p},\underline{a}}^{\omega_1, \epsilon}(u)^{-1} du,
\end{equation}
where $U^{\epsilon}_{\ul{p},2} = \epsilon U_{\ul{p},2} \epsilon^{-1}$ 
whose elements have the same structure as $U_{\ul{p},2}$ (see \eqref{main1equ4}),
and $\psi_{\underline{p},\underline{a}}^{\omega_1,\epsilon}(u) = \psi_{\underline{p},\underline{a}}^{\omega_1}(\epsilon^{-1} u \epsilon)$.

Note that now
\begin{align*}
& \psi_{\underline{p},\underline{a}}^{\omega_1, \epsilon}
(\begin{pmatrix}
z_{2^n} & q_1 & q_2\\
0 & I_{(4m-4)n} & q_1^*\\
0 & 0 & z_{2^n}^*
\end{pmatrix}) \\
= & \psi(z_{2^n}(1,3)+\cdots + z_{2^n}(i,i+2) + \cdots + z_{2^n}(2n-2,2n))\\
\cdot & \psi(q_2(2n-1, 1)).
\end{align*}
Note that the $a$ in (2.18) and (2.35) of \cite{GJS12} is $-1$ here.

Let $\nu$ be the following Weyl element of $\Sp_{4n}$
which is defined on Page 14 of \cite{GJS12}, 
also in (4.9) of \cite{GRS99}:
\begin{align}\label{main1equ6}
\begin{split}
\nu_{i,2i-1} & = 1, i = 1, \ldots, 2n,\\
\nu_{2n+i,2i} & = -1, i = 1, \ldots, n,\\
\nu_{2n+i,2i} & = 1, i = n+1, \ldots, 2n,\\
\nu_{i,j} & = 0, \text{ otherwise}.
\end{split}
\end{align}
Write $\nu$ as 
$\begin{pmatrix}
\nu_1 & \nu_2\\
\nu_3 & \nu_4
\end{pmatrix}$, where $\nu_i$'s are of size $2n \times 2n$. 

Let $\omega_2 = 
\begin{pmatrix}
\nu_1 & & \nu_2\\
& I_{(4m-4)n} & \\
\nu_3 & & \nu_4
\end{pmatrix}$, a Weyl element of $\Sp_{4mn}$. 
Conjugating cross the integral in \eqref{main1equ5}
by $\omega_2$, it becomes:
\begin{equation}\label{main1equ7}
\int_{[U^{\epsilon, \omega_2}_{\ul{p},2}]} \varphi(u \omega_2 \epsilon \omega_1 g) \psi_{\underline{p},\underline{a}}^{\omega_1, \epsilon, \omega_2}(u)^{-1} du,
\end{equation}
where $U^{\epsilon, \omega_2}_{\ul{p},2} = \omega_2 U^{\epsilon}_{\ul{p},2} \omega_2^{-1}$,
and $\psi_{\underline{p},\underline{a}}^{\omega_1,\epsilon, \omega_2}(u) = \psi_{\underline{p},\underline{a}}^{\omega_1, \epsilon}(\omega_2^{-1} u \omega_2)$.

Now let's describe the structure of elements in $U^{\epsilon, \omega_2}_{\ul{p},2}$.
Any element in $U^{\epsilon, \omega_2}_{\ul{p},2}$
has the following form:
\begin{equation}\label{main1equ8}
u = 
\begin{pmatrix}
z_{2n} & q_1 & q_2\\
0 & u' & q_1^*\\
0 & 0 & z_{2n}^*
\end{pmatrix}
\begin{pmatrix}
I_{2n} & 0 & 0\\
p_1 & I_{(4m-4)n} & 0\\
p_2 & p_1^* & I_{2n}
\end{pmatrix},
\end{equation}
where $z_{2n} \in V_{2n}$, the standard maximal unipotent subgroup of $GL_{2n}$ as before;
$u' \in U_{[(2n)^{2m-2}],2} := \omega_1' V_{[(2n)^{2m-2}],2} \omega_1'^{-1}$;
$q_1 \in M_{2n\times (4m-4)n}$, with $q_1(i,j) = 0$, for $j \leq (2m-2)i$;
$q_2, p_2 \in M_{(2n)\times (2n)}$, with $q_2(i,j) = p_2(i,j) = 0$, for $j \leq i$; 
$p_1 \in M_{(4m-4)n \times 2n}$, with $p_1(i,j) = 0$, for $i \geq (2m-2)(j-1)$.

Note that now
\begin{align}\label{main1equ9}
\begin{split}
& \psi_{\underline{p},\underline{a}}^{\omega_1,\epsilon, \omega_2}
(\begin{pmatrix}
z_{2n} & q_1 & q_2\\
0 & I_{(4m-4)n} & q_1^*\\
0 & 0 & z_{2n}^*
\end{pmatrix}) \\
= & \psi(z_{2n}(1,2)+\cdots + z_{2n}(n,n+1) \\
& - z_{2n}(n+1, n+2) - \cdots - z_{2n}(2n-1,2n)).
\end{split}
\end{align}

To proceed, we need to
define some unipotent subgroups. 
Let $R^1_i = \prod_{j=1}^i R^1_{i,j}$, for $1 \leq i \leq n$, 
where $R^1_{i,j} = \prod_{s=1}^{2m-2} X_{\alpha^i_{j,s}}$,
with $\alpha^i_{j,s} = e_i - e_{2n+(2m-2)(i-j+1)-s+1}$.
Let $R^1_i = \prod_{j=1}^i R^1_{i,j}$, for $n+1 \leq i \leq 2n-1$, 
where $R^1_{i,j} = \prod_{s=1}^{2m-2} X_{\alpha^i_{j,s}}$,
with $\alpha^i_{j,s} = e_i - e_{2n+(2m-2)(i-j+1)-s+1}$ if $j \geq i-n+1$,
and $\alpha^i_{j,s} = e_i + e_{2mn-(2m-2)(i-n)+(2m-1)(j-1)+s}$ if $j \leq i-n$.

Let $C^1_i = \prod_{j=1}^i C^1_{i,j}$, for $1 \leq i \leq n$, 
where $C^1_{i,j} = \prod_{s=1}^{2m-2} X_{\beta^i_{j,s}}$,
with $\beta^i_{j,s} = e_{2n+(2m-2)(i-j+1)-s+1} - e_{i+1}$.
Let $C^1_i = \prod_{j=1}^i C^1_{i,j}$, for $n+1 \leq i \leq 2n-1$, 
where $C^1_{i,j} = \prod_{s=1}^{2m-2} X_{\beta^i_{j,s}}$,
with $\beta^i_{j,s} = e_{2n+(2m-2)(i-j+1)-s+1}-e_{i+1}$ if $j \geq i-n+1$,
and $\beta^i_{j,s} = -e_{2mn-(2m-2)(i-n)+(2m-1)(j-1)+s}-e_{i+1}$ if $j \leq i-n$.

Let $R^2_i = \prod_{j=1}^i X_{\alpha^i_j}$, for $1 \leq i \leq n$,
with $\alpha^i_j = e_i+e_{2n-i+j}$, 
and $R^2_i = \prod_{j=1}^{2n-i} X_{\alpha^i_j}$, for $n+1 \leq i \leq 2n-1$,
with $\alpha^i_j = e_i+e_{i+j}$. 
Let $C^2_i = \prod_{j=1}^i X_{\beta^i_j}$, for $1 \leq i \leq n$,
with $\beta^i_j = -e_{2n-i+j}-e_{i+1}$, 
and $C^2_i = \prod_{j=1}^{2n-i} X_{\beta^i_j}$, for $n+1 \leq i \leq 2n-1$,
with $\beta^i_j = -e_{i+j}-e_{i+1}$. 

Now, we are ready to apply Lemma 2.3 of \cite{JL13b} repeatedly to the integration
over $\prod_{i=1}^{2n-1} C^2_i C^1_i$. We will consider the following pairs of groups:
\begin{align*}
& (R^2_1 R^1_1, C^2_1 C^1_1), \ldots, (R^2_n R^1_n, C^2_n C^1_n);\\
& (R^2_{n+1}, C^2_{n+1}), (R^1_{n+1}, C^1_{n+1});\\
& \ldots;\\
& (R^2_{2n-1}, C^2_{2n-1}), (R^1_{2n-1}, C^1_{2n-1}). 
\end{align*} 

Let $\wt{U}^{\epsilon, \omega_2}_{\ul{p},2}$ be the subgroup of $U^{\epsilon, \omega_2}_{\ul{p},2}$ with
$p_1$ and $p_2$-parts zero. 
Then, $U^{\epsilon, \omega_2}_{\ul{p},2} = \wt{U}^{\epsilon, \omega_2}_{\ul{p},2}
\prod_{i=1}^{2n-1} C^2_i C^1_i$. Let $W = U^{\epsilon, \omega_2}_{\ul{p},2}$,
$\wt{W} = \wt{U}^{\epsilon, \omega_2}_{\ul{p},2}$, and 
$\psi_W = \psi_{\underline{p},\underline{a}}^{\omega_1, \epsilon, \omega_2}$.

First, we apply Lemma 2.3 of \cite{JL13b} to $(R^2_1 R^1_1, C^2_1 C^1_1)$. For this, we need to consider the quadruple
$ (\wt{W} \prod_{i=2}^{2n-1} C^2_i C^1_i, \psi_W, X_{\alpha^1_1} R^1_{1,1}, X_{\beta^1_1} C^1_{1,1})$.
It is easy to see that this quadruple satisfies all the conditions for this lemma.
By this lemma, the integral in \eqref{main1equ7} is non-vanishing
if and only if the following integral is non-vanishing:
\begin{equation}\label{main1equ10}
\int_{[R^2_1 R^1_1 \wt{W} \prod_{i=2}^{2n-1} C^2_i C^1_i]} \varphi(rwc \omega_2 \epsilon \omega_1 g) \psi_W(w)^{-1}
dcdwdr.
\end{equation} 
Note that $R^2_1 = X_{\alpha^1_1}$, and $R^1_1 = R^1_{1,1}$.

Then we apply
Lemma 2.3 of \cite{JL13b} to $(R^2_2 R^1_2, C^2_2 C^1_2)$. For this, we need to consider the following sequence of quadruples:
\begin{align*}
& (R^2_1 R^1_1 \wt{W} \prod_{i=3}^{2n-1} C^2_iC^1_i X_{\beta^2_2}R^1_{2,2}, \psi_W, X_{\alpha^2_1}R^1_{2,1},
X_{\beta^2_1}C^1_{2,1}),\\
& (R^2_1 R^1_1 X_{\alpha^2_1} R^1_{2,1} \wt{W} \prod_{i=3}^{2n-1} C^2_iC^1_i, \psi_W, X_{\alpha^2_2}R^1_{2,2},
X_{\beta^2_2}C^1_{2,2}).
\end{align*}
Applying this lemma twice, we get that the integral in \eqref{main1equ10} is non-vanishing
if and only if the following integral is non-vanishing:
\begin{equation}\label{main1equ11}
\int_{[R^2_1R^2_2R^1_1R^1_2 \wt{W} \prod_{i=3}^{2n-1} C^2_iC^1_i]} \varphi(rwc \omega_2 \epsilon \omega_1 g) \psi_W(w)^{-1} dcdwdr.
\end{equation} 

Then we continue the above procedure, applying Lemma 2.3 of \cite{JL13b}
to pairs $(R^2_3 R^1_3, C^2_3 C^1_3), \ldots, (R^2_n R^1_n, C^2_n C^1_n)$. 
For the pair $(R^2_n R^1_n, C^2_n C^1_n)$, 
we need to consider the following sequence of quadruples:
\begin{align*}
& (\prod_{i=1}^{n-1} R^2_i R^1_i \wt{W} \prod_{i=n+1}^{2n-1} C^2_iC^1_i \prod_{s=2}^{n} X_{\beta^n_s}R^1_{n,s}, \psi_W, X_{\alpha^n_1}R^1_{s,1},
X_{\beta^n_1}C^1_{n,1}),\\
& \cdots,\\
& (\prod_{i=1}^{n-1} R^2_i R^1_i \prod_{s=1}^{n-1} X_{\alpha^l_s}R^1_{n,s} \wt{W} \prod_{i=n+1}^{2n-1} C^2_iC^1_i, \psi_W, X_{\alpha^n_n}R^1_{n,n},
X_{\beta^n_n}C^1_{n,n}).
\end{align*}
Applying Lemma 2.3 of \cite{JL13b} repeatedly, we get that the integral in \eqref{main1equ11} is non-vanishing
if and only if the following integral is non-vanishing:
\begin{equation}\label{main1equ12}
\int_{[\prod_{i=1}^{n} R^2_i R^1_i \wt{W} \prod_{i=n+1}^{2n-1} C^2_iC^1_i]} \varphi(rwc \omega_2 \epsilon \omega_1 g) \psi_W(w)^{-1} dcdwdr.
\end{equation} 

Before applying Lemma 2.3 of \cite{JL13b} to pairs
$(R^2_s, C^2_s), (R^1_s, C^1_s)$, $s = n+1, n+2, \ldots, 2n-1$, we need to 
take Fourier expansion along the one-dimensional root subgroup
$X_{e_s+e_s}$. And then we need to consider the pair 
$(R^2_s, C^2_s)$ first, then $(R^1_s, C^1_s)$.

For example, for $s=n+1$, we first take the Fourier expansion of the integral
in \eqref{main1equ12} along the one-dimensional root subgroup
$X_{e_s+e_s}$. Under the action of $GL_1$, we get two kinds of Fourier coefficients corresponding 
to the two orbits of the dual of $[X_{e_s+e_s}]$: the trivial one and the non-trivial one.
For the Fourier coefficients attached to the non-trivial orbit,
we can see that there is an inner integral $\varphi_{[(2n+2)1^{4mn-2n-2}],\{a\}}$, which is identically zero by Proposition 6.4 and Remark 6.5 of \cite{JL13c}.
Therefore only the Fourier coefficient attached to the trivial orbit,
which actually equals to the integral in \eqref{main1equ12},
survives. 

Then, we can apply Lemma 2.3 of \cite{JL13b} to pairs 
$$(R^2_{n+1}, C^2_{n+1}), (R^1_{n+1}, C^1_{n+1}).$$
We need to consider the following sequence of quadruples:
\begin{align*}
& (\prod_{i=1}^{n} R^2_i R^1_i \wt{W} 
X_{e_{n+1}+e_{n+1}} \prod_{i=n+2}^{2n-1} C^2_i
\prod_{t=n+1}^{2n-1} C^1_t \prod_{s=2}^{n-1} X_{\beta^n_s}, \psi_W, X_{\alpha^{n+1}_1},
X_{\beta^{n+1}_1}),\\
& \cdots,\\
& (\prod_{i=1}^{n} R^2_i R^1_i \prod_{s=1}^{n-2} X_{\alpha^n_s} \wt{W} X_{e_{n+1}+e_{n+1}} \prod_{i=n+2}^{2n-1} C^2_i
\prod_{t=n+1}^{2n-1} C^1_t, \psi_W, X_{\alpha^{n+1}_{n-1}},
X_{\beta^{n+1}_{n-1}}),\\
& (\prod_{i=1}^{n} R^2_i R^1_i R^2_{n+1} \wt{W} 
X_{e_{n+1}+e_{n+1}} \prod_{i=n+2}^{2n-1} C^2_i \prod_{t=n+2}^{2n-1} C^1_t \prod_{s=2}^{n+1} C^1_{n+1,s}, \psi_W, R^1_{n+1,1}, \\
& C^1_{n+1,1}),\\
& \cdots,\\
& (\prod_{i=1}^{n} R^2_i R^1_i R^2_{n+1} \prod_{s=1}^{n} R^1_{n+1,s}\wt{W} X_{e_{n+1}+e_{n+1}} \prod_{i=n+2}^{2n-1} C^2_i \prod_{t=n+2}^{2n-1} C^1_t, \psi_W, R^1_{n+1,n+1}, \\
& C^1_{n+1,n+1}),
\end{align*}
Applying Lemma 2.3 of \cite{JL13b} $2n$ times, we get that the integral in \eqref{main1equ12} is non-vanishing
if and only if the following integral is non-vanishing:
\begin{equation}\label{main1equ13}
\int_{[\prod_{i=1}^{n+1} R^2_i R^1_i \wt{W}
X_{e_{n+1}+e_{n+1}} \prod_{i=n+2}^{2n-1} C^2_iC^1_i]} \varphi(rwc \omega_2 \epsilon \omega_1 g) \psi_W(w)^{-1} dcdwdr.
\end{equation} 

After repeating the above procedure to the pairs $(R^2_s, C^2_s), (R^1_s, C^1_s)$, 
$s = n+1, n+2, \ldots, 2n-1$,
we get that the integral 
in \eqref{main1equ13} is non-vanishing if and only if the following integral is 
non-vanishing:
\begin{equation}\label{main1equ14}
\int_{[\prod_{i=1}^{2n-1}R^2_i R^1_i \wt{W}
\prod_{j=n+1}^{2n-1}X_{e_j+e_j}]} \varphi(rwx \omega_2 \epsilon \omega_1 g) \psi_W(w)^{-1} dxdwdr.
\end{equation} 
Now, let's see the structure of elements in $\prod_{i=1}^{2n-1}R^2_i R^1_i \wt{W}\prod_{j=n+1}^{2n-1}X_{e_j+e_j}$. Any element in $\prod_{i=1}^{2n-1}R^2_i R^1_i \wt{W}\prod_{j=n+1}^{2n-1}X_{e_j+e_j}$
has the following form:
$$
w = 
\begin{pmatrix}
z_{2n} & q_1 & q_2\\
0 & u' & q_1^*\\
0 & 0 & z_{2n}^*
\end{pmatrix},
$$
where $z_{2n} \in V_{2n}$, the standard standard maximal unipotent subgroup of 
$GL_{2n}$, $u' \in U_{[(2n)^{2m-2}],2} := \omega_1' V_{[(2n)^{2m-2}],2} \omega_1'^{-1}$,
$q_1 \in M_{(2n)\times (4m-4)l}$, with the last row zero,
$q_2 \in M_{(2n)\times (2n)}$, with $q_2(2n,1)=0$.
And by \eqref{main1equ9}, the restriction of $\psi_W$
to the $z_{2n}$-part gives a Whittaker character of 
$GL_{2n}$.
Note that $\psi_W(\begin{pmatrix}
z_{2n} & 0 & 0\\
0 & I_{4mn-4n} & 0\\
0 & 0 & z_{2n}^*
\end{pmatrix})=\psi(z_{2n}(1,2)+\cdots + z_{2n}(n,n+1) - z_{2n}(n+1, n+2) - \cdots - z_{2n}(2n-1,2n)).$

It is clear that the integral in \ref{main1equ14} is non-vanishing
if and only if the following integral is non-vanishing:
\begin{equation}\label{main1equ17}
\int_{[\prod_{i=1}^{2n-1}R^2_i R^1_i \wt{W}
\prod_{j=n+1}^{2n-1}X_{e_j+e_j}]} \varphi(rwx \omega_2 \epsilon \omega_1 g) \psi_W'(w)^{-1} dxdwdr,
\end{equation} 
where $\psi_W'(\begin{pmatrix}
z_{2n} & 0 & 0\\
0 & I_{4mn-4n} & 0\\
0 & 0 & z_{2n}^*
\end{pmatrix})=\psi(\sum_{i=1}^{2n-1} z_{2n}(i,i+1)).$

Then it is easy to see that the integral in \eqref{main1equ17}
has an inner integral which is exactly
$\varphi^{\psi_{N_{1^{2n-1}}}}$, using notation in Lemma \ref{lem2}.
On the other hand, we know that by Lemma \ref{lem2},
$\varphi^{\psi_{N_{1^{2n-1}}}} = \varphi^{\wt{\psi}_{N_{1^{2n}}}}$.
Therefore, the
integral in \eqref{main1equ14}
becomes
\begin{equation}\label{main1equ15}
\int_{[\wt{U}]} \varphi(u \omega_2 \epsilon \omega_1 g) \psi_{\wt{U}}(u)^{-1} du,
\end{equation} 
where any element in $\wt{U}$
has the following form:
$$
u = u(z_{2l}, u',q_1, q_2) =
\begin{pmatrix}
z_{2n} & q_1 & q_2\\
0 & u' & q_1^*\\
0 & 0 & z_{2n}^*
\end{pmatrix},
$$
where $z_{2n} \in V_{2n}$, the standard maximal unipotent subgroup of 
$GL_{2n}$, $u' \in U_{[(2n)^{2m-2}],2} := \omega_1' V_{[(2n)^{2m-2}],2} \omega_1'^{-1}$,
$q_1 \in M_{(2n)\times (4m-4)n}$, $q_2 \in M_{(2n)\times (2n)}$, such that $q_2^t v_{2n} - v_{2n} q_2 = 0$,
where $v_{2n}$ is a matrix only with ones on the second diagonal.

Hence, the integral in \eqref{main1equ15} can be written as
\begin{equation}\label{main1equ16}
\int_{u(z_{2n}, u',0, 0)} \varphi(u \omega_2 \epsilon \omega_1 g)_{P_{2n}} \psi_{\wt{U}}(u)^{-1} du,
\end{equation} 
where $\varphi_{P_{2n}}$ is the constant term of 
$\varphi$ along the parabolic subgroup $P_{2n}=M_{2n}U_{2n}$ of $\Sp_{4mn}$ with 
Levi isomorphic to $GL_{2n} \times Sp_{(4m-4)n}$.

By Lemma \ref{constantterm}, there is an automorphic function
$$
f \in A(N_{2n}(\BA)M_{2n}(F)\bs \Sp_{4mn}(\BA))_{\tau \lvert \cdot \rvert^{-\frac{2m-1}{2}}
\otimes \CE_{\Delta(\tau,m-1)}},
$$
such that 
$$
\varphi(g)_{P_{2n}} = f(g), \forall g \in Sp_{4mn}(\BA).
$$
Therefore, 
$\varphi(u \omega_2 \epsilon \omega_1 g)_{P_{2n}}$ is an automorphic form
in $\tau \lvert \cdot \rvert^{-\frac{2m-1}{2}}
\otimes \mathcal{E}_{\Delta(\tau,m-1)}$.
Note that the restriction of $\psi_{\wt{U}}$
to the $z_{2n}$-part gives us a Whittaker coefficient, and the restriction
to the $u'$-part gives a Fourier coefficient of $\mathcal{E}_{\Delta(\tau,m-1)}$
attached to the partition $[(2n)^{2m-2}]$ up to a conjugation of 
the Weyl element $\omega_1'$. On the other hand, $\tau$ is 
generic, and by induction assumption, $\mathcal{E}_{\Delta(\tau,m-1)}$
has a nonzero Fourier coefficient attached to the partition $[(2n)^{2m-2}]$.
Therefore, we can make the conclusion that 
$\mathcal{E}_{\Delta(\tau,m)}$
has a nonzero $\psi_{\underline{p},\underline{a}}$-Fourier coefficient
attached to the partition $[(2n)^{2m}]$,
for some $\ul{a} = \{1, -1, a_3, \ldots, a_{2m}\}$, $a_i \in F^*/(F^*)^2$,
for $3 \leq i \leq 2m$.
Hence, $\mathcal{O}(\mathcal{E}_{\Delta(\tau,m)}) = [(2n)^{2m}]$.

Since the proof of Lemma \ref{lem2} will be reduced to that of Lemma \ref{constantterm},  
we have completed the proof of the theorem, up to Lemma \ref{constantterm}.
\end{proof}

\begin{rmk}\label{rmk3}
By similar argument, we can also prove that
for the residual representation 
$\wt{\CE}_{\Delta(\tau, k) \otimes \wt{\pi}}$ of 
$\wt{Sp}_{4kn+2n}(\BA)$, 
$$\mathfrak{p}^m(\wt{\CE}_{\Delta(\tau, k) \otimes \wt{\pi}}) = [(2n)^{2k+1}].$$
We will not give the details here.
\end{rmk}

The following lemma is an $\Sp$-analogue of the Lemma 4.1 of \cite{JL13a}.
It gives formulas for certain constant terms
of automorphic forms in $\CE_{\Delta(\tau, m)}$.
The proof of this lemma will be given in the next section.

\begin{lem}\label{constantterm}
Let $P_{2ni} = M_{2ni} N_{2ni}$, $1 \leq i \leq m$ be the parabolic subgroup of $Sp_{4mn}$ with
Levi part $M_{2ni} \cong GL_{2ni} \times Sp_{4n(m-i)}$. Let $\varphi$
be an arbitrary automorphic form in $\CE_{\Delta(\tau, m)}$. 
Denote by $\varphi(g)_{P_{2ni}}$ the constant term of $\varphi$
along $P_{2ni}$.
Then for $1 \leq i \leq m-1$, there is an automorphic function
$$
f \in A(N_{2ni}(\BA)M_{2ni}(F)\bs Sp_{4mn}(\BA))_{\Delta(\tau, i)\lvert \cdot \rvert^{-\frac{2m-i}{2}}
\otimes \mathcal{E}_{\Delta(\tau,m-i)}},
$$
such that 
$$
\varphi(g)_{P_{2ni}} = f(g), \forall g \in Sp_{4mn}(\BA).
$$
For $i=m$, there is an automorphic function
$$
f \in A(N_{2mn}(\BA)M_{2mn}(F)\bs Sp_{4mn}(\BA))_{\Delta(\tau, m)\lvert \cdot \rvert^{-\frac{m}{2}}},
$$
such that 
$$
\varphi(g)_{P_{2mn}} = f(g), \forall g \in Sp_{4mn}(\BA).
$$
\end{lem}

Next, we prove two important lemmas, which can be viewed as
an $\Sp$-analogue of Lemmas 6.1 and 6.2 of \cite{JL13a}.

\begin{lem}\label{lem1}
Let $N_{1^{2mn}}$ be the standard maximal unipotent subgroup of $Sp_{4mn}$
consisting of upper triangular matrices. For $p \geq 2n$, define 
a character of $N_{1^{2mn}}$ as follows:
\begin{align} \label{lem1equ1}
\begin{split}
\psi_p^{\underline{\epsilon}}(n) := & \psi(n_{1,2} + \cdots n_{p-1,p} + n_{p,p+1})\\
\cdot & \psi(\epsilon_1 n_{p+1, p+2} + \cdots + \epsilon_{2mn-p}n_{2mn,2mn+1}),
\end{split}
\end{align}
where $n \in N_{1^{2mn}}$, $\epsilon_i \in \{0, 1\}$, for $1 \leq i \leq 2mn-p-1$,
$\epsilon_{2mn-p} \in F/(F^*)^2$,
$\underline{\epsilon}= \{\epsilon_1 ,\ldots, \epsilon_{2mn-p}\}$.
Then for any automorphic form $\varphi \in \CE_{\Delta(\tau,m)}$, the following
$\psi_p^{\underline{\epsilon}}$-Fourier coefficient is identically zero:
\begin{equation}\label{lem1equ2}
\varphi^{\psi_p^{\underline{\epsilon}}}(g) := \int_{[N_{1^{2mn}}]} \varphi(ng)\psi_p^{\underline{\epsilon}}(n)^{-1}
dn \equiv 0.
\end{equation}
\end{lem}

\begin{proof}
If $\epsilon_i \neq 0$, $\forall 1 \leq i \leq 2mn-p$,
then $\psi_p^{\underline{\epsilon}}$ is a generic character of $G_n$.
Since
$\mathcal{E}_{\Delta(\tau,m)}$ is not generic, it has no nonzero 
$\psi_p^{\underline{\epsilon}}$-Fourier coefficients,
i.e., $\varphi^{\psi_p^{\underline{\epsilon}}} \equiv 0$, $\forall \varphi \in \mathcal{E}_{\Delta(\tau,m)}$.

Assume that $1 \leq i \leq 2mn-p$ is the first number such that $\epsilon_i = 0$.
Then $\varphi^{\psi_p^{\underline{\epsilon}}}$ has an inner integral
$\varphi_{P_{p+i}}$, which is the constant term of $\varphi$
along $P_{p+i}$, the parabolic subgroup of $Sp_{4mn}$
with Levi isomorphic to $GL_{p+i} \times Sp_{4mn-2(p+i)}$.
Note that $p+i > 2n$. 

By the cuspidal support of $\mathcal{E}_{\Delta(\tau,m)}$, we can see that 
$\varphi_{P_{p+i}} = 0$ unless $p+i = 2ns$ with $2 \leq s \leq m$.
By Lemma \ref{constantterm}, there is an automorphic function
$$
f \in A(N_{2ns}(\BA)M_{2ns}(F)\bs Sp_{4mn}(\BA))_{\Delta(\tau, s)\lvert \cdot \rvert^{-\frac{2m-s}{2}}
\otimes \mathcal{E}_{\Delta(\tau,b-s)}},
$$
such that 
$$
\varphi(g)_{P_{2ns}} = f(g), \forall g \in Sp_{4mn}(\BA),
$$
Therefore, after taking the constant term, $\varphi(g)_{P_{2ns}}$ is an automorphic
function over $GL_{2ns}(\BA) \times Sp_{4n(m-s)}(\BA)$. 
Note that the character $\psi(n_{1,2}+\cdots+n_{2ns-1,2ns})$
gives a Whittaker character of $GL_{2ns}$. 
Since by Proposition 2.1 of \cite{JL13a}, $\Delta(\tau,s)$ is not generic for $s > 1$,
i.e., it has no nonzero Whittaker Fourier coefficients, 
we conclude that 
$\varphi^{\psi_p^{\underline{\epsilon}}} \equiv 0$.

This completes the proof of the lemma.
\end{proof}

\begin{lem}\label{lem2}
Let $N_{1^{p}}$ be the unipotent radical of the parabolic subgroup
$P_{1^{p}}$ of $Sp_{4mn}$ with Levi part isomorphic to
$GL_1^{\times p} \times Sp_{4mn-2p}$. Let 
$$
\psi_{N_{1^{p}}}(n) := \psi(n_{1,2}+ \cdots + n_{p,p+1}),
$$
and 
$$
\widetilde{\psi}_{N_{1^{p}}}(n) := \psi(n_{1,2}+ \cdots + n_{p-1,p}),
$$
be two characters of $N_{1^{p}}$. For any automorphic form $\varphi \in \mathcal{E}_{\Delta(\tau,m)}$,
define $\psi_{N_{1^{p}}}$ and $\widetilde{\psi}_{N_{1^{p}}}$-Fourier coefficients 
as follows:
\begin{equation}\label{lem2equ1}
\varphi^{\psi_{N_{1^{p}}}}(g) := \int_{[N_{1^{p}}]} \varphi(ng)\psi_{N_{1^{p}}}(n)^{-1}
dn,
\end{equation}
\begin{equation}\label{lem2equ2}
\varphi^{\widetilde{\psi}_{N_{1^{p}}}}(g) := \int_{[N_{1^{p}}]} \varphi(ng)\widetilde{\psi}_{N_{1^{p}}}(n)^{-1}
du.
\end{equation}

Then $\varphi^{\psi_{N_{1^{p}}}} \equiv 0, \forall p \geq 2n$;
$\varphi^{\psi_{N_{1^{2n-1}}}} = \varphi^{\widetilde{\psi}_{N_{1^{2n}}}}$.
\end{lem}

\begin{proof}
First we assume that $p \geq n$. 
Let $X_{e_{p+1}+e_{p+1}}$ be the root subgroup corresponding to the 
root $e_{p+1}+e_{p+1}$.
Since the conjugating action of $X_{e_{p+1}+e_{p+1}}$ normalizes $N_{1^{p}}$, 
and preserves the character $\psi_{N_{1^{p}}}$, we can take the Fourier expansion 
of $\varphi^{\psi_{N_{1^{p}}}}$ along $X_{e_{p+1}+e_{p+1}}$.

Under the action of $GL_1$, we will get Fourier coefficients corresponding the two orbits
of the dual of $X_{e_{p+1}+e_{p+1}}$: the trivial one and the non-trivial one. 
Note that the non-trivial orbit gives us Fourier coefficients which are exactly
the Fourier coefficients 
attached to $[(2(p+1))1^{4mn-2q-2}]$, with $p+1 >n$.
On the other hand, by Proposition 6.4 and Remark 6.5 of \cite{JL13c},
all the Fourier coefficients attached to the non-trivial orbit
vanish, and the Fourier coefficient attached to the trivial orbit,
i.e., the Fourier coefficient with respect to the trivial character,
survives. Hence, $\varphi^{\psi_{N_{1^{p}}}}$ becomes:
\begin{equation}\label{lem2equ3}
\int_{[X_{e_{p+1}+e_{p+1}}]} \int_{[N_{1^{p}}]} \varphi(nxg)\psi_{N_{1^{p}}}(n)^{-1} dn dx.
\end{equation}

Let $R_{i}$, $p+1 \leq i \leq 2mn-1$ be the following subgroup of $N_{1^{2mn}}$:
\begin{align*}
R_{i} := & \{n \in N_{1^{2mn}}| n(j,l)=0, \forall (j,l)\neq (p+1,w), \\
& p+2 \leq w \leq 4mn-p-1\}.
\end{align*}

Since $R_{p+1}$ normalizes the group $X_{e_{p+1}+e_{p+1}} N_{1^{p}}$,
and preserves the character $\psi_{N_{1^{p}}}$, we can take the Fourier expansion
of the integral in \eqref{lem2equ3} along $R_{p+1}$. Since the subgroup
$I_{p+1} \times Sp_{4mn-2p-2}$ (image of $Sp_{4mn-2p-2}$ under the 
embedding $g \mapsto \diag(I_{p+1}, g, I_{p+1})$ to $Sp_{4mn}$)
of $Sp_{4mn}$ acts on the dual space of $[R_{p+1}]$ with two orbits:
the trivial one and the non-trivial one,
the integral in \eqref{lem2equ3} becomes:
\begin{equation}\label{lem2equ4}
\sum_{\gamma} \varphi^{\psi_{N_{1^{p+1}}}}(\gamma g) + \varphi^{\widetilde{\psi}_{N_{1^{p}}}}(g),
\end{equation}
where $\gamma$ is in some quotient space which we will not specify here.

Then we continue the above process of Fourier expansion for 
the two kinds of Fourier coefficients
$\varphi^{\psi_{N_{1^{p+1}}}}$ and $\varphi^{\widetilde{\psi}_{N_{1^{p+1}}}}$
along the pair of groups $(X_{e_{p+2}+e_{p+2}}, R_{p+2})$.
We will get four kinds of Fourier coefficients:
$$
\int_{[N_{1^{p+2}}]} \varphi(ng)\psi_{N_{1^{p+2}}}^{\underline{\epsilon}}(n)^{-1}
dn,
$$
where
\begin{align*}
\psi_{N_{1^{p+2}}}^{\underline{\epsilon}}(n) := & \psi(n_{1,2} + \cdots n_{p-1,p} + n_{p,p+1})\\
\cdot & \psi(\epsilon_1 n_{p+1, p+2} + \epsilon_2 n_{p+2, p+3}),
\end{align*}
and $\underline{\epsilon} = \{\epsilon_1, \epsilon_2\}$,
$\epsilon_1, \epsilon_2 \in \{0,1\}$.
Then we can continue the above process of Fourier expansion 
for each of these four kinds of Fourier coefficients
along the pair of groups $(X_{e_{p+3}+e_{p+3}}, R_{p+3})$. 
We will get six kinds of Fourier coefficients:
$$
\int_{[N_{1^{p+3}}]} \varphi(ng)\psi_{N_{1^{p+3}}}^{\underline{\epsilon}}(n)^{-1}
dn,
$$
where
\begin{align*}
\psi_{N_{1^{p+3}}}^{\underline{\epsilon}}(n) := & \psi(n_{1,2} + \cdots n_{p-1,p} + n_{p,p+1})\\
\cdot & \psi(\epsilon_1 n_{p+1, p+2} + \epsilon_2 n_{p+2, p+3} + \epsilon_2 n_{p+3, p+4}),
\end{align*}
and $\underline{\epsilon} = \{\epsilon_1, \epsilon_2, \epsilon_3\}$,
$\epsilon_1, \epsilon_2, \epsilon_3 \in \{0,1\}$.

Keep repeating the above procedure, we will get Fourier coefficients 
of the following form:
$\varphi^{\psi_p^{\underline{\epsilon}}}$,
with $\underline{\epsilon} = \{\epsilon_1, \ldots, \epsilon_{2mn-p}\}$,
$\epsilon_i\in \{0,1\}$, $1 \leq i \leq 2mn-p-1$,
and $\epsilon_{2mn-p} \in F/(F^*)^2$. 

By Lemma \ref{lem1},
all Fourier coefficients of this kind are zero, if $p \geq 2n$. 
Therefore, $\varphi^{\psi_{N_{1^{p}}}} \equiv 0$, if $p \geq 2n$.

For $p=2n-1$, by \eqref{lem2equ4}, we can see that $\varphi^{\psi_{N_{1^{2n-1}}}} = \varphi^{\widetilde{\psi}_{N_{1^{2n}}}}$,
since $\varphi^{\psi_{N_{1^{2n}}}} \equiv 0$ by above discussion.

This completes the proof of the lemma.
\end{proof}

\section{Proof of Lemma \ref{constantterm}}

In this section, we prove Lemma \ref{constantterm}.
Before we start, we recall the
definition of the Eisenstein series in Section 1:
$$
E(\phi, s)(g) = \sum_{\gamma\in P_{2mn}(F)\bks Sp_{4mn}(F)}\lambda_s \phi(\gamma g),
$$
where $\phi \in A(N_{2mn}(\BA)M_{2mn}(F) \bs Sp_{4mn}(\BA))_{\Delta(\tau,m)}$.
Assume that $\varphi = \Res_{s=\frac{m}{2}} E(\phi, s)$. 

To compute $\varphi_{P_{2ni}} = (\Res_{s=\frac{m}{2}} E(\phi, s))_{P_{2ni}}$,
we use the fact that the residue operator and the constant term operator are interchangeable.
So, first, we are going to calculate the constant term of 
$E(\phi,s)$ along $P_{2ni}$.
We follow the calculation in Section 2 of \cite{JLZ12}.

\begin{align}\label{ctequ1}
\begin{split}
& E_{P_{2ni}}(\phi,s)(g)\\
= & \int_{N_{2ni}(F)\bks N_{2ni}(\BA)} E(\phi,s)(ug)du\\
= & \sum_{\omega^{-1} \in P_{2mn}\bks Sp_{4mn}/P_{2ni}}\sum_{\gamma \in M_{2ni}^{\omega}(F)\bs M_{2ni}(F)}\int_{[N_{2ni}^{\omega}]}
\int_{N_{2ni, \omega}(\BA)}\lam_s\phi({\omega}^{-1}\gamma u' u'' g)\\
& d u' d u''
\end{split}
\end{align}
where we define $M_{2ni}^{\omega}:=\omega P_{2mn}{\omega}^{-1}\cap M_{2ni}$ and 
$N_{2ni}^{\omega}:=\omega P_{2mn}{\omega}^{-1}\cap N_{2ni}$ and 
$[N_{2ni}^{\omega}]:=N_{2ni}^{\omega}(F)\bks N_{2ni}^{\omega}(\BA)$.
Note that the unipotent radical $N_{2ni}$ can be decomposed as a product 
$N_{2ni, \omega}N_{2ni}^{\omega}$, where $N_{2ni,\omega}$ satisfies $N_{2ni,\omega}\cap N_{2ni}^{\omega}=\{1\}$
and $N_{2ni}=N_{2ni,\omega}N_{2ni}^{\omega}=N_{2ni}^{\omega}N_{2ni,\omega}$.

Using the explicit calculation about the generalized Bruhat decomposition 
$P_{2mn}\bks Sp_{4mn}/P_{2ni}$
(see Chapter 4 of \cite{GRS11}), we can see that all the double cosets 
are killed by the cuspidal support of the Eisenstein series except
two, which have the following representatives: $\omega_0 = Id$,
and 
$$
\omega_1 = 
\begin{pmatrix}
0&0&I_{2ni}&0\\
I_{2n(m-i)}&0&0&0\\
0&0&0&I_{2n(m-i)}\\
0&-I_{2ni}&0&0
\end{pmatrix}.
$$

Then
\begin{equation}\label{ctequ2}
E_{P_{2ni}}(\phi,s)(g)=E_{P_{2ni}}(\phi,s)_{\omega_0}(g)+E_{P_{2ni}}(\phi,s)_{\omega_1}(g),
\end{equation}
where
\begin{align*}
& E_{P_{2ni}}(\phi,s)_{\omega_j}(g)\\
= & \sum_{\gamma \in M_{2ni}^{\omega_j}(F)\bs M_{2ni}(F)}\int_{[N_{2ni}^{\omega_j}]}
\int_{N_{2ni, \omega_j}(\BA)}\lam_s\phi({\omega_j}^{-1}\gamma u' u'' g)d u' d u'',
\end{align*}
$j=0,1$.

We will consider these two terms separately in the next two subsections.

\subsection{$\omega_0$-term}
Write elements in $N_{2ni}$ as follows:
$$
u(X,Z,W)=\begin{pmatrix}
I_{2ni}&X&Z&W\\
&I_{2n(m-i)}&&Z^{'}\\
&&I_{2n(m-i)}&X^{'}\\
&&&I_{2ni}
\end{pmatrix}.
$$

Note that $P_{2mn}\cap M_{2ni}\bks M_{2ni}\cong P_{2n(m-i)}\bks Sp_{4n(m-i)}$,
where $P_{2n(m-i)}$ is the parabolic subgroup of $Sp_{4n(m-i)}$
with Levi isomorphic to $GL_{2n(m-i)} \times 1_{Sp_0}$. 
Then the $\omega_0$-term of the constant term can be written as
\begin{equation}\label{ctequ3}
E_{P_{2ni}}(\phi,s)_{\omega_0}(g)
=\sum_{\gamma\in P_{2n(m-i)}(F)\bks Sp_{4n(m-i)}(F)}\int_{[N_{2ni}]}\lam_s\phi(\gamma u g)du.
\end{equation}

The integral can be calculated as follows:

\begin{align*}
& \int_{[N_{2ni}]}\lam_s\phi(\gamma u g)du\\
= & \int_{[N_{2ni}]}\lam_s\phi(u \gamma g)du\\
= & \int_{[M_{2ni\times 2n(m-i)}]}\int_{[N_{2mn}\cap U_{2ni}]}\lam_s\phi( u'u(X)\gamma g)d u' d X\\
= & \int_{[M_{2ni\times 2n(m-i)}]}\lam_s\phi(u(X)\gamma g)d X,
\end{align*}
where $u(X)=u(X,0,0)$ with $X \in M_{2ni\times 2n(m-i)}$.

As in Section 2.2 of \cite{JLZ12}, the integral
$\int_{[M_{2ni\times 2n(m-i)}]}\lam_s\phi(u(X)\gamma g)d X$
can be viewed as the constant term of the automorphic
function in $\Delta(\tau,m)$: $x \mapsto \phi(\diag(x,x^*)g)$,
along the maximal parabolic subgroup $Q_{2ni,2n(m-i)}$ of $GL_{2mn}$
with Levi isomorphic to $GL_{2ni} \times GL_{2n(m-i)}$.
We will denote it by $\phi_{Q_{2ni,2n(m-i)}}$.

Let $P_{2ni,2n(m-i)} = M_{2ni,2n(m-i)}N_{2ni,2n(m-i)}$ be a standard parabolic subgroup of $Sp_{4mn}$, whose Levi 
$M_{2ni,2n(m-i)} \cong \GL_{2ni}\times \GL_{2n(m-i)}\times 1_{Sp_0}$.

The following lemma is parallel to Lemma 2.1 of \cite{JLZ12}.

\begin{lem}\label{omega0term}
The constant term $\lam_s\phi_{Q_{2ni,2n(m-i)}}$ belongs to the following space
$$
 A(N_{2ni,2n(m-i)}(\BA)M_{2ni,2n(m-i)}(F)\bks Sp_{4mn}(\BA))_{\Delta(\tau,i) |\cdot|^{s-\frac{m-i}{2}}\otimes 
 \Delta(\tau,m-i)|\cdot|^{s+\frac{i}{2}}\otimes 1_{Sp_0}}.
$$
\end{lem}

\begin{proof}
The proof is omitted here, since it is almost the same as that of Lemma 2.1 of \cite{JLZ12}, word-by-word.
\end{proof}

Therefore, by \eqref{ctequ3} and Lemma \ref{omega0term}, we can see that
$E_{P_{2ni}}(\phi,s)_{\omega_0}$ 
 belongs to the following space
$$
 A(N_{2ni}(\BA)M_{2ni}(F)\bks Sp_{4mn}(\BA))_{\Delta(\tau,i) |\cdot|^{s-\frac{m-i}{2}}\otimes 
 (\Delta(\tau,m-i)|\cdot|^{s+\frac{i}{2}}\rtimes 1_{Sp_0})}.
$$
Hence, the residue operator will kill 
$E_{P_{2ni}}(\phi,s)_{\omega_0}$ , since $s=\frac{m}{2}$ is 
not a pole of the Eisenstein series on $Sp_{4n(m-i)}$ with
inducing data $\Delta(\tau,m-i)|\cdot|^{s+\frac{i}{2}}\otimes 1_{Sp_0}$.

\subsection{$\omega_1$-term}
Note that $N_{2ni}^{\omega_1}=\{ u(Z)=u(0,Z,0)\mid Z\in M_{2ni\times 2n(m-i)}\}$.
and $M_{2ni}^{\omega_1}(F)\bks M_{2ni}(F)$ is isomorphic to $P_{2n(m-i)}(F)\bks Sp_{4n(m-i)}(F)$.

Therefore, we have
\begin{align}\label{ctequ4}
\begin{split}
&E_{P_{2ni}}(\phi,s)_{\omega_1}(g)\\
=&\sum_{\gamma\in P_{2n(m-i)}(F)\bks Sp_{4n(m-i)}}\int_{N_{2ni,\omega_1}(\BA)}
\int_{[M_{2ni\times 2n(m-i)}]}\lam_s\phi(\omega_1^{-1}\gamma u(Z) u g)d Z d u\\
=&\sum_{\gamma\in P_{2n(m-i)}(F)\bks Sp_{4n(m-i)}}\int_{N_{2ni,\omega_1}(\BA)}
\int_{[M_{2n(m-i)\times2ni}]}\lam_s\phi(u'(X) \omega_1^{-1} u \gamma g)d X du,
\end{split}
\end{align}
where
$$u'(X)=
\begin{pmatrix}
I_{2n(m-i)}&X&&\\
&I_{2ni}&&\\
&&I_{2ni}&X'\\
&&&I_{2n(m-i)}
\end{pmatrix} \text{ for } X\in M_{2n(m-i)\times 2ni}.
$$

As in case of $\omega_0$, the integral
$\int_{[M_{2n(m-i)\times2ni}]} \phi(u'(X)g)d X$
can be viewed as the constant term of the automorphic
function in $\Delta(\tau,m)$: $x \mapsto \phi(\diag(x,x^*)g)$,
along the maximal parabolic subgroup $Q_{2n(m-i),2ni}$ of $GL_{2mn}$
with Levi isomorphic to $GL_{2n(m-i)} \times GL_{2ni}$.
We will denote it by $\phi_{Q_{2n(m-i),2ni}}$.

Let $P_{2n(m-i),2ni} = M_{2n(m-i),2ni}U_{2n(m-i),2ni}$ be a standard parabolic subgroup of $Sp_{4mn}$, whose Levi 
$M_{2n(m-i),2ni} \cong \GL_{2n(m-i) \times \GL_{2ni}}\times 1_{Sp_0}$.
Then, by Lemma \ref{omega0term}, 
the constant term $\lam_s\phi_{Q_{2n(m-i),2ni}}$ belongs to the following space
$$
 A(N_{2n(m-i),2ni}(\BA)M_{2n(m-i),2ni}(F)\bks Sp_{4mn}(\BA))_{\Delta(\tau,m-i) |\cdot|^{s-\frac{i}{2}}\otimes 
 \Delta(\tau,i)|\cdot|^{s+\frac{m-i}{2}}\otimes 1_{Sp_0}}.
$$

Note that the outer integral in \eqref{ctequ4} is the intertwining operator
corresponding to $\omega_1$, which 
maps
$$
 A(N_{2n(m-i),2ni}(\BA)M_{2n(m-i),2ni}(F)\bks Sp_{4mn}(\BA))_{\Delta(\tau,m-i) |\cdot|^{s-\frac{i}{2}}\otimes 
 \Delta(\tau,i)|\cdot|^{s+\frac{m-i}{2}}\otimes 1_{Sp_0}}
$$
to
$$
 A(N_{2ni,2n(m-i)}(\BA)M_{ai,a(b-i)}(F)\bks Sp_{4mn}(\BA))_{ \Delta(\tau,i)|\cdot|^{-(s+\frac{m-i}{2})} \otimes 
\Delta(\tau,m-i) |\cdot|^{s-\frac{i}{2}}\otimes 1_{Sp_0}}.
$$
Note that $\tau$ is self-dual.

Therefore, by \eqref{ctequ4}, and the the above discussion, 
we can see that 
$E_{P_{2ni}}(\phi,s)_{\omega_1}$ 
 belongs to the following space
$$
 A(N_{2ni}(\BA)M_{2ni}(F)\bks Sp_{4mn}(\BA))_{ \Delta(\tau,i)|\cdot|^{-(s+\frac{m-i}{2})} \otimes 
(\Delta(\tau,m-i) |\cdot|^{s-\frac{i}{2}}\rtimes 1_{Sp_0})}.
$$
And for $1 \leq i \leq m-1$, after taking the residue operator,
$\Res_{s=\frac{m}{2}} E_{P_{2ni}}(\phi,s)_{\omega_1}$ 
belongs to the following space
$$
A(N_{2ni}(\BA)M_{2ni}(F)\bks Sp_{4mn}(\BA))_{\Delta(\tau, i)\lvert \cdot \rvert^{-\frac{2m-i}{2}}
\otimes \mathcal{E}_{\Delta(\tau,m-i)}},
$$
since $s=\frac{m}{2}$ is the rightmost simple pole of the (normalized)
Eisenstein series with inducing data $\Delta(\tau,m-i) |\cdot|^{s-\frac{i}{2}}\otimes 1_{Sp_0}$,
and it is not a pole of the intertwining operator corresponding to
$\omega_1$.

Therefore, for $1 \leq i \leq m-1$,
\begin{align*}
& \varphi_{P_{2ni}} \\
= & (\Res_{s=\frac{m}{2}} E(\phi,s))_{P_{2ni}} \\
= & \Res_{s=\frac{m}{2}} (E_{P_{2ni}}(\phi,s)) \\
= & \Res_{s=\frac{m}{2}} (E_{P_{2ni}}(\phi,s)_{\omega_1}),
\end{align*}
belongs to the following space
$$
A(N_{2ni}(\BA)M_{2ni}(F)\bks Sp_{4mn}(\BA))_{\Delta(\tau, i)\lvert \cdot \rvert^{-\frac{2m-i}{2}}
\otimes \mathcal{E}_{\Delta(\tau,m-i)}}.
$$

For $i=m$, $\varphi_{P_{2mn}} = \Res_{s=\frac{m}{2}} (E_{P_{2mn}}(\phi,s)_{\omega_1})$
belongs to the following space
$$
A(N_{2mn}(\BA)M_{2mn}(F)\bks Sp_{4mn}(\BA))_{ \Delta(\tau,m)|\cdot|^{-\frac{m}{2}} \otimes 1_{Sp_0}},
$$
since $E_{P_{2mn}}(\phi,s)_{\omega_1}$ 
belongs to the following space
$$
A(N_{2mn}(\BA)M_{2mn}(F)\bks Sp_{4mn}(\BA))_{ \Delta(\tau,m)|\cdot|^{-s} \otimes 1_{Sp_0}},
$$
and $s=\frac{m}{2}$ is a simple pole of the intertwining operator 
corresponding to $\omega_1$.

This completes the proof of Lemma \ref{constantterm}.

\section{Proof of Part (1) of Theorem \ref{thm1}}

In this section, we will prove that $\Psi$ is surjective,
which comes from 
the computation of composition of 
two descents
$$
\CD_{2n, \psi^{-1}}^{4mn} \circ \CD_{2n, \psi^{1}}^{4mn+2n}
(\wt{\CE}_{\tau, \wt{\sigma}_{4(m-1)n+2n}}),
$$
where $\wt{\sigma}_{4(m-1)n+2n} \in \CN'_{\wt{Sp}_{4(m-1)n+2n}}(\tau, \psi)$,
and $$
\CD_{2n, \psi^{1}}^{4mn+2n}
(\wt{\CE}_{\tau, \wt{\sigma}_{4(m-1)n+2n}}) \subset 
\CN_{{Sp}_{4mn}}(\tau, \psi).
$$
It turns out that there is a similar identity:
$$
\CD_{2n, \psi^{-1}}^{4mn} \circ \CD_{2n, \psi^{1}}^{4mn+2n}
(\wt{\CE}_{\tau, \wt{\sigma}_{4(m-1)n+2n}})=\wt{\sigma}_{4(m-1)n+2n},
$$
as in Proposition 5.2 of \cite{GJS12}.
We will prove in the next section that $\Psi$ is well-defined, i.e.,
$\CD_{2n, \psi^{-1}}^{4mn}(\sigma_{4mn})$ is irreducible (see Theorem \ref{thm6}).

Note that, from this section to Section 6, we assume that $F$ is a number field which is not totally imaginary, unless specified.

To start, we recall Proposition 4.1 of \cite{JL13c} which 
generalizes Theorem 6.3 of \cite{GRS11} and is true for any number field.

\begin{prop}\label{prop1}
Assume that $F$ is any number field.

(1) Let $\mu_i$, $1 \leq i \leq r$, be characters of $F_v^*$, $a \in F_v^*$,
then
\begin{align}\label{prop1equ1}
\begin{split}
 & FJ_{\psi^a_{k-1}} (\Ind_{P_{m_1, \ldots, m_k}}^{Sp_{2n}} \nu^{\alpha_1} \chi_1 (det_{GL_{m_1}}) \otimes \cdots \otimes \nu^{\alpha_k} \chi_k(det_{GL_{m_k}}))\\
\cong & \Ind_{P_{m_1-1, \ldots, m_k-1}}^{\wt{Sp}_{2n-2k}} \mu_{\psi^{-a}} \nu^{\alpha_1} \chi_1 (det_{GL_{m_1-1}}) \otimes \cdots \otimes \nu^{\alpha_k} \chi_k(det_{GL_{m_k-1}}).
\end{split}
\end{align}

(2) Let $\mu_i$, $1 \leq i \leq r$, be characters of $F_v^*$, $a, b \in F_v^*$,
then
\begin{align}\label{prop1equ2}
\begin{split}
 & FJ_{\psi^b_{k-1}} (\Ind_{P_{m_1, \ldots, m_k}}^{\wt{Sp}_{2n}} \mu_{\psi^{a}}\nu^{\alpha_1} \chi_1 (det_{GL_{m_1}}) \otimes \cdots \otimes \nu^{\alpha_k} \chi_k(det_{GL_{m_k}}))\\
\cong & \Ind_{P_{m_1-1, \ldots, m_k-1}}^{{Sp}_{2n-2k}} \chi_{\frac{b}{a}} \nu^{\alpha_1} \chi_1 (det_{GL_{m_1-1}}) \otimes \cdots \otimes \nu^{\alpha_k} \chi_k(det_{GL_{m_k-1}}),
\end{split}
\end{align}
where $\chi_{\frac{b}{a}}$ is a quadratic character of $F_v^*$ defined by
Hilbert symbol as follows: 
$\chi_{\frac{b}{a}}(x)=(\frac{b}{a}, x)$, and $\nu=\lvert \det(\cdot) \rvert$.
\end{prop}


Note that when $a=b=1$, Proposition \ref{prop1} is exactly Theorem 6.3 of \cite{GRS11}.

Next, we prove the equality mentioned at the beginning of
this section, which will imply later that $\Psi$ is surjective.

\begin{thm}\label{thm7}
(1)
$$
\CD_{2n, \psi^{-1}}^{4mn} \circ \CD_{2n, \psi^{1}}^{4mn+2n}
(\wt{\CE}_{\tau, \wt{\sigma}_{4(m-1)n+2n}}) \neq 0.
$$

(2)
$$
\CD_{2n, \psi^{-1}}^{4mn} \circ \CD_{2n, \psi^{1}}^{4mn+2n}
(\wt{\CE}_{\tau, \wt{\sigma}_{4(m-1)n+2n}})=\wt{\sigma}_{4(m-1)n+2n}.
$$
\end{thm}

\begin{proof}
The proof of Part (1) is similar to that of Theorem \ref{main1part2}
and to that of Theorem 2.1 of \cite{GJS12}.
The proof of Part (2) is similar to those of Theorem 5.1 and Proposition 
5.2 of \cite{GJS12}.

\textbf{Proof of Part (1)}.
Note that descents $\CD_{2n, \psi^{1}}^{4mn+2n}$
and $\CD_{2n, \psi^{-1}}^{4mn}$ are defined in 
Section 3.2 of \cite{GRS11}.

By Corollary 2.4 of \cite{JL13b} or Lemma 1.1 of \cite{GRS03}, and the discussion at the end of Section 1 of \cite{GRS03},
$\CD_{2n, \psi^{-1}}^{4mn} \circ \CD_{2n, \psi^{1}}^{4mn+2n}
(\wt{\CE}_{\tau, \wt{\sigma}_{4(m-1)n+2n}}) \neq 0$
if and only if the following integral is non-vanishing:
\begin{align}\label{thm7equ1}
\begin{split}
& \int_{[Y_2 V_{[(2n)1^{4mn-2n}],2}]} \int_{[Y_1 V_{[(2n)1^{4mn}],2}]}
\wt{\varphi}(v v_1 g) \\
& \psi_{[(2n)1^{4mn}],1}^{-1}(v) 
\psi_{[(2n)1^{4mn-2n}],-1}^{-1}(v_1) dv dv_1,
\end{split}
\end{align}
where $\wt{\varphi} \in \wt{\CE}_{\tau, \wt{\sigma}_{4(m-1)n+2n}}$,
$g \in \wt{Sp}_{4mn-2n}(\BA)$, embedded into $\wt{Sp}_{4mn+2n}(\BA)$
via the map $g \mapsto \diag(I_{2n}, g, I_{2n})$;
$Y_1$, $Y_2$ are the groups defined in (2.5) of \cite{JL13b}
corresponding to the partitions $[(2n)1^{4mn}]$ and $[(2n)1^{4mn-2n}]$
respectively; $V_{[(2n)1^{4mn-2n}],2}$ and $V_{[(2n)1^{4mn}],2}$ are defined in Section 2 of \cite{JL13b}.

Explicitly, let  
$N_{1^n}$ be the unipotent radical of the parabolic ${P}_{1^n}$
of $Sp_{4mn+2n}$
with the Levi subgroup isomorphic to $GL_1^n \times {Sp}_{4mn}$,
then $Y_1 V_{[(2n)1^{4mn}],2}$ is a subgroup of $N_{1^n}$
consists of elements $v$ with $v_{n, j}=0$,
for $n+1 \leq j \leq 2mn+n$.
$\psi_{[(2n)1^{4mn}],1}(v) = \psi(\sum_{i=1}^{n-1} v_{i,i+1} + v_{n,4mn+n+1})$.

Identify $Sp_{4mn}$ with its embedding into $Sp_{4mn+2n}$
via the map $g \mapsto \diag(I_n, g, I_n)$.
Let  
$N_{1^n}$ be the unipotent radical of the parabolic ${P}_{1^n}$
of $Sp_{4mn}$
with the Levi subgroup isomorphic to $GL_1^n \times {Sp}_{4mn-2n}$,
then $Y_1 V_{[(2n)1^{4mn-2n}],2}$ is a subgroup of $N_{1^n}$
consists of elements $v$ with $v_{n, j}=0$,
for $n+1 \leq j \leq 2mn$. 
$\psi_{[(2n)1^{4mn-2n}],-1}(v) = \psi(\sum_{i=1}^{n-1} v_{i,i+1} - v_{n,4mn-n+1})$.

Let $\wt{\omega}$ be the Weyl element of $GL_{2n}$
defined in (4.31) of \cite{GRS99}:
\begin{align*}
\wt{\omega}_{2i,i}=1, & \text{ } i=1, \ldots, n,\\
\wt{\omega}_{2i-1,i+n}=1, & \text{ } i=1, \ldots, n,\\
\wt{\omega}_{i,j}=0, & \text{ } \text{otherwise.}
\end{align*}
Let 
\begin{equation}\label{thm7equ2}
\omega_1 = \begin{pmatrix}
\wt{\omega} &  & \\
& I_{4mn-2n} &\\
&& \wt{\omega}^*
\end{pmatrix} \in Sp_{4mn+2n}(F).
\end{equation}
As in \cite{GJS12}, we identify 
$Sp_{4mn+2n}(F)$ with the subgroup $Sp_{4mn+2n}(F) \times 1$
of $\wt{Sp}_{4mn+2n}(\BA)$.

Conjugating cross the integral in \ref{thm7equ1}
by $\omega_1$, it becomes:
\begin{align}\label{thm7equ3}
\int_{[W_1]} \wt{\varphi}(w \omega_1 g) \psi_{W_1}^{-1}(w) dw,
\end{align}
where $W_1 = \omega_1 Y_2 V_{[(2n)1^{4mn-2n}],2}
Y_1 V_{[(2n)1^{4mn}],2} \omega_1^{-1}$, 
if $\omega_1^{-1} w \omega_1 = v_1 v$, $v_1 \in Y_2 V_{[(2n)1^{4mn-2n}],2}$,
$v \in Y_1 V_{[(2n)1^{4mn}],2}$,
then 
$$\psi_{W_1}(w)=\psi_{[(2n)1^{4mn}],1}(v) 
\psi_{[(2n)1^{4mn-2n}],-1}(v_1).$$

Note that the metaplectic cover splits over $W_1(\BA)$,
and $W_1(\BA) \times 1$ is a subgroup of $\wt{Sp}_{4mn+2n}(\BA)$.
We identify $W_1$ with $W_1 \times 1$.

Elements in $W_1$ are of the following form
\begin{equation}\label{thm7equ29}
w=\begin{pmatrix}
Z & q_1 & q_2\\
& I_{4mn-2n} & q_1^*\\
&&Z^*
\end{pmatrix},
\end{equation}
where $q_1(i,j)=0$, for $i=2n-1, 2n$, $1 \leq j \leq 2mn-n$,
$Z \in GL_{2n}$ has the form (4.34) of \cite{GRS99}.
Write $Z$ as an $n \times n$ matrix of $2 \times 2$ block matrices
$Z = ([Z]_{i,j})$, $1 \leq i, j \leq n$,
then $[Z]_{n,1} = \cdots = [Z]_{n,n-1}=0$,
$[Z]_{n,n}=I_2$; $[Z]_{i,i}$ is lower unipotent, for $i < n$;
$[Z]_{i,j}$ is lower triangular, for $i < j$;
$[Z]_{i,j}$ is lower nilpotent, for $j < i < n$.
And 
\begin{equation}\label{thm7equ4}
\psi_{W_1}(w) = \psi(\sum_{i=1}^{n-1} \tr([Z]_{i,i+1})+(q_2(2n,1)-q_2(2n-1,2))).
\end{equation}

Let $R^1_i=\prod_{j=1}^{i} R^1_{i,j}$, for $1 \leq i \leq n-1$,
with $R^1_{i,j} = X_{\alpha_{i,j}}$, the root
subgroup corresponding to the root $\alpha_{i,j} = e_{2i} -e_{2(j-1)+1}$.
Let $R^1 = \prod_{i=1}^{n-1} R^1_i$.
Actually $R^1$ is the subgroup of $W_1$ consists of lower unipotent 
matrices. Write $W_1 = R^1 \wt{W}_1$, with $R^1 \cap \wt{W}_1 = \{1\}$.

Let $C^1_i=\prod_{j=1}^{i} C^1_{i,j}$, for $1 \leq i \leq n-1$,
with $C^1_{i,j} = X_{\beta_{i,j}}$, the root
subgroup corresponding to the root $\beta_{i,j} = e_{2(j-1)+1} -e_{2(i+1)}$.

We consider the quadruple
\begin{equation}\label{thm7equ5}
(\wt{W}_1 \prod_{i=1}^{n-2} R^1_i \prod_{j=2}^{n-1} R^1_{n-1,j},
\psi_{W_1}, R^1_{n-1,1}, C^1_{n-1,1}).
\end{equation}
It is easy to see that this quadruple satisfies all the conditions
for Lemma 2.3 of \cite{JL13b}. Hence, by Lemma 2.3 of \cite{JL13b},
the integral in \eqref{thm7equ3} is non-vanishing if and only
if the following integral is non-vanishing:
\begin{align}\label{thm7equ6}
\int_{[C^1_{n-1,1}\wt{W}_1 \prod_{i=1}^{n-2} R^1_i \prod_{j=2}^{n-1} R^1_{n-1,j}]} 
\wt{\varphi}(cwr \omega_1 g) \psi_{W_1}^{-1}(w) drdwdc.
\end{align}

We continue to consider the following sequence of quadruples:
\begin{align}\label{thm7equ7}
\begin{split}
& (C^1_{n-1,1}\wt{W}_1 \prod_{i=1}^{n-1} R^1_i \prod_{j=3}^{n-1} R^1_{n-1,j},
\psi_{W_1}, R^1_{n-1,2}, C^1_{n-1,2}),\\
& (\prod_{k=1}^2 C^1_{n-1,k}\wt{W}_1 \prod_{i=1}^{n-2} R^1_i \prod_{j=4}^{n-1} R^1_{n-1,j},
\psi_{W_1}, R^1_{n-1,3}, C^1_{n-1,3}),\\
& \cdots,\\
& (\prod_{k=1}^{n-2} C^1_{n-1,k}\wt{W}_1 \prod_{i=1}^{n-2} R^1_i ,
\psi_{W_1}, R^1_{n-1,n-1}, C^1_{n-1,n-1}).
\end{split}
\end{align}
Applying Lemma 2.3 of \cite{JL13b} $(n-1)$ times, we get that the integral
in \eqref{thm7equ6} is non-vanishing if and only if
the following integral is non-vanishing:
\begin{align}\label{thm7equ8}
\int_{[C^1_{n-1} \wt{W}_1 \prod_{i=1}^{n-2} R^1_i]} 
\wt{\varphi}(cwr \omega_1 g) \psi_{W_1}^{-1}(w) drdwdc.
\end{align}

Then, we repeat the above procedure for the pairs
$$(R^1_{n-2}, C^1_{n-2}), \ldots, (R^1_1, C^1_1).$$
For example, after repeating the above procedure for 
the pairs 
$$(R^1_{n-2}, C^1_{n-2}), \ldots, (R^1_{s+1}, C^1_{s+1}),$$
we need the following sequence of quadruples for the pair
$(R^1_s, C^1_s)$:
\begin{align}\label{thm7equ9}
\begin{split}
& (\prod_{i=s+1}^{n-1} C^1_i \wt{W}_1 \prod_{i=1}^{s-1} R^1_i \prod_{j=2}^{s} R^1_{s,j},
\psi_{W_1}, R^1_{s,1}, C^1_{s,1}),\\
& (\prod_{i=s+1}^{n-1} C^1_i C^1_{s,1} \wt{W}_1 \prod_{i=1}^{s-1} R^1_i \prod_{j=3}^{s} R^1_{s,j},
\psi_{W_1}, R^1_{s,2}, C^1_{s,2}),\\
& \cdots,\\
& (\prod_{i=s+1}^{n-1} C^1_i \prod_{k=1}^{s-1} C^1_{s,k} \wt{W}_1 \prod_{i=1}^{s-1} R^1_i,
\psi_{W_1}, R^1_{s,s}, C^1_{s,s}).
\end{split}
\end{align}
After applying Lemma 2.3 of \cite{JL13b} $s$ times, 
the integral in \eqref{thm7equ8} is non-vanishing if and only if 
the following integral is non-vanishing:
\begin{align}\label{thm7equ10}
\int_{[\prod_{i=s}^{n-1} C^1_i \wt{W}_1 \prod_{j=1}^{s-1} R^1_j]} 
\wt{\varphi}(cwr \omega_1 g) \psi_{W_1}^{-1}(w) drdwdc.
\end{align}

After repeating the above procedure for all the pairs
$$(R^1_{n-2}, C^1_{n-2}), \ldots, (R^1_1, C^1_1),$$
we will see that the integral in 
\eqref{thm7equ8} is non-vanishing if and only if 
the following integral is non-vanishing:
\begin{align}\label{thm7equ11}
\int_{[\prod_{i=1}^{n-1} C^1_i \wt{W}_1]} 
\wt{\varphi}(cw \omega_1 g) \psi_{W_1}^{-1}(w) dwdc.
\end{align}
Note that $\prod_{i=1}^{n-1} C^1_i \wt{W}_1 = \omega_1 V_{[(2n)^2 1^{4mn-2n}],2} \omega_1^{-1}$,
and $\psi_{\prod_{i=1}^{n-1} C^1_i\wt{W}_1}(cw)=\psi_{W_1}(w)=\psi_{[(2n)^2 1^{4mn-2n}], \{-1, 1\}}
(\omega_1^{-1} cw \omega_1)$.

Let 
$A = 
\begin{pmatrix}
1 & -1\\
1 & 1
\end{pmatrix}$, and 
$\epsilon=\diag(A, \ldots, A; I_{(4m-4)n}; A^*, \ldots, A^*)$,
as in (2.31) of \cite{GJS12}.
Conjugating cross the integral in \eqref{thm7equ11} by 
$\epsilon$, it becomes:
\begin{align}\label{thm7equ12}
\int_{[W_2]} 
\wt{\varphi}(w \epsilon \omega_1 g) \psi_{W_2}^{-1}(w) dw,
\end{align}
where $W_2 = \epsilon \prod_{i=1}^{n-1} C^1_i\wt{W}_1 \epsilon^{-1}$,
$\psi_{W_2}(w)=\psi_{\prod_{i=1}^{n-1} C^1_i\wt{W}_1}(\epsilon^{-1} w \epsilon)$.

Elements in $W_2$ are of the following form
\begin{equation}\label{thm7equ35}
w=\begin{pmatrix}
Z & q_1 & q_2\\
& I_{4mn-2n} & q_1^*\\
&&Z^*
\end{pmatrix},
\end{equation}
where $q_1(i,j)=0$, for $i=2n-1, 2n$, $1 \leq j \leq 2mn-n$,
$Z$ is in the unipotent radical of the parabolic subgroup of
$GL_{2n}$ with the Levi subgroup isomorphic to $GL_2^n$.
And 
\begin{equation}\label{thm7equ13}
\psi_{W_2}(w) = \psi(\sum_{i=1}^{2n-2} Z_{i,i+2}+q_2(2n-1,1)).
\end{equation}

Let $\nu=\begin{pmatrix}
\nu_1 & \nu_2\\
\nu_3 & \nu_4
\end{pmatrix}$ be the Weyl element in \eqref{main1equ6}.
Let 
\begin{equation}\label{thm7equ37}
\omega_2 = 
\begin{pmatrix}
\nu_1 & & \nu_2\\
& I_{4mn-2n} & \\
\nu_3 & & \nu_4
\end{pmatrix},
\end{equation}
a Weyl element of $\Sp_{4mn+2n}$.

Conjugating cross the integral in \eqref{thm7equ12}
by $\omega_2$, it becomes:
\begin{equation}\label{thm7equ14}
\int_{[W_3]} 
\wt{\varphi}(w \omega_2 \epsilon \omega_1 g) \psi_{W_3}^{-1}(w) dw,
\end{equation}
where $W_3 = \omega_2 W_2 \omega_2^{-1}$,
$\psi_{W_3}(w)=\psi_{W_2}(\omega_2^{-1} w \omega_2)$.

Elements in $W_3$
have the following form:
\begin{equation}\label{thm7equ15}
w = 
\begin{pmatrix}
Z & q_1 & q_2\\
0 & I_{4mn-2n} & q_1^*\\
0 & 0 & Z^*
\end{pmatrix}
\begin{pmatrix}
I_{2n} & 0 & 0\\
p_1 & I_{4mn-2n} & 0\\
p_2 & p_1^* & I_{2n}
\end{pmatrix},
\end{equation}
where $Z$ is in the standard maximal unipotent subgroup of $GL_{2n}$;
$q_1(i,j)=0$, for $n+1 \leq i \leq 2n$, and for $i=n$, $1 \leq j \leq 2mn-n$;
$q_2(i,j)=p_2(i,j)=0$, for $1 \leq j \leq i \leq 2n$;
$p_1(i,j)=0$, for $1 \leq j \leq n$, and for $j=n+1$, $n+1 \leq i \leq 2n$.

Let $\wt{C}$ be the unipotent subgroup consisting of 
elements of the following form:
\begin{equation*}
\begin{pmatrix}
I_{2n} & 0 & 0\\
p_1 & I_{4mn-2n} & 0\\
0 & p_1^* & I_{2n}
\end{pmatrix},
\end{equation*}
where $p_1(i,j)=0$ for any $i,j$, except
$j=n+1$, $1 \leq i \leq 2mn-n$.
Let $\wt{R}$ be the unipotent subgroup consisting of 
elements of the following form:
\begin{equation*}
\begin{pmatrix}
I_{2n} & q_1 & 0\\
0 & I_{4mn-2n} & q_1^*\\
0 & 0 & I_{2n}
\end{pmatrix},
\end{equation*}
where $q_1(i,j)=0$ for any $i,j$, except
$i=n$, $1 \leq j \leq 2mn-n$.

Write $W_3 = \wt{C} \wt{W}_3$, with $\wt{C} \cap \wt{W_3}=\{1\}$.
Consider the quadruple $(\wt{W}_3, \psi_{W_3}, \wt{R}, \wt{C})$.
It is easy to see that this quadruple satisfies all the conditions
of Lemma 2.3 of \cite{JL13b}. Hence, by Lemma 2.3 of \cite{JL13b},
the integral \eqref{thm7equ14} is non-vanishing if and only
if the following integral is non-vanishing:
\begin{equation}\label{thm7equ16}
\int_{[W_4]} 
\wt{\varphi}(w \omega_2 \epsilon \omega_1 g) \psi_{W_4}^{-1}(w) dw,
\end{equation}
where $W_4=\wt{R} \wt{W}_3$, for $w= rw'$, $r \in \wt{R}$,
$w \in \wt{W}_3$, $\psi_{W_4}(rw') = \psi_{W_3}(w')$.

Let $R^2_i = \prod_{j=1}^i X_{\alpha^i_j}$, for $1 \leq i \leq n$,
with $\alpha^i_j = e_i+e_{2n-i+j}$, 
and $R^2_i = \prod_{j=1}^{2n-i} X_{\alpha^i_j}$, for $n+1 \leq i \leq 2n-1$,
with $\alpha^i_j = e_i+e_{i+j}$. 
Let $C^2_i = \prod_{j=1}^i X_{\beta^i_j}$, for $1 \leq i \leq n$,
with $\beta^i_j = -e_{2n-i+j}-e_{i+1}$, 
and $C^2_i = \prod_{j=1}^{2n-i} X_{\beta^i_j}$, for $n+1 \leq i \leq 2n-1$,
with $\beta^i_j = -e_{i+j}-e_{i+1}$.

For $n+1 \leq i \leq 2n-1$, let $R^3_i$ be the unipotent subgroup consisting of 
elements of the following form:
\begin{equation*}
\begin{pmatrix}
I_{2n} & q_1 & 0\\
0 & I_{4mn-2n} & q_1^*\\
0 & 0 & I_{2n}
\end{pmatrix},
\end{equation*}
where $q_1(k,j)=0$ for any $k,j$, except
$k=i$, $1 \leq j \leq 4mn-2n$.
For $n+1 \leq i \leq 2n-1$, let $C^3_i$ 
be the unipotent subgroup consisting of 
elements of the following form:
\begin{equation*}
\begin{pmatrix}
I_{2n} & 0 & 0\\
p_1 & I_{4mn-2n} & 0\\
0 & p_1^* & I_{2n}
\end{pmatrix},
\end{equation*}
where $p_1(k,j)=0$ for any $k,j$, except
$j=i+1$, $1 \leq k \leq 4mn-2n$.
Write $W_4 = \prod_{i=1}^{2n-1} C^2_i \prod_{j=n+1}^{2n-1} C^3_j \wt{W}_4$,
with $\prod_{i=1}^{2n-1} C^2_i \prod_{j=n+1}^{2n-1} C^3_j \cap \wt{W}_4 = \{1\}$.

Then, we apply Lemma 2.3 of \cite{JL13b} to the pairs
$$(R^2_1, C^2_1), (R^2_2, C^2_2), \ldots, (R^2_n, C^2_n).$$
For example, for $(R^2_s, C^2_s)$, $1 \leq s \leq n$,
we need to consider the following sequence of 
quadruples:
\begin{align*}
&(\prod_{i=1}^{s-1} R^2_i \prod_{j=n+1}^{2n-1} C^3_j \wt{W}_4 \prod_{t=s+1}^{2n-1}
C^2_t \prod_{l=2}^s X_{\beta^s_l}, \psi_{W_4}, X_{\beta^s_1}, X_{\alpha^s_1}),\\
& (\prod_{i=1}^{s-1} R^2_i X_{\alpha^s_1} \prod_{j=n+1}^{2n-1} C^3_j \wt{W}_4 \prod_{t=s+1}^{2n-1}
C^2_t \prod_{l=3}^s X_{\beta^s_l}, \psi_{W_4}, X_{\beta^s_2}, X_{\alpha^s_2}),\\
& \cdots,\\
& (\prod_{i=1}^{s-1} R^2_i \prod_{k=1}^{s-1} X_{\alpha^s_k} \prod_{j=n+1}^{2n-1} C^3_j \wt{W}_4 \prod_{t=s+1}^{2n-1}
C^2_t, \psi_{W_4}, X_{\beta^s_s}, X_{\alpha^s_s}).\\
\end{align*}

After applying Lemma 2.3 of \cite{JL13b} to all the pairs
$$(R^2_1, C^2_1), (R^2_2, C^2_2), \ldots, (R^2_n, C^2_n),$$
we get that the integral \eqref{thm7equ16} is non-vanishing if and only
if the following integral is non-vanishing:
\begin{equation}\label{thm7equ17}
\int_{[\prod_{i=1}^{n} R^2_i \wt{W}_4 \prod_{t=n+1}^{2n-1}
C^2_t C^3_t]} 
\wt{\varphi}(rwc \omega_2 \epsilon \omega_1 g) \psi_{W_4}^{-1}(w) dcdwdr,
\end{equation}

Next, we apply Lemma 2.3 of \cite{JL13b} to the pairs
$$
(R^2_{n+1}, C^2_{n+1}), (R^3_{n+1}, C^3_{n+1}); \cdots;
(R^2_{2n-1}, C^2_{2n-1}), (R^3_{2n-1}, C^3_{2n-1}).
$$
Note that before applying Lemma 2.3 of \cite{JL13b} to
each pair $(R^2_{s}, C^2_{s})$, $n+1 \leq s \leq 2n-1$, 
we need to take the Fourier expansion along the one-dimensional
root subgroup $X_{e_s+e_s}$, as in the proof of Theorem \ref{main1part2}.

For example, for $s=n+1$, we first take the Fourier expansion of the integral
in \eqref{thm7equ17} along the one-dimensional root subgroup
$X_{e_s+e_s}$. Under the action of $GL_1$, we get two kinds of Fourier coefficients corresponding 
to the two orbits of the dual of $[X_{e_s+e_s}]$: the trivial one and the non-trivial one.
For any Fourier coefficient attached to the non-trivial orbit,
we can see that there is an inner integral $\varphi_{[(2n+2)1^{4mn-2}],\{a\}}$, which is identically zero 
by (1) in the proof of Theorem \ref{thm10}.
Therefore only the Fourier coefficient attached to the trivial orbit,
which actually equals to the integral in \eqref{thm7equ17},
survives. 

After applying Lemma 2.3 of \cite{JL13b} to all the pairs
$$
(R^2_{n+1}, C^2_{n+1}), (R^3_{n+1}, C^3_{n+1}); \cdots;
(R^2_{2n-1}, C^2_{2n-1}), (R^3_{2n-1}, C^3_{2n-1}),
$$
the integral \eqref{thm7equ17} is non-vanishing if and only
if the following integral is non-vanishing:
\begin{equation}\label{thm7equ18}
\int_{[\prod_{i=1}^{2n-1} R^2_i \prod_{t=n+1}^{2n-1} R^3_t X_{e_t+e_t} \wt{W}_4]} 
\wt{\varphi}(rxw \omega_2 \epsilon \omega_1 g) \psi_{W_4}^{-1}(w) dwdxdr,
\end{equation}

Note that elements in 
$\prod_{i=1}^{2n-1} R^2_i \prod_{t=n+1}^{2n-1} R^3_t X_{e_t+e_t} \wt{W}_4$
have the following form:
\begin{equation}\label{thm7equ19}
w = 
\begin{pmatrix}
Z & q_1 & q_2\\
0 & I_{4mn-2n} & q_1^*\\
0 & 0 & Z^*
\end{pmatrix},
\end{equation}
where $Z$ is in the standard maximal unipotent subgroup of $GL_{2n}$;
the last row of $q_1$ is zero. 
And $\psi_{W_4}(\begin{pmatrix}
Z & 0 & 0\\
0 & I_{4mn-2n} & 0\\
0 & 0 & Z^*
\end{pmatrix}) = \psi(\sum_{i=1}^{n-1}Z_{i,i+1} - \sum_{j=n}^{2n-1} Z_{i,i+1})$.

Clearly the integral \eqref{thm7equ18} is non-vanishing if and only
if the following integral is non-vanishing:
\begin{equation}\label{thm7equ20}
\int_{[\prod_{i=1}^{2n-1} R^2_i \prod_{t=n+1}^{2n-1} R^3_t X_{e_t+e_t} \wt{W}_4]} 
\wt{\varphi}(rxw \omega_2 \epsilon \omega_1 g) \psi_{W_4}^{'-1}(w) dwdxdr,
\end{equation}
where $\psi_{W_4}'(\begin{pmatrix}
Z & 0 & 0\\
0 & I_{4mn-2n} & 0\\
0 & 0 & Z^*
\end{pmatrix}) = \psi(\sum_{i=1}^{2n-1}Z_{i,i+1})$.
Note that the integral \eqref{thm7equ20}
is exactly $\wt{\varphi}^{\psi_{N_{1^{2n-1}}}}$, using notation in Lemma \ref{lem2}.
On the other hand, we know that by Lemma \ref{lem2},
$\wt{\varphi}^{\psi_{N_{1^{2n-1}}}} = \wt{\varphi}^{\wt{\psi}_{N_{1^{2n}}}}$.
Note that Lemma \ref{lem2} also applies to metaplectic groups.

Therefore, the
integral in \eqref{thm7equ20}
becomes
\begin{equation}\label{thm7equ21}
\int_{[\wt{U}]} \wt{\varphi}(u \omega_2 \epsilon \omega_1 g) \psi_{\wt{U}}(u)^{-1} du,
\end{equation} 
where any element in $\wt{U}$
has the following form:
$$
u = u(Z, q_1, q_2) =
\begin{pmatrix}
Z & q_1 & q_2\\
0 & I_{4mn-2n} & q_1^*\\
0 & 0 & Z^*
\end{pmatrix},
$$
where $Z$ is in the standard maximal unipotent subgroup of 
$GL_{2n}$,
$q_1 \in M_{(2n)\times (4m-4)n}$, $q_2 \in M_{(2n)\times (2n)}$, such that $q_2^t v_{2n} - v_{2n} q_2 = 0$,
where $v_{2n}$ is a matrix only with ones on the second diagonal.
$\psi_{\wt{U}}(u) = \psi(\sum_{i=1}^{2n-1} u_{i,i+1})$.

Hence, the integral in \eqref{thm7equ21} can be written as
\begin{equation}\label{thm7equ22}
\int_{u(Z, 0, 0)} \wt{\varphi}(u \omega_2 \epsilon \omega_1 g)_{\wt{P}_{2n}} \psi_{\wt{U}}(u)^{-1} du,
\end{equation} 
where $\wt{\varphi}_{P_{2n}}$ is the constant term of 
$\wt{\varphi}$ along the pre-image of the parabolic subgroup 
$P_{2n}=M_{2n}U_{2n}$ of $\Sp_{4mn+2n}$ with 
Levi isomorphic to $GL_{2n} \times Sp_{(4m-2)n}$.

By the similar calculation as in the proof of Lemma \ref{constantterm},
or the calculation at the end of Theorem 2.1 of \cite{GJS12},
there is an automorphic function
$$
f \in A(N_{2n}(\BA)\wt{M}_{2n}(F)\bs \wt{Sp}_{4mn+2n}(\BA))_{\tau \lvert \cdot \rvert^{-m}
\otimes \wt{\sigma}_{4(m-1)n+2n}},
$$
such that 
$$
\wt{\varphi}(g)_{\wt{P}_{2n}} = f(g), \forall g \in \wt{Sp}_{4mn+2n}(\BA).
$$

Therefore, the integral \eqref{thm7equ22} is the Whittaker
Fourier coefficient of an element in $\tau$, hence not identically zero.
This completes the proof of Part (1).

\textbf{Proof of Part (2)}.
By definition of Fourier-Jacobi coefficients ((3.14) of \cite{GRS11}),
for $\phi_1 \in \CS(\BA^{2mn-n})$, $\phi_2 = \phi_{21} \otimes \phi_{22}$,
$\phi_{21} \in \CS(\BA^n)$, $\phi_{22} \in \CS(\BA^{2mn-n})$,
we need to compute the composition of two Fourier-Jacobi coefficients
$FJ^{\phi_1}_{\psi^{-1}_{n-1}}$ and $FJ^{\phi_2}_{\psi^{1}_{n-1}}$:
\begin{align}\label{thm7equ23}
\begin{split}
& FJ^{\phi_1}_{\psi^{-1}_{n-1}} \circ FJ^{\phi_2}_{\psi^{1}_{n-1}}(\wt{\xi})(g)\\
= & \int_{[V_{[(2n)1^{4mn-2n}],1}]} \int_{[V_{[(2n)1^{4mn}],1}]} 
\wt{\varphi}(uvg) \theta^{2mn, \phi_2}_{\psi^{-1}}(l_2(u)vg) \psi_{[(2n)1^{4mn}], 1}^{-1}(u)du\\
& \theta^{2mn-n, \phi_1}_{\psi^{1}}(l_1(v)g) \psi_{[(2n)1^{4mn-2n}], -1}^{-1}(v)dv,
\end{split}
\end{align}
where $\wt{\varphi} \in \wt{\CE}_{\tau, \wt{\sigma}_{4(m-1)n+2n}}$,
$g \in  \wt{Sp}_{4mn-2n}(\BA)$, 
the theta series are defined in Section 1.2 \cite{GRS11}.
$V_{[(2n)1^{4mn-2n}],1}$ and 
$V_{[(2n)1^{4mn}],1}$,  $\psi_{[(2n)1^{4mn}], 1}$
and $\psi_{[(2n)1^{4mn-2n}], -1}$
are defined in Section 2 of \cite{JL13b}. $V_{[(2n)1^{4mn-2n}],1}$ and 
$V_{[(2n)1^{4mn}],1}$ are as $N_n$ in (3.14) of \cite{GRS11}. 
Explicitly,
$V_{[(2n)1^{4mn}],1}$ is the unipotent radical of the parabolic
subgroup $P_{1^n}^{4mn+2n}$ of $Sp_{4mn+2n}$ with Levi subgroup isomorphic to
$GL_1^n \times Sp_{4mn}$. $V_{[(2n)1^{4mn-2n}],1}$ is the 
unipotent radical of the parabolic
subgroup $P_{1^n}^{4mn}$ of $Sp_{4mn}$,
with Levi subgroup isomorphic to
$GL_1^n \times Sp_{4mn-2n}$. Note that $Sp_{4mn}$ is embedded
into $Sp_{4mn+2n}$ via the map $g \mapsto \diag(I_n, g, I_n)$,
and we identify it with the image. 
Then for $u \in V_{[(2n)1^{4mn}],1}$,
$\psi_{[(2n)1^{4mn}], 1}(u) = \psi(\sum_{i=1}^{n-1} u_{i,i+1})$,
$l_2(u) = \prod_{i=1}^{2mn} X_{\alpha_i}(u_{n,n+i})$,
with $\alpha_i = e_n-e_{n+i}$.
And for $v \in V_{[(2n)1^{4mn-2n}],1}$, 
$\psi_{[(2n)1^{4mn-2n}], -1}(v) = \psi(\sum_{i=1}^{n-1} v_{n+i, n+i+1})$,
$l_1(v) = \prod_{i=1}^{2mn-n} X_{\beta_i}(v_{2n, 2n+i})$,
with $\beta_i = e_{2n}-e_{2n+i}$.

First, we want to unfold the theta series  
$\theta^{2mn, \phi_2}_{\psi^{-1}}(l_2(u)vg)$. 
Write $l_2(u)$ as $l_2(u)=(q_1, q_2, q_3; z)$,
where $q_1, q_3 \in \BA^n$, $q_2 \in \BA^{4mn-2n}$, 
$z \in \BA$. Then
\begin{align}\label{thm7equ24}
\begin{split}
& \theta^{2mn, \phi_2}_{\psi^{-1}}(l_2(u)vg)\\
= & \sum_{\xi \in F^{2mn}} \omega_{\psi^{-1}}^{2mn} (l_2(u)vg)\phi_2(\xi)\\
= & \sum_{\xi_1 \in F^n, \xi_2 \in F^{2mn-n}} \omega_{\psi^{-1}}^{2mn} 
((q_1, q_2, q_3; z)vg)\phi_2(\xi_1, \xi_2)\\
= & \sum_{\xi_1 \in F^n, \xi_2 \in F^{2mn-n}} \omega_{\psi^{-1}}^{2mn} 
((\xi_1, 0, 0; 0)(q_1, q_2, q_3; z)vg)\phi_2(0, \xi_2)\\
= & \sum_{\xi_1 \in F^n, \xi_2 \in F^{2mn-n}} \omega_{\psi^{-1}}^{2mn} 
((q_1 + \xi_1, q_2, q_3; z + \xi_1 \nu_n q_3')vg)\phi_2(0, \xi_2)\\
= & \sum_{\xi_1 \in F^n, \xi_2 \in F^{2mn-n}} \omega_{\psi^{-1}}^{2mn} 
((0, q_2, q_3; z + \wt{\xi}_1 )(q_1 + \xi_1, 0, 0; 0)vg)\phi_2(0, \xi_2)
\end{split}
\end{align}
\begin{align*}
= & \sum_{\xi_1 \in F^n, \xi_2 \in F^{2mn-n}} \omega_{\psi^{-1}}^{2mn} 
((0, \wt{q}_2, q_3; z + \wt{\xi}_1 )vg(q_1 + \xi_1, 0, 0; 0))\phi_2(0, \xi_2)\\
\end{align*}
where $\nu_n$ is the matrix only with 1's on the second diagonal,
$\wt{\xi}_1 = 2 \xi_1 \nu_n q_3'+ q_3 \nu_n q_1'$, $\wt{q}_2 = q_2 + (q_1+\xi_1)(n)v$,
$(q_1+\xi_1)(n)$ is the $n$-th coordinate of the vector $q_1+\xi_1$.
Note that $(q_1 + \xi_1, 0, 0; 0))$ commutes with $g$.

Plugging \eqref{thm7equ24} into the integral in \eqref{thm7equ23}, 
collapsing the summation over $\xi_1$ with the integration over $q_1$,
changing variables for $q_2$ and $z$, we will have

\begin{align}\label{thm7equ25}
\begin{split}
& FJ^{\phi_1}_{\psi^{-1}_{n-1}} \circ FJ^{\phi_2}_{\psi^{1}_{n-1}}(\wt{\xi})(g)\\
= & \int_{[V_{[(2n)1^{4mn-2n}],1}]} \int_{[V_{[(2n)1^{4mn}],1}']} 
\int_{\BA^n}
\wt{\varphi}(uvg\hat{q}_1) \\
&\sum_{\xi_2 \in F^{2mn-n}} \omega_{\psi^{-1}}^{2mn} 
(l_2(u)vg(q_1, 0, 0; 0))\phi_2(0, \xi_2)
\psi_{[(2n)1^{4mn}], 1}^{-1}(u)du\\
& \theta^{2mn-n, \phi_1}_{\psi^{1}}(l_1(v)g) \psi_{[(2n)1^{4mn-2n}], -1}^{-1}(v)dq_1dv,
\end{split}
\end{align}
where $\hat{q}_1 = \prod_{i=1}^{n} X_{\alpha_i}(q_1(n,n+i))$,
with $\alpha_i = e_n-e_{n+i}$; $V_{[(2n)1^{4mn}],1}'$ is a subgroup
of $V_{[(2n)1^{4mn}],1}$ consisting of elements $u$ with
$u_{n,n+i}=0$, for $1 \leq i \leq n$; $l_2(u) = (0, q_2, q_3; z)$. Note that, 
$V_{[(2n)1^{4mn}],1}'$ is actually $Y V_{[(2n)1^{4mn}],2}$, where $Y$ is defined 
in (2.5) of \cite{JL13b} corresponding to the partition
$[(2n)1^{4mn}]$, and $V_{[(2n)1^{4mn}],2}$ is defined Section 2 of \cite{JL13b}.
For short, we will write $q_1$ for $(q_1, 0, 0; 0)$.

By Formula (1.4) \cite{GRS11}, 
\begin{align}\label{thm7equ26}
\begin{split}
& \sum_{\xi_2 \in F^{2mn-n}} \omega_{\psi^{-1}}^{2mn} ((0, q_2, q_3; z)vg(q_1, 0, 0; 0))\phi_2(0, \xi_2)\\
= & \sum_{\xi_2 \in F^{2mn-n}} \omega_{\psi^{-1}}^{2mn} ((0, q_2, 0; z)vg(q_1, 0, 0; 0))\phi_2(0, \xi_2)\\
= & \sum_{\xi_2 \in F^{2mn-n}} \omega_{\psi^{-1}}^{2mn} ((0, q_2, 0; z)vg(q_1, 0, 0; 0))\phi_2(q_1, \xi_2)\\
= & \theta^{2mn-n, \phi_{22}}_{\psi^{-1}}((q_2, z)vg) \phi_{21}(q_1),
\end{split}
\end{align}
since we assumed that $\phi_2 = \phi_{21} \otimes \phi_{22}$.
Let $l_3(u) = (q_2, z)$, for $u \in V_{[(2n)1^{4mn}],1}'$.
Then, the integral in \eqref{thm7equ25} becomes
\begin{align}\label{thm7equ27}
\begin{split}
& \int_{[V_{[(2n)1^{4mn-2n}],1}]} \int_{[V_{[(2n)1^{4mn}],1}']} 
\int_{\BA^n}
\wt{\varphi}(uvg\hat{q}_1) \phi_{21}(q_1) dq_1\\
&\theta^{2mn-n, \phi_{22}}_{\psi^{-1}}(l_3(u)vg) 
\psi_{[(2n)1^{4mn}], 1}^{-1}(u)du \theta^{2mn-n, \phi_1}_{\psi^{1}}(l_1(v)g) \psi_{[(2n)1^{4mn-2n}], -1}^{-1}(v)dv\\
= & \int_{[V_{[(2n)1^{4mn-2n}],1}]} \int_{[V_{[(2n)1^{4mn}],1}']} 
\wt{\varphi}'(uvg) \theta^{2mn-n, \phi_{22}}_{\psi^{-1}}(l_3(u)vg) 
\psi_{[(2n)1^{4mn}], 1}^{-1}(u)du \\
& \theta^{2mn-n, \phi_1}_{\psi^{1}}(l_1(v)g) \psi_{[(2n)1^{4mn-2n}], -1}^{-1}(v)dv,
\end{split}
\end{align}
where $\wt{\varphi}'(uvg) = \int_{\BA^n}
\wt{\varphi}(uvg\hat{q}_1) \phi_{21}(q_1) dq_1$. Note that we still
have that
$\wt{\varphi}' \in \wt{\CE}_{\tau, \wt{\sigma}_{4(m-1)n+2n}}$.

Let $\omega_1$ be the Weyl element of $Sp_{4mn+2n}$ in \eqref{thm7equ2}.
Conjugating cross by $\omega_1$, the integral in \eqref{thm7equ27} becomes
\begin{align}\label{thm7equ28}
\begin{split}
& \int_{[W_1']} 
\wt{\varphi}'(w \omega_1 g) \theta^{2mn-n, \phi_{22}}_{\psi^{-1}}(l_4(w)g) 
\theta^{2mn-n, \phi_1}_{\psi^{1}}(l_5(w)g) \psi_{W_1'}^{-1}(w)dw,
\end{split}
\end{align}
where $W_1'=\omega_1 V_{[(2n)1^{4mn-2n}],1} V_{[(2n)1^{4mn}],1}' \omega_1^{-1}$,
its elements have the form as in \eqref{thm7equ29}, except that 
there is no requirement that $q_1(i,j)=0$, for $i=2n-1, 2n$, $1 \leq j \leq 2mn-n$.
For $w=\begin{pmatrix}
Z & q_1 & q_2\\
& I_{4mn-2n} & q_1^*\\
&&Z^*
\end{pmatrix}$ as in \eqref{thm7equ29}, $l_4(w) = (q_1(2n-1), q_2(2n-1,2))$, 
$l_5(w)=(q_1(2n), q_2(2n,1))$, $q_1(2n-1)$, $q_1(2n)$ are the $(2n-1)$-th,
$(2n)$-th rows of $q_1$, respectively.
And $\psi_{W_1'}(w)= \psi(\sum_{i=1}^{n-1} \tr([Z]_{i,i+1})$,
with notation as in \eqref{thm7equ4}.

Next, we repeat the steps from \eqref{thm7equ5} to \eqref{thm7equ11},
and use Lemma 2.5 of \cite{JL13b} whenever Lemma 2.3 of \cite{JL13b} is used.
We will get that the integral in \eqref{thm7equ28}
becomes 
\begin{align}\label{thm7equ30}
\begin{split}
& \int_{[W_1'']} 
\wt{\varphi}''(w \omega_1 g) \theta^{2mn-n, \phi_{22}}_{\psi^{-1}}(l_4(w)g) 
\theta^{2mn-n, \phi_1}_{\psi^{1}}(l_5(w)g) \psi_{W_1''}^{-1}(w)dw,
\end{split}
\end{align}
where $W_1''$ is unipotent radical of the parabolic subgroup 
$P_{2^n}^{4mn+2n}$ of $Sp_{4mn+2n}$ with Levi subgroup
isomorphic to $GL_2^n \times Sp_{4mn-2n}$. 
$\wt{\varphi}'' \in \wt{\CE}_{\tau, \wt{\sigma}_{4(m-1)n+2n}}$.
And for $w \in W_1''$, 
$\psi_{W_1''}(w)= \psi(\sum_{i=1}^{2n-2} w_{i,i+2})$.

Next, we want to unfold the two theta series as before.
To do this, we need to use certain property
of theta series as in (5.9) \cite{GJS12}:
\begin{align}\label{thm7equ31}
\begin{split}
& \theta^{2mn-n, \phi_{22}}_{\psi^{-1}}((x_1, y_1; z_1)g) 
\theta^{2mn-n, \phi_1}_{\psi^{1}}((x_2, y_2; z_1)g)\\
= & \theta^{4mn-2n, \phi_{22} \otimes \phi_1}_{\psi}
((x_1, x_2, y_2, -y_1; z_2-z_1) \wt{g}),
\end{split}
\end{align}
where for $w \in W_1'$, we write $l_4(w) = (x_1, y_1; z_1)$,
$l_5(w) = (x_2, y_2; z_2)$, $x_1, y_1, x_2, y_2 \in \BA^{2mn-n}$,
and $$\wt{g} = \begin{pmatrix}
A & & & -B\\
  &A&B&\\
  &C&D&\\
-C& & &D
\end{pmatrix},$$
if we write $g=\left( \begin{pmatrix}
A&B\\
C&D
\end{pmatrix}, \varepsilon \right)$. Let $\phi_3 = \phi_{22} \otimes \phi_1$.

Let 
$$
\gamma = 
\begin{pmatrix}
I_{2mn-n} & 0 & \frac{1}{2} I_{2mn-n} & 0\\
I_{2mn-n} & 0 & -\frac{1}{2} I_{2mn-n} & 0\\
0 & I_{2mn-n} & 0 & \frac{1}{2} I_{2mn-n}\\
0 & - I_{2mn-n} & 0 &\frac{1}{2} I_{2mn-n}
\end{pmatrix} \in Sp_{8mn-4n}.
$$
Then,
\begin{align}\label{thm7equ32}
\begin{split}
(x_1, x_2, y_2, -y_1) \gamma & = (x_1+x_2, y_1+y_2, \frac{1}{2}(x_1-x_2), 
-\frac{1}{2}(y_1-y_2)),\\
\hat{g} & := \gamma^{-1} \wt{g} \gamma = \left( \begin{pmatrix}
g &\\
& g^*
\end{pmatrix}, 1 \right).
\end{split}
\end{align}

Therefore, by \eqref{thm7equ32}, the right hand side of 
\eqref{thm7equ31} becomes:
\begin{align}\label{thm7equ33}
\begin{split}
& \theta^{4mn-2n, \phi_3}_{\psi}
((x_1, x_2, y_2, -y_1; z_2-z_1) \wt{g})\\
= & \theta^{4mn-2n, \phi_3'}_{\psi}
((x_1+x_2, y_1+y_2, \frac{1}{2}(x_1-x_2), 
-\frac{1}{2}(y_1-y_2); z_2-z_1) \hat{g}),
\end{split}
\end{align}
where $\phi_3' = \omega^{4mn-2n}_{\psi} (\gamma^{-1}) \phi_3$.

Let 
$A = 
\begin{pmatrix}
1 & -1\\
1 & 1
\end{pmatrix}$, and 
$\epsilon=\diag(A, \ldots, A; I_{(4m-4)n}; A^*, \ldots, A^*)$,
as in (2.31) of \cite{GJS12}.
Conjugating cross the integral in \eqref{thm7equ30} by 
$\epsilon$, it becomes:
\begin{align}\label{thm7equ34}
\begin{split}
& \int_{[W_2']} 
\wt{\varphi}''(w \epsilon \omega_1 g) \theta^{4mn-2n, \phi_3'}_{\psi}
(l_6(w)\hat{g}) \psi_{W_2'}^{-1}(w)dw\\
= & \int_{[W_2']} 
\wt{\varphi}''(w g \epsilon \omega_1) \theta^{4mn-2n, \phi_3'}_{\psi}
(l_6(w)\hat{g}) \psi_{W_2'}^{-1}(w)dw\\
= & \int_{[W_2']} 
\wt{\varphi}'''(w g) \theta^{4mn-2n, \phi_3'}_{\psi}
(l_6(w)\hat{g}) \psi_{W_2'}^{-1}(w)dw,
\end{split}
\end{align}
where $W_2' = \epsilon W_1'' \epsilon^{-1}$,
and elements in $W_2'$ are as in \eqref{thm7equ35},
except that there is no requirement that 
$q_1(i,j)=0$, for $i=2n-1, 2n$, $1 \leq j \leq 2mn-n$,
still $\wt{\varphi}''' \in \wt{\CE}_{\tau, \wt{\sigma}_{4(m-1)n+2n}}$.
Note that $g$ commutes with $\epsilon \omega_1$.
And $\psi_{W_2'}(w) = \psi(\sum_{i=1}^{2n-2} w_{i,i+2}+w_{2n-1,4mn+1})$.
For $w=\begin{pmatrix}
Z & q_1 & q_2\\
& I_{4mn-2n} & q_1^*\\
&&Z^*
\end{pmatrix}$ as in \eqref{thm7equ35}, for $i=2n-1, 2n$,
write the $i$-th row of $q_1$ as $q_1(i)=(x_i, y_i)$,
then 
$$l_6(w) = (x_{2n}, y_{2n}, \frac{1}{2}x_{2n-1}, -\frac{1}{2}y_{2n-1}, -q_2(2n-1,1)).$$

Then, we unfold the theta series 
$\theta^{4mn-2n, \phi_3'}_{\psi}
(l_6(w)\hat{g})$ as in \eqref{thm7equ24},
the integral in \eqref{thm7equ34} becomes:
\begin{align}\label{thm7equ36}
\begin{split}
& \int_{[W_3']} \int_{\BA^{4mn-2n}}
\wt{\varphi}'''(w g \widehat{\xi}) \phi_3(\xi) d\xi\psi_{W_2'}^{-1}(w)dw\\
= & \int_{[W_3']} 
\wt{\varphi}^{(4)}(w g)\psi_{W_3'}^{-1}(w)dw
\end{split}
\end{align}
where $\widehat{\xi} = \prod_{i=1}^{2mn-n} X_{\alpha_i} (\xi(i))
\prod_{j=1}^{2mn-n} X_{\beta_j} (\xi(2mn-n+j))$,
with $\alpha_i = e_{2n}-e_{2n+i}$, $\beta_j = e_{2n} + e_{2mn+n-j+1}$,
and $\wt{\varphi}^{(4)}(w g)=\int_{\BA^{4mn-2n}}
\wt{\varphi}'''(w g \widehat{\xi}) \phi_3(\xi)$,
still $\wt{\varphi}^{(4)} \in \wt{\CE}_{\tau, \wt{\sigma}_{4(m-1)n+2n}}$.
And $W_3'$ is the subgroup of $W_2'$ consisting of elements
$w=\begin{pmatrix}
Z & q_1 & q_2\\
& I_{4mn-2n} & q_1^*\\
&&Z^*
\end{pmatrix}$, with $q_1(2n)=0$.

Conjugate cross the integral in \eqref{thm7equ36} by 
the Weyl elements $\omega_2$ in \eqref{thm7equ37}, it becomes:
\begin{align}\label{thm7equ38}
\begin{split}
& \int_{[W_4]} 
\wt{\varphi}^{(4)}(w \omega_2 g)\psi_{W_4}^{-1}(w)dw\\
= & \int_{[W_4]} 
\wt{\varphi}^{(4)}(w g \omega_2)\psi_{W_4}^{-1}(w)dw\\
= & \int_{[W_4]} 
\wt{\varphi}^{(5)}(w g)\psi_{W_4}^{-1}(w)dw,
\end{split}
\end{align}
where $W_4$, $\psi_{W_4}$ are exactly as in \eqref{thm7equ16}.

Now, we repeat the steps from \eqref{thm7equ16} to \eqref{thm7equ18},
and use Lemma 2.5 of \cite{JL13b} whenever Lemma 2.3 of \cite{JL13b} is used.
Then, we get that the integral in \eqref{thm7equ38} becomes:
\begin{align}\label{thm7equ39}
\begin{split}
\int_{[W_5]} 
\wt{\varphi}^{(5)}(w g)\psi_{W_5}^{-1}(w)dw,
\end{split}
\end{align}
where $W_5=\prod_{i=1}^{2n-1} R^2_i \prod_{t=n+1}^{2n-1} R^3_t X_{e_t+e_t} \wt{W}_4$ 
as in \eqref{thm7equ18}.
And given $w = \begin{pmatrix}
Z & q_1 & q_2\\
0 & I_{4mn-2n} & q_1^*\\
0 & 0 & Z^*
\end{pmatrix}$ as in \eqref{thm7equ19},
$\psi_{W_5}(w) = \psi(\sum_{i=1}^{n-1}Z_{i,i+1} - \sum_{j=n}^{2n-1} Z_{i,i+1})$.

Let $t = \diag(I_{n-1}, -I_{n+1}; I_{4mn-2n}; -I_{n+1}, I_{n-1}) \in Sp_{4mn+2n}(F)$.
Since $\wt{\varphi}^{(5)}$ is automorphic, the integral in 
\eqref{thm7equ39} becomes:
\begin{align}\label{thm7equ40}
\begin{split}
& \int_{[W_5]} 
\wt{\varphi}^{(5)}(tw g)\psi_{W_5}^{-1}(w)dw\\
= & \int_{[W_5]} 
\wt{\varphi}^{(5)}(w g)\psi_{W_5}^{',-1}(w)dw,
\end{split}
\end{align}
after changing variable $w \mapsto t^{-1} w$,
and $\psi_{W_5}^{'}(w)=\psi(\sum_{i=1}^{2n-1}Z_{i,i+1})$,
for $w = \begin{pmatrix}
Z & q_1 & q_2\\
0 & I_{4mn-2n} & q_1^*\\
0 & 0 & Z^*
\end{pmatrix} \in W_5$.

Note that the integral on the right hand side of the identity in \eqref{thm7equ40}
is exactly $\wt{\varphi}^{\psi_{N_{1^{2n-1}}}}$, using notation in Lemma \ref{lem2}.
On the other hand, we know that by Lemma \ref{lem2},
$\wt{\varphi}^{\psi_{N_{1^{2n-1}}}} = \wt{\varphi}^{\wt{\psi}_{N_{1^{2n}}}}$.
Note that Lemma \ref{lem2} also applies to metaplectic groups.

Therefore, the integral in \eqref{thm7equ40} becomes
\begin{equation}\label{thm7equ41}
\int_{[\wt{U}]} \wt{\varphi}^{(5)}(u g) \psi_{\wt{U}}(u)^{-1} du,
\end{equation} 
where $\wt{U}$ and $\psi_{\wt{U}}$ are exactly as in \eqref{thm7equ21}.

Now, it follows easily from the end of the proof of Part (1) that
as a function of $g \in \wt{Sp}_{4mn-2n}$, 
the integral in \eqref{thm7equ41} gives a section in 
$\wt{\sigma}_{4(m-1)n+2n}$.
Since starting from the integral in 
\eqref{thm7equ23}, we always get equalities,
$FJ^{\phi_1}_{\psi^{-1}_{n-1}} \circ FJ^{\phi_2}_{\psi^{1}_{n-1}}(\wt{\xi})
\in \wt{\sigma}_{4(m-1)n+2n}$.
Therefore, 
$$\CD_{2n, \psi^{-1}}^{4mn} \circ \CD_{2n, \psi^{1}}^{4mn+2n}
(\wt{\CE}_{\tau, \wt{\sigma}_{4(m-1)n+2n}}) \subset \wt{\sigma}_{4(m-1)n+2n}.$$

On the other hand, 
by Part (1), 
$$\CD_{2n, \psi^{-1}}^{4mn} \circ \CD_{2n, \psi^{1}}^{4mn+2n}
(\wt{\CE}_{\tau, \wt{\sigma}_{4(m-1)n+2n}}) \neq 0.$$
Since $\wt{\sigma}_{4(m-1)n+2n}$ is irreducible,
we have that 
$$\CD_{2n, \psi^{-1}}^{4mn} \circ \CD_{2n, \psi^{1}}^{4mn+2n}
(\wt{\CE}_{\tau, \wt{\sigma}_{4(m-1)n+2n}}) = \wt{\sigma}_{4(m-1)n+2n}.$$

This finishes the proof of Part (2),
and completes the proof of the theorem.
\end{proof}

\begin{rmk}\label{rmk2}
Note that in the proof of Part (2), we could easily get a similar identity
as in Theorem 5.1 \cite{GJS12}, but for simplicity, we did not 
write it down explicitly.
\end{rmk}

Theorem \ref{thm7} easily implies the following result.

\begin{thm}\label{thm10}
$\ul{p}=[(2n)^{2m}(2n_1)^{s_1}(2n_2)^{s_2} \cdots (2n_k)^{s_k}]$
is a maximal partition providing non-vanishing Fourier coefficients
for $\wt{\CE}_{\tau, \wt{\sigma}_{4(m-1)n+2n}}$.
\end{thm}

\begin{proof}
By Theorem \ref{thm7}, 
$$
[(2n)1^{4mn}] \circ [(2n)1^{4mn-2n}] \circ [(2n)^{2m-2}(2n_1)^{s_1}(2n_2)^{s_2} \cdots (2n_k)^{s_k}]
$$
is a composite partition providing non-vanishing Fourier coefficients
for $\wt{\CE}_{\tau, \wt{\sigma}_{4(m-1)n+2n}}$.
By Lemma 2.6 of \cite{GRS03} or Lemma 3.1 of \cite{JL13b},  
$[(2n)^{2m}(2n_1)^{s_1}(2n_2)^{s_2} \cdots (2n_k)^{s_k}]$
is a partition providing non-vanishing Fourier coefficients
for $\wt{\CE}_{\tau, \wt{\sigma}_{4(m-1)n+2n}}$.

Since by Theorem \ref{thm7}, $\CD_{2n, \psi^{-1}}^{4mn} \circ \CD_{2n, \psi^{1}}^{4mn+2n}
(\wt{\CE}_{\tau, \wt{\sigma}_{4(m-1)n+2n}})=\wt{\sigma}_{4(m-1)n+2n},$
and $\ul{p}(\wt{\sigma}_{4(m-1)n+2n})=[(2n)^{2(m-1)}(2n_1)^{s_1}(2n_2)^{s_2} \cdots (2n_k)^{s_k}],$
to show the maximality of $\ul{p}$, we just have to show that 

(1) at the step of taking $\CD_{2n, \psi^{1}}^{4mn+2n}$, 
$\wt{\CE}_{\tau, \wt{\sigma}_{4(m-1)n+2n}}$ has no nonzero Fourier coefficients 
attached to the symplectic partitions 
$[(2l)1^{4mn+2n-2l}]$ for any $l \geq n+1$, or $[(2l+1)^21^{4mn+2n-4l-2}]$, for any $l \geq n$;

(2) at the step of taking $\CD_{2n, \psi^{-1}}^{4mn}$, 
$\CD_{2n, \psi^{1}}^{4mn+2n}(\wt{\CE}_{\tau, \wt{\sigma}_{4(m-1)n+2n}})$ has no nonzero Fourier coefficients 
attached to the symplectic partitions 
$[(2l)1^{4mn-2l}]$ for any $l \geq n+1$, or $[(2l+1)^21^{4mn-4l-2}]$, for any $l \geq n$.

We will show (1) and (2) using calculations of unramified local components.
Let $v$ be a finite place such that $\wt{\CE}_{\tau, \wt{\sigma}_{4(m-1)n+2n},v}$
is unramified. This means that both $\tau_v$ and $\wt{\sigma}_{4(m-1)n+2n, v}$
are also unramified. Assume that 
$\tau_v = \times_{i=1}^n \nu^{\alpha_i} \chi_i \times \times_{i=1}^n \nu^{-\alpha_i} \chi_i^{-1}$, where $\nu^{\alpha_i}(\cdot) = \lvert det(\cdot) \rvert^{\alpha_i}$,
$0 \leq \alpha_i < \frac{1}{2}$, and $\chi_i$'s are unitary 
unramified characters of $F^*_v$, for $1 \leq i \leq n$.
Since $\wt{\pi}$ lifts weakly to $\tau$ with respect to $\psi$,
$\wt{\pi}_v = \mu_{\psi} \times_{i=1}^n \nu^{\alpha_i} \chi_i \rtimes 1_{\wt{Sp}_0}$.

Since by the definition of the set 
$\CN'_{\wt{Sp}_{4(m-1)n+2n}}(\tau, \psi)$, 
$\wt{\sigma}_{4(m-1)n+2n}$ is an irreducible cuspidal
automorphic representation of $\wt{Sp}_{4(m-1)n+2n}(\BA)$, which is nearly equivalent to
the residual representation $\wt{\CE}_{\Delta(\tau, m-1) \otimes \wt{\pi}}$,
similarly as Lemma 3.1 of \cite{GRS05},
it is easy to see that $\wt{\CE}_{\tau, \wt{\sigma}_{4(m-1)n+2n},v}$
is the unique unramified component of the following induced representation
\begin{equation}\label{thm10equ1}
\Ind^{\wt{Sp}_{4mn+2n}(F^*_v)}_{\wt{P}_{(2m+1)^n}(F^*_v)} \mu_{\psi}
\otimes_{i=1}^n \lvert \cdot \rvert^{\alpha_i} \chi_i(det_{GL_{2m+1}}) \otimes 1_{\wt{Sp}_0},
\end{equation}
where ${P}_{(2m+1)^n}$ is the parabolic subgroup of ${Sp}_{4mn+2n}$
with Levi isomorphic to $GL_{2m+1}^n \times 1_{Sp_0}$, and $\wt{P}_{(2m+1)^n}$
is its full pre-image in $\wt{Sp}_{4mn+2n}$.

By Lemma 3.2 of \cite{JL13c}, we can easily see that (1) holds.
By \eqref{prop1equ2} of Proposition \ref{prop1}, 
\begin{align}\label{thm10equ2}
\begin{split}
 & FJ_{\psi^1_{n-1}} (\Ind^{\wt{Sp}_{4mn+2n}(F^*_v)}_{\wt{P}_{(2m+1)^n}(F^*_v)} \mu_{\psi}
\otimes_{i=1}^n \lvert \cdot \rvert^{\alpha_i} \chi_i(det_{GL_{2m+1}}) \otimes 1_{\wt{Sp}_0})\\
\cong & \Ind^{{Sp}_{4mn}(F^*_v)}_{{P}_{(2m)^n}(F^*_v)}
\otimes_{i=1}^n \lvert \cdot \rvert^{\alpha_i} \chi_i(det_{GL_{2m}}) \otimes 1_{{Sp}_0},
\end{split}
\end{align}
which is actually irreducible, by results in \cite{Jan96},
and is an unramified local component of $\CD_{2n, \psi^{1}}^{4mn+2n}(\wt{\CE}_{\tau, \wt{\sigma}_{4(m-1)n+2n}})$.
Again, by Lemma 3.1 of \cite{JL13c}, we can easily see that (2) also holds.
Therefore, we have shown that $\ul{p}=[(2n)^{2m}(2n_1)^{s_1}(2n_2)^{s_2} \cdots (2n_k)^{s_k}]$
is a maximal partition providing non-vanishing Fourier coefficients
for $\wt{\CE}_{\tau, \wt{\sigma}_{4(m-1)n+2n}}$,
which completes the proof of this theorem.
\end{proof}

To continue, we prove that $\CD_{2n, \psi^{1}}^{4mn+2n}
(\wt{\CE}_{\tau, \wt{\sigma}_{4(m-1)n+2n}})$ is a cuspidal
representation, every component of which is 
inside
$\CN_{{Sp}_{4mn}}(\tau, \psi)$.

\begin{thm}\label{thm5}
$\CD_{2n, \psi^{1}}^{4mn+2n}
(\wt{\CE}_{\tau, \wt{\sigma}_{4(m-1)n+2n}}) \subset 
\CN_{{Sp}_{4mn}}(\tau, \psi)$.
\end{thm}

\begin{proof}
We will prove that 

(1) $\CD_{2n, \psi^{1}}^{4mn+2n}
(\wt{\CE}_{\tau, \wt{\sigma}_{4(m-1)n+2n}})$ is a cuspidal
representation;

(2) every irreducible component of $\CD_{2n, \psi^{1}}^{4mn+2n}
(\wt{\CE}_{\tau, \wt{\sigma}_{4(m-1)n+2n}})$ is 
inside
$\CN_{{Sp}_{4mn}}(\tau, \psi)$.

To prove (1), we will show that the constant terms of 
elements in $\CD_{2n, \psi^{1}}^{4mn+2n}
(\wt{\CE}_{\tau, \wt{\sigma}_{4(m-1)n+2n}})$ along all maximal parabolic subgroups
of ${Sp}_{4mn}$ are all zero.

Recall that $P^{4mn}_r = M^{4mn}_r N^{4mn}_r$ (with $1 \leq r \leq 2mn$) is the standard parabolic subgroup
of ${Sp}_{4mn}$ with Levi part $M^{4mn}_r$ isomorphic to $GL_r \times Sp_{4mn-2r}$,
$N^{4mn}_r$ is the unipotent radical. Take any $\wt{\xi} \in \wt{\CE}_{\tau, \wt{\sigma}_{4(m-1)n+2n}}$,
we will calculate the constant term of $FJ^{\phi}_{\psi^1_{n-1}}(\wt{\xi})$ along 
$P^{4mn}_r$, which is denoted by $\CC_{N^{4mn}_r}(FJ^{\phi}_{\psi^1_{n-1}}(\wt{\xi}))$.

By Theorem 7.8 of \cite{GRS11}, 
\begin{align}\label{thm5equ1}
\begin{split}
& \CC_{N^{4mn}_r}(FJ^{\phi}_{\psi^1_{n-1}}(\wt{\xi}))\\
= & \sum_{k=0}^r \sum_{\gamma \in P^1_{r-k, 1^k}(F) \bs GL_r(F)}
\int_{L(\BA)} \phi_1(i(\lambda)) FJ^{\phi_2}_{\psi^1_{n-1+k}} 
(\CC_{N^{4mn+2n}_{r-k}} (\wt{\xi}))(\hat{\gamma} \lambda \beta) d \lambda,
\end{split}
\end{align}
where $N^{4mn+2n}_{r-k}$ is the unipotent radical of the parabolic 
subgroup $P^{4mn+2n}_{r-k}$ of ${Sp}_{4mn+2n}$ with Levi isomorphic to 
$GL_{r-k} \times Sp_{4mn+2n-2r+2k}$, and it is identified with
it's image in $\wt{Sp}_{4mn+2n}$; $P^1_{r-k, 1^k}$ is a subgroup of $GL_r$
consisting of matrices of the form 
$\begin{pmatrix}
g & x\\
0 & z
\end{pmatrix}$, with $z \in U_k$, the standard maximal unipotent subgroup of $GL_k$;
for $a \in GL_j$, $j \leq 2mn+n$, $\hat{a}=\diag(a, I_{4mn+2n-2j}, a^*)$;
$L$ is a unipotent subgroup, consisting of matrices of the form
$\lambda = \begin{pmatrix}
I_r & 0\\
x & I_n
\end{pmatrix}^{\wedge}$, and $i(\lambda)$ is the last row of $x$;
$\beta=\begin{pmatrix}
0 & I_r\\
I_n & 0
\end{pmatrix}^{\wedge}$; $\phi = \phi_1 \otimes \phi_2$, with 
$\phi_1 \in \CS(\BA^r)$, $\phi_2 \in \CS(\BA^{2mn-r})$; 
$$FJ^{\phi_2}_{\psi^1_{n-1+k}} 
(\CC_{N^{4mn+2n}_{r-k}} (\wt{\xi}))(\hat{\gamma} \lambda \beta):=
FJ^{\phi_2}_{\psi^1_{n-1+k}} 
(\CC_{N^{4mn+2n}_{r-k}} 
(\rho(\hat{\gamma} \lambda \beta)\wt{\xi}))(I),$$
with $\rho(\hat{\gamma} \lambda \beta)$ denoting the right translation by $\hat{\gamma} \lambda \beta$; 
$\CC_{N^{4mn+2n}_{r-k}} (\rho(\hat{\gamma} \lambda \beta)\wt{\xi})$ is restricted to $\wt{Sp}_{4mn+2n-2r+2k}(\BA)$,
then we apply the Fourier-Jacobi coefficient $FJ^{\phi_2}_{\psi^1_{n-1+k}}$, 
taking automorphic forms on $\wt{Sp}_{4mn+2n-2r+2k}(\BA)$
to $Sp_{4mn-2r}(\BA)$.

By the cuspidal support of $\wt{\xi}$, 
$\CC_{N^{4mn+2n}_{r-k}} (\wt{\xi})$ is identically zero, unless $r=k$
or $r-k = 2n$. When $r=k$, the corresponding term is zero, because
$FJ^{\phi_2}_{\psi^1_{n-1+r}} 
(\wt{\xi})$ is zero, by Theorem \ref{thm10}.
When $r-k=2n$, the restriction of $\CC_{N^{4mn+2n}_{2n}} (\wt{\xi})$ to $\wt{Sp}_{4mn-2n}(\BA)$
is actually a vector in $\wt{\sigma}_{4(m-1)n+2n}$.
Hence, $FJ^{\phi_2}_{\psi^1_{n-1+k}} 
(\CC_{N^{4mn+2n}_{r-k}} (\wt{\xi}))$ is identically zero,
for $0 \leq k \leq r$,
because $\wt{\sigma}_{4(m-1)n+2n}$ has no nonzero Fourier coefficients $FJ_{\psi^{1}_{n-1}}$, and 
$\ul{p}(\wt{\sigma}_{4(m-1)n+2n})=
[(2n)^{2(m-1)}(2n_1)^{s_1}(2n_2)^{s_2} \cdots (2n_k)^{s_k}]$.
So, when $r-k=2n$, the corresponding term is also zero.

Therefore, we have shown that $\CC_{N^{4mn}_r}(FJ^{\phi}_{\psi^1_{n-1}}(\wt{\xi}))$
is identically zero for any $1 \leq r \leq 2mn$,
and for any
$\wt{\xi} \in \wt{\CE}_{\tau, \wt{\sigma}_{4(m-1)n+2n}}$,
which implies that $\CD_{2n, \psi^{1}}^{4mn+2n}
(\wt{\CE}_{\tau, \wt{\sigma}_{4(m-1)n+2n}})$ is a cuspidal
representation. This completes the proof of (1).

To prove (2), we need to show that for every irreducible component 
$\pi$ of $\CD_{2n, \psi^{1}}^{4mn+2n}
(\wt{\CE}_{\tau, \wt{\sigma}_{4(m-1)n+2n}})$,

(2-1) $\ul{p}(\pi)=[(2n)^{2m-1}(2n_1)^{s_1}(2n_2)^{s_2} \cdots (2n_k)^{s_k}]$;

(2-2) $\pi$ is nearly equivalent to
the residual representation $\CE_{\Delta(\tau, m)}$;

(2-3) $\pi$ has a nonzero Fourier coefficient $FJ_{\psi^{-1}_{n-1}}$.

(2-1) is obvious by Theorem \ref{thm10},
and by Lemma 2.6 of \cite{GRS05} or Lemma 3.1 of \cite{JL13b}.
(2-2) follows easily from \eqref{thm10equ2}, because
the induced representation on the right hand side
of \eqref{thm10equ2} is also an unramified 
component of $\CE_{\Delta(\tau, m)}$.

To show (2-3), as in the proof of Proposition 3.4 of \cite{GJS12},
we need to consider the following integral
\begin{equation}\label{thm5equ2}
\langle \varphi_{\pi}, FJ^{\phi}_{\psi^1_{n-1}}(\wt{\xi}) \rangle
=\int_{[Sp_{4mn}]} \varphi_{\pi}(h)\ol{FJ^{\phi}_{\psi^1_{n-1}}(\wt{\xi})(h)}
dh,
\end{equation}
which is nonzero for some data $\varphi_{\pi} \in \pi$,
$\wt{\xi} \in \wt{\CE}_{\tau, \wt{\sigma}_{4(m-1)n+2n}}$,
since $\pi$ is an irreducible component 
of $\CD_{2n, \psi^{1}}^{4mn+2n}
(\wt{\CE}_{\tau, \wt{\sigma}_{4(m-1)n+2n}})$.

Assume that $\wt{\xi} = \Res_{s=\frac{m+1}{2}} \wt{E}(\phi_s, \cdot)$,
then from \eqref{thm5equ2}, we know that the following integral
is also nonzero for some choice of data:
\begin{equation}\label{thm5equ3}
\langle \varphi_{\pi}, FJ^{\phi}_{\psi^1_{n-1}}(\wt{E}(\phi_s, \cdot)) \rangle
=\int_{[Sp_{4mn}]} \varphi_{\pi}(h)\ol{FJ^{\phi}_{\psi^1_{n-1}}(\wt{E}(\phi_s, \cdot))(h)}
dh.
\end{equation}
Then, by the unfolding in Theorem 3.3 of \cite{GJRS11} (take $m=2n$, $r=n$ there),
the non-vanishing of the integral in \eqref{thm5equ2}
implies the non-vanishing of $FJ_{\psi^{-1}_{n-1}}(\pi)$.

This finishes the proof of (2),
and completes the proof of the theorem.
\end{proof}

The next theorem implies that $\Psi$ is well-defined.
We will prove it in the next section.

\begin{thm}\label{thm6}
$\CD_{2n, \psi^{-1}}^{4mn}(\sigma_{4mn})$
is irreducible, and 
$$
\CD_{2n, \psi^{-1}}^{4mn}(\sigma_{4mn}) \in \CN'_{\wt{Sp}_{4(m-1)n+2n}}(\tau, \psi),
$$
for
any $\sigma_{4mn} \in \CN_{{Sp}_{4mn}}(\tau, \psi)$.
\end{thm}

Now by Theorems \ref{thm7}, \ref{thm5}, \ref{thm6},
we are able to conclude the Part (1) of Theorem \ref{thm1}.

\begin{thm}[Part (1) of Theorem \ref{thm1}]\label{thm11}
There is a surjective map 
$$
\Psi: \CN_{{Sp}_{4mn}}(\tau, \psi) \rightarrow \CN'_{\wt{Sp}_{4(m-1)n+2n}}(\tau, \psi)
$$
$$
\sigma_{4mn} \mapsto \CD_{2n, \psi^{-1}}^{4mn}(\sigma_{4mn}).
$$
\end{thm}

\begin{proof}
Theorem \ref{thm6} implies that $\Psi$ is well-defined.
Theorem \ref{thm7} and Theorem \ref{thm5} imply that
for any $ \wt{\sigma}_{4(m-1)n+2n} \in \CN'_{\wt{Sp}_{4(m-1)n+2n}}(\tau, \psi)$,
take any irreducible component $\pi$ of
$\CD_{2n, \psi^{1}}^{4mn+2n}
(\wt{\CE}_{\tau, \wt{\sigma}_{4(m-1)n+2n}})$,
which is inside $\CN_{{Sp}_{4mn}}(\tau, \psi)$,
then 
\begin{align*}
& \Psi(\pi)\\
= & \CD_{2n, \psi^{-1}}^{4mn} (\pi)\\
\subseteq & \CD_{2n, \psi^{-1}}^{4mn} \circ \CD_{2n, \psi^{1}}^{4mn+2n}
(\wt{\CE}_{\tau, \wt{\sigma}_{4(m-1)n+2n}})\\
= & \wt{\sigma}_{4(m-1)n+2n}.
\end{align*}
Theorem \ref{thm5} also implies that 
$\CD_{2n, \psi^{-1}}^{4mn} (\pi) \neq 0$.
Since
$\wt{\sigma}_{4(m-1)n+2n}$ is irreducible, 
actually we have 
$\CD_{2n, \psi^{-1}}^{4mn} (\pi)
= \wt{\sigma}_{4(m-1)n+2n}$.

Hence $\Psi$ is surjective.
\end{proof}

\section{Proof of Theorem \ref{thm6}}

For any $\sigma_{4mn} \in \CN_{{Sp}_{4mn}}(\tau, \psi)$, 
we know that the Eisenstein
series corresponding to 
$$
\Ind_{P_{2n}(\BA)}^{Sp_{4(m+1)n}(\BA)} \tau \lvert \cdot \rvert^{s} \otimes \sigma_{4mn}
$$
has a simple pole at $s=\frac{m+1}{2}$. Let $\CE_{\tau, \sigma_{4mn}}$ be the residual
representation of $Sp_{4(m+1)n}(\BA)$ generated by
the corresponding residues.

First, we have a similar result as that in Theorem \ref{thm7}.

\begin{thm}\label{thm8}
(1)
$$
\CD_{2n, \psi^{1}}^{4mn+2n} \circ \CD_{2n, \psi^{-1}}^{4(m+1)n}
(\CE_{\tau, \sigma_{4mn}}) \neq 0.
$$

(2)
$$
\CD_{2n, \psi^{1}}^{4mn+2n} \circ \CD_{2n, \psi^{-1}}^{4(m+1)n}
(\CE_{\tau, \sigma_{4mn}})=\sigma_{4mn}.
$$
\end{thm}

\begin{proof}
The proof is very similar to that of Theorem \ref{thm7}.
We omit it here.
\end{proof}

Next, we have a similar result as that in Theorem 5.8 of \cite{GJS12}.

\begin{thm}\label{thm9}
For any $\sigma_{4mn} \in \CN_{{Sp}_{4mn}}(\tau, \psi)$, the automorphic 
representation $\CD_{2n, \psi^{-1}}^{4(m+1)n}
(\CE_{\tau, \sigma_{4mn}})$ is square-integrable. Moreover,
there is an irreducible representation $\wt{\sigma}_{4(m-1)+2n}$,
which is a component of
$$
\CD_{2n, \psi^{-1}}^{4mn}(\sigma_{4mn}) \subset \CN'_{\wt{Sp}_{4(m-1)n+2n}}(\tau, \psi),
$$
such that the representation space of $ \CD_{2n, \psi^{-1}}^{4(m+1)n}
(\CE_{\tau, \sigma_{4mn}})$ has a non-trivial intersection 
with the representation space of $\wt{\CE}_{\tau, \wt{\sigma}_{4(m-1)n+2n}}$.
\end{thm}

\begin{proof}
We follow the constant term calculation in the proof of the Theorem \ref{thm5}.

Recall that $P^{4mn+2n}_r = M^{4mn+2n}_r N^{4mn+2n}_r$ (with $1 \leq r \leq 2mn+n$) is the standard parabolic subgroup
of ${Sp}_{4mn+2n}$ with Levi part $M^{4mn+2n}_r$ isomorphic to $GL_r \times Sp_{4mn+2n-2r}$,
$N^{4mn+2n}_r$ is the unipotent radical, and
$\wt{P}^{4mn+2n}_r$ is the pre-image of $P^{4mn+2n}_r$ in $\wt{Sp}_{4mn+2n}$.
Take any $\xi \in \CE_{\tau, \sigma_{4mn}}$,
we will calculate the constant term of $FJ^{\phi}_{\psi^{-1}_{n-1}}({\xi})$ along 
$\wt{P}^{4mn+2n}_r$, which is denoted by $\CC_{N^{4mn+2n}_r}(FJ^{\phi}_{\psi^{-1}_{n-1}}({\xi}))$.

By Theorem 7.8 of \cite{GRS11}, 
\begin{align}\label{thm9equ1}
\begin{split}
& \CC_{N^{4mn+2n}_r}(FJ^{\phi}_{\psi^{-1}_{n-1}}({\xi}))\\
= & \sum_{k=0}^r \sum_{\gamma \in P^1_{r-k, 1^k}(F) \bs GL_r(F)}
\int_{L(\BA)} \phi_1(i(\lambda)) FJ^{\phi_2}_{\psi^{-1}_{n-1+k}} 
(\CC_{N^{4mn+4n}_{r-k}} ({\xi}))(\hat{\gamma} \lambda \beta) d \lambda,
\end{split}
\end{align}
where $N^{4mn+4n}_{r-k}$ is the unipotent radical of the parabolic 
subgroup $P^{4mn+4n}_{r-k}$ of ${Sp}_{4mn+4n}$ with Levi isomorphic to 
$GL_{r-k} \times Sp_{4mn+4n-2r+2k}$; $P^1_{r-k, 1^k}$ is a subgroup of $GL_r$
consisting of matrices of the form 
$\begin{pmatrix}
g & x\\
0 & z
\end{pmatrix}$, with $z \in U_k$, the standard maximal unipotent subgroup of $GL_k$;
for $a \in GL_j$, $j \leq 2mn+2n$, $\hat{a}=\diag(a, I_{4mn+4n-2j}, a^*)$;
$L$ is a unipotent subgroup, consisting of matrices of the form
$\lambda = \begin{pmatrix}
I_r & 0\\
x & I_n
\end{pmatrix}^{\wedge}$, and $i(\lambda)$ is the last row of $x$;
$\beta=\begin{pmatrix}
0 & I_r\\
I_n & 0
\end{pmatrix}^{\wedge}$; $\phi = \phi_1 \otimes \phi_2$, with 
$\phi_1 \in \CS(\BA^r)$, $\phi_2 \in \CS(\BA^{2mn+n-r})$;
$$FJ^{\phi_2}_{\psi^{-1}_{n-1+k}} 
(\CC_{N^{4mn+4n}_{r-k}} ({\xi}))(\hat{\gamma} \lambda \beta):=
FJ^{\phi_2}_{\psi^{-1}_{n-1+k}} 
(\CC_{N^{4mn+4n}_{r-k}} (\rho(\hat{\gamma} \lambda \beta){\xi}))(I),$$
with $\rho(\hat{\gamma} \lambda \beta)$ denoting the right translation by $\hat{\gamma} \lambda \beta$;
$\CC_{N^{4mn+4n}_{r-k}} (\rho(\hat{\gamma} \lambda \beta){\xi})$ is restricted to ${Sp}_{4mn+4n-2r+2k}(\BA)$,
then we apply the Fourier-Jacobi coefficient $FJ^{\phi_2}_{\psi^{-1}_{n-1+k}}$, 
taking automorphic forms on ${Sp}_{4mn+4n-2r+2k}(\BA)$
to $\wt{Sp}_{4mn+2n-2r}(\BA)$.

By the cuspidal support of ${\xi}$, 
$\CC_{N^{4mn+4n}_{r-k}} ({\xi})$ is identically zero, unless $r=k$
or $r-k = 2n$. When $r=k$, the corresponding term is zero, because
$FJ^{\phi_2}_{\psi^{-1}_{n-1+r}} 
({\xi})$ is zero, by Theorem \ref{thm10}.
When $r-k=2n$, the restriction of $\CC_{N^{4mn+4n}_{2n}} ({\xi})$ to ${Sp}_{4mn}(\BA)$
is actually a vector inside ${\sigma}_{4mn}$.
Hence, $FJ^{\phi_2}_{\psi^{-1}_{n-1+k}} 
(\CC_{N^{4mn+4n}_{r-k}} ({\xi}))$ is not zero for $k=0$,
and is identically zero
for $1 \leq k \leq r$,
because ${\sigma}_{4mn}$ has a nonzero Fourier coefficient $FJ_{\psi^{-1}_{n-1}}$, and 
$\ul{p}({\sigma}_{4mn})=
[(2n)^{2m-1}(2n_1)^{s_1}(2n_2)^{s_2} \cdots (2n_k)^{s_k}]$.

Therefore,
\begin{align}\label{thm9equ2}
\begin{split}
& \CC_{N^{4mn+2n}_{2n}}(FJ^{\phi}_{\psi^{-1}_{n-1}}({\xi}))\\
= & \int_{L(\BA)} \phi_1(i(\lambda)) FJ^{\phi_2}_{\psi^{-1}_{n-1}} 
(\CC_{N^{4mn+4n}_{2n}} ({\xi}))(\lambda \beta) d \lambda.
\end{split}
\end{align}

By similar calculation as in the proof of Lemma \ref{constantterm},
when restricting to $GL_{2n}(\BA) \times {Sp}_{4mn}(\BA)$,
$$
\CC_{N^{4mn+4n}_{2n}} ({\xi}) \in \delta_{P^{4mn+4n}_{2n}}^{\frac{1}{2}} 
\lvert \det \rvert^{-\frac{2m+1}{2}} \tau \otimes \sigma_{4mn}.
$$

As in the proof of Theorem 2.5 \cite{GJS12}, we need to calculate 
the automorphic exponents attached to this non-trivial constant term
(for definition see \Rmnum{1}.3.3 \cite{MW95}). 
We consider the action of 
$$\hat{g}=\diag(g, I_{4mn-2n}, g^*) \in GL_{2n}(\BA) \times \wt{Sp}_{4mn-2n}(\BA).$$
Since $r=2n$, $\beta=\begin{pmatrix}
0 & I_{2n}\\
I_n & 0
\end{pmatrix}^{\wedge}$. 
$\beta \diag(I_n, \hat{g}, I_n) \beta^{-1} = \diag(g, I_{4mn}, g^*) =: \wt{g}$.
Then changing variables in
\eqref{thm9equ2} $\lambda \mapsto \wt{g} \lambda \wt{g}^{-1}$ will give 
a Jacobian $\lvert \det(g) \rvert^{-n}$. 
On the other hand, by Formula (1.4) \cite{GRS11}, 
the action of $\hat{g}$ on $\phi_1$ gives $\gamma_{\psi}(\det(g)) \lvert \det(g) \rvert^{\frac{1}{2}}$. Therefore, the $\hat{g}$ acts by $\tau(g)$ with 
character 
\begin{align*}
& \delta_{P^{4mn+4n}_{2n}}^{\frac{1}{2}} 
\lvert \det(g) \rvert^{-\frac{2m+1}{2}} \lvert \det(g) \rvert^{-n}
\gamma_{\psi}(\det(g)) \lvert \det(g) \rvert^{\frac{1}{2}}\\
= & \gamma_{\psi}(\det(g)) \delta_{P^{4mn+2n}_{2n}}^{\frac{1}{2}} 
\lvert \det(g) \rvert^{-m}.
\end{align*}
Hence, by Langlands square-integrability criterion (Lemma \Rmnum{1}.4.11 \cite{MW95}),
the automorphic 
representation $\CD_{2n, \psi^{-1}}^{4(m+1)n}
(\CE_{\tau, \sigma_{4mn}})$ is square-integrable.
And as a representation of $GL_{2n}(\BA) \times \wt{Sp}_{4mn-2n}(\BA)$,
\begin{equation}\label{thm9equ3}
\CC_{N^{4mn+4n}_{2n}}(\CD_{2n, \psi^{-1}}^{4(m+1)n}
(\CE_{\tau, \sigma_{4mn}})) = \gamma_{\psi} \delta_{P^{4mn+2n}_{2n}}^{\frac{1}{2}} 
\lvert \det \rvert^{-m} \tau \otimes \CD_{2n, \psi^{-1}}^{4mn}
(\sigma_{4mn}).
\end{equation}

By \eqref{thm9equ3}, it is easy to see that any non-cuspidal
irreducible subrepresentation of $\CD_{2n, \psi^{-1}}^{4(m+1)n}
(\CE_{\tau, \sigma_{4mn}})$ must be an irreducible subrepresentation of 
$\wt{\CE}_{\tau, \wt{\sigma}_{4(m-1)n+2n}}$,
for some irreducible subrepresentation 
$\wt{\sigma}_{4(m-1)+2n}$ of
$\CD_{2n, \psi^{-1}}^{4mn}(\sigma_{4mn})$.

To prove 
$\CD_{2n, \psi^{-1}}^{4mn}(\sigma_{4mn}) \subset \CN'_{\wt{Sp}_{4(m-1)n+2n}}(\tau, \psi)$,
we need to show that for every irreducible component 
$\sigma$ of $\CD_{2n, \psi^{-1}}^{4mn}(\sigma_{4mn})$,

(1) $\sigma$ is cuspidal;

(2) $\ul{p}(\sigma)=[(2n)^{2m-2}(2n_1)^{s_1}(2n_2)^{s_2} \cdots (2n_k)^{s_k}]$;

(3) $\sigma$ is nearly equivalent to
the residual representation $\wt{\CE}_{\Delta(\tau, m-1) \otimes \wt{\pi}}$;

(4) $\sigma$ has no nonzero Fourier coefficient $FJ_{\psi^{1}_{n-1}}$.

(1) follows easily from the tower property (Theorem 7.10 \cite{GRS11}).
(2) is implied by Lemma 2.6 \cite{GRS03} or Lemma 3.1 \cite{JL13b}.
(3) can be read out easily from the right hand side of 
\eqref{thm10equ2} by \eqref{prop1equ1} of Proposition \ref{prop1}.
Note that the right hand side of \eqref{thm10equ2} is an unramified component of
$\CE_{\Delta(\tau, m)}$, hence unramified component of
$\sigma_{4mn}$. By Theorem 5.2 of \cite{JL13b},
as a cuspidal representation, $\sigma_{4mn}$ has no nonzero Fourier coefficient
with respect to character $\psi_{[(2n)^{2m-1}(2n_1)^{s_1}(2n_2)^{s_2} \cdots (2n_k)^{s_k}],
\ul{a}}$, where $\ul{a}=\{-1, 1\} \cup \ul{a}'$. Now (4) follows easily 
from Lemma 3.1 \cite{JL13b}.

Therefore, the representation space of $ \CD_{2n, \psi^{-1}}^{4(m+1)n}
(\CE_{\tau, \sigma_{4mn}})$ has a non-trivial intersection 
with the representation space of $\wt{\CE}_{\tau, \wt{\sigma}_{4(m-1)n+2n}}$,
for some component $\wt{\sigma}_{4(m-1)n+2n}$ of
$$
\CD_{2n, \psi^{-1}}^{4mn}(\sigma_{4mn}) \subset \CN'_{\wt{Sp}_{4(m-1)n+2n}}(\tau, \psi).
$$

This completes the proof of the theorem.
\end{proof}

{\textbf{\it Proof of Theorem \ref{thm6}.}}

For any $\sigma_{4mn} \in \CN_{{Sp}_{4mn}}(\tau, \psi)$,
by Theorem \ref{thm8}, 
\begin{equation}\label{thm6equ1}
\CD_{2n, \psi^{1}}^{4mn+2n} \circ \CD_{2n, \psi^{-1}}^{4(m+1)n}
(\CE_{\tau, \sigma_{4mn}})=\sigma_{4mn}.
\end{equation}
By Theorem \ref{thm9}, there is an irreducible representation 
$\wt{\sigma}_{4(m-1)+2n}$,
which is a component of
$\CD_{2n, \psi^{-1}}^{4mn}(\sigma_{4mn}) \subset \CN'_{\wt{Sp}_{4(m-1)n+2n}}(\tau, \psi),$
such that the representation space of $ \CD_{2n, \psi^{-1}}^{4(m+1)n}
(\CE_{\tau, \sigma_{4mn}})$ contains an irreducible
subrepresentation $\pi$ of $\wt{\CE}_{\tau, \wt{\sigma}_{4(m-1)n+2n}}$.
Since $\sigma_{4mn}$ is irreducible,
by \eqref{thm6equ1}, 
\begin{equation}\label{thm6equ2}
\sigma_{4mn} = \CD_{2n, \psi^{1}}^{4mn+2n}(\pi) \subset 
\CD_{2n, \psi^{1}}^{4mn+2n}(\wt{\CE}_{\tau, \wt{\sigma}_{4(m-1)n+2n}}).
\end{equation}
Therefore,
\begin{align*}
\CD_{2n, \psi^{-1}}^{4mn}(\sigma_{4mn}) & \subset
\CD_{2n, \psi^{-1}}^{4mn} \circ \CD_{2n, \psi^{+1}}^{4mn+2n}(\wt{\CE}_{\tau, \wt{\sigma}_{4(m-1)n+2n}})\\
& = \wt{\sigma}_{4(m-1)n+2n},
\end{align*}
by Theorem \ref{thm7}.
Hence, $\CD_{2n, \psi^{-1}}^{4mn}(\sigma_{4mn}) = \wt{\sigma}_{4(m-1)n+2n}$,
irreducible as an element in $\CN'_{\wt{Sp}_{4(m-1)n+2n}}(\tau, \psi)$.

This completes the proof of Theorem \ref{thm6},
showing that $\Psi$ is well-defined. \hfill $\square$

\section{Proof of Part (2) of Theorem \ref{thm1}}

In this section, we will prove that $\Psi$ is injective.
For this, we need to assume that for any $\wt{\sigma}_{4(m-1)n+2n} \in \CN'_{\wt{Sp}_{4(m-1)n+2n}}(\tau, \psi)$,
$\wt{\CE}_{\tau, \wt{\sigma}_{4(m-1)n+2n}}$ is irreducible.

For any $\sigma_{4mn} \in \CN_{{Sp}_{4mn}}(\tau, \psi)$,
by Theorem \ref{thm6}, 
$\CD_{2n, \psi^{-1}}^{4mn}(\sigma_{4mn}) = \wt{\sigma}_{4(m-1)n+2n}
\in \CN'_{\wt{Sp}_{4(m-1)n+2n}}(\tau, \psi)$,
which is irreducible. To show $\Psi$ is injective, we only
need to show that $\sigma_{4mn}$ is uniquely determined by 
$\wt{\sigma}_{4(m-1)n+2n}$.

By \eqref{thm6equ2}, $\sigma_{4mn} = \CD_{2n, \psi^{1}}^{4mn+2n}(\pi) \subset 
\CD_{2n, \psi^{1}}^{4mn+2n}(\wt{\CE}_{\tau, \wt{\sigma}_{4(m-1)n+2n}})$,
where $\pi$ is an irreducible
subrepresentation of $\wt{\CE}_{\tau, \wt{\sigma}_{4(m-1)n+2n}}$.
Since we assume that $\wt{\CE}_{\tau, \wt{\sigma}_{4(m-1)n+2n}}$ is irreducible,
we have that $\pi = \wt{\CE}_{\tau, \wt{\sigma}_{4(m-1)n+2n}}$.
Hence $\sigma_{4mn} = 
\CD_{2n, \psi^{1}}^{4mn+2n}(\wt{\CE}_{\tau, \wt{\sigma}_{4(m-1)n+2n}})$,
which means that $\sigma_{4mn}$ is uniquely determined by 
$\wt{\sigma}_{4(m-1)n+2n}$.

This completes the proof of Part (2) of Theorem \ref{thm1}.

\section{Irreducibility of Certain Descent Representations}

In Theorem \ref{main1part2}, for the residual representation 
$\CE_{\Delta(\tau, m)}$, we have proved that 
$\mathfrak{p}^m(\CE_{\Delta(\tau, m)}) = [(2n)^{2m}]$.
From the proof, and by Lemma 2.6 \cite{GRS03} or Lemma 3.1 \cite{JL13b},
we can see that it has a nonzero Fourier coefficient
attached to the partition $[(2n)1^{4mn-2n}]$ with respect
to the character $\psi_{[(2n)1^{4mn-2n}], -1}$. 
In this section, for any number field $F$,
we show that both $\CE_{\Delta(\tau, m)}$
and $\CD^{4mn}_{2n, \psi^{-1}}(\CE_{\Delta(\tau, m)})$
are irreducible. The result can be stated as follows.

\begin{thm}\label{irre}
Assume that $F$ is any number field.

(1) $\CD^{4mn}_{2n, \psi^{-1}}(\CE_{\Delta(\tau, m)})$ is square-integrable and is in the 
discrete spectrum.

(2) Both $\CE_{\Delta(\tau, m)}$ and $\CD^{4mn}_{2n, \psi^{-1}}(\CE_{\Delta(\tau, m)})$
are irreducible.
\end{thm}

\begin{proof}
\textbf{Proof of Part (1)}.
As in Theorem \ref{thm9}, we follow the constant term calculation in the proof of the Theorem \ref{thm5}.

Recall that $P^{4mn-2n}_r = M^{4mn-2n}_r N^{4mn-2n}_r$ (with $1 \leq r \leq 2mn-n$) is the standard parabolic subgroup
of ${Sp}_{4mn-2n}$ with Levi part $M^{4mn-2n}_r$ isomorphic to $GL_r \times Sp_{4mn-2n-2r}$,
$N^{4mn-2n}_r$ is the unipotent radical, and
$\wt{P}^{4mn-2n}_r$ is the pre-image of $P^{4mn-2n}_r$ in $\wt{Sp}_{4mn-2n}$.
Take any $\xi \in \CE_{\Delta(\tau, m)}$,
we will calculate the constant term of $FJ^{\phi}_{\psi^{-1}_{n-1}}({\xi})$ along 
$\wt{P}^{4mn-2n}_r$, which is denoted by $\CC_{N^{4mn-2n}_r}(FJ^{\phi}_{\psi^{-1}_{n-1}}({\xi}))$.

By Theorem 7.8 of \cite{GRS11}, 
\begin{align}\label{irreequ3}
\begin{split}
& \CC_{N^{4mn-2n}_r}(FJ^{\phi}_{\psi^{-1}_{n-1}}({\xi}))\\
= & \sum_{k=0}^r \sum_{\gamma \in P^1_{r-k, 1^k}(F) \bs GL_r(F)}
\int_{L(\BA)} \phi_1(i(\lambda)) FJ^{\phi_2}_{\psi^{-1}_{n-1+k}} 
(\CC_{N^{4mn}_{r-k}} ({\xi}))(\hat{\gamma} \lambda \beta) d \lambda,
\end{split}
\end{align}
where $N^{4mn}_{r-k}$ is the unipotent radical of the parabolic 
subgroup $P^{4mn}_{r-k}$ of ${Sp}_{4mn}$ with Levi isomorphic to 
$GL_{r-k} \times Sp_{4mn-2r+2k}$; $P^1_{r-k, 1^k}$ is a subgroup of $GL_r$
consisting of matrices of the form 
$\begin{pmatrix}
g & x\\
0 & z
\end{pmatrix}$, with $z \in U_k$, the standard maximal unipotent subgroup of $GL_k$;
for $a \in GL_j$, $j \leq 2mn$, $\hat{a}=\diag(a, I_{4mn-2j}, a^*)$;
$L$ is a unipotent subgroup, consisting of matrices of the form
$\lambda = \begin{pmatrix}
I_r & 0\\
x & I_n
\end{pmatrix}^{\wedge}$, and $i(\lambda)$ is the last row of $x$;
$\beta=\begin{pmatrix}
0 & I_r\\
I_n & 0
\end{pmatrix}^{\wedge}$; $\phi = \phi_1 \otimes \phi_2$, with 
$\phi_1 \in \CS(\BA^r)$, $\phi_2 \in \CS(\BA^{2mn-n-r})$;
$$FJ^{\phi_2}_{\psi^{-1}_{n-1+k}} 
(\CC_{N^{4mn}_{r-k}} ({\xi}))(\hat{\gamma} \lambda \beta):=
FJ^{\phi_2}_{\psi^{-1}_{n-1+k}} 
(\CC_{N^{4mn}_{r-k}} (\rho(\hat{\gamma} \lambda \beta){\xi}))(I),$$
with $\rho(\hat{\gamma} \lambda \beta)$ denoting the right translation by $\hat{\gamma} \lambda \beta$;
$\CC_{N^{4mn}_{r-k}} (\rho(\hat{\gamma} \lambda \beta){\xi})$ is restricted to ${Sp}_{4mn-2r+2k}(\BA)$,
then we apply the Fourier-Jacobi coefficient $FJ^{\phi_2}_{\psi^{-1}_{n-1+k}}$, 
taking automorphic forms on ${Sp}_{4mn-2r+2k}(\BA)$
to $\wt{Sp}_{4mn-2n-2r}(\BA)$.

By the cuspidal support of ${\xi}$, 
$\CC_{N^{4mn}_{r-k}} ({\xi})$ is identically zero, unless $r=k$
or $r-k = 2ln$, $1 \leq l \leq m-1$. When $r=k$, the corresponding term is zero, because
$FJ^{\phi_2}_{\psi^{-1}_{n-1+r}} 
({\xi})$ is zero, by Theorem \ref{thm10}.
When $r-k=2ln$, $1 \leq l \leq m-1$, $FJ^{\phi_2}_{\psi^{-1}_{n-1+k}} 
(\CC_{N^{4mn}_{r-k}} ({\xi}))$ is not zero for $k=0$,
and is identically zero
for $1 \leq k \leq r$,
because $\mathfrak{p}^m(\CE_{\Delta(\tau, m)}) = [(2n)^{2m}]$.

Therefore, $\CC_{N^{4mn-2n}_r}(FJ^{\phi}_{\psi^{-1}_{n-1}}({\xi})) \neq 0$,
only for $r = 2ln$, $1 \leq l \leq m-1$.
And for $1 \leq l \leq m-1$,
\begin{align}\label{irreequ4}
\begin{split}
& \CC_{N^{4mn-2n}_{2ln}}(FJ^{\phi}_{\psi^{-1}_{n-1}}({\xi}))\\
= & \int_{L(\BA)} \phi_1(i(\lambda)) FJ^{\phi_2}_{\psi^{-1}_{n-1}} 
(\CC_{N^{4mn}_{2ln}} ({\xi}))(\lambda \beta) d \lambda.
\end{split}
\end{align}

To prove square-integrability of $\CD^{4mn}_{2n, \psi^{-1}}(\CE_{\Delta(\tau, m)})$, it turns out we only need to consider 
$r=2(m-1)n$, which will be clear from the following discussion.

For $r = 2(m-1)n$,
\begin{align}\label{irreequ5}
\begin{split}
& \CC_{N^{4mn-2n}_{2(m-1)n}}(FJ^{\phi}_{\psi^{-1}_{n-1}}({\xi}))\\
= & \int_{L(\BA)} \phi_1(i(\lambda)) FJ^{\phi_2}_{\psi^{-1}_{n-1}} 
(\CC_{N^{4mn}_{2(m-1)n}} ({\xi}))(\lambda \beta) d \lambda.
\end{split}
\end{align}
By Lemma \ref{constantterm}, 
when restricted to $GL_{2(m-1)n}(\BA) \times Sp_{4n}(\BA)$, 
$$
\CC_{N^{4mn}_{2(m-1)n}} ({\xi}) \in \delta_{P^{4mn}_{2(m-1)n}}^{\frac{1}{2}}
\lvert \det \rvert^{-\frac{m+1}{2}} \Delta(\tau,m-1) \otimes \CE_{\Delta(\tau,1)}.
$$

As in the proof of Theorem 2.5 \cite{GJS12}, to calculate 
the automorphic exponent attached to this non-trivial constant term
(for definition see \Rmnum{1}.3.3 \cite{MW95}),
we need to consider the action of 
$$\hat{g}=\diag(g, I_{2n}, g^*) \in GL_{2(m-1)n}(\BA) \times \wt{Sp}_{2n}(\BA).$$
Since $r=2(m-1)n$, $\beta=\begin{pmatrix}
0 & I_{2(m-1)n}\\
I_n & 0
\end{pmatrix}^{\wedge}$. 
$\beta \diag(I_n, \hat{g}, I_n) \beta^{-1} = \diag(g, I_{4n}, g^*) =: \wt{g}$.
Then changing variables in
\eqref{irreequ5} $\lambda \mapsto \wt{g} \lambda \wt{g}^{-1}$ will give 
a Jacobian $\lvert \det(g) \rvert^{-n}$. 
On the other hand, by Formula (1.4) \cite{GRS11}, 
the action of $\hat{g}$ on $\phi_1$ gives $\gamma_{\psi}(\det(g)) \lvert \det(g) \rvert^{\frac{1}{2}}$. Therefore, the $\hat{g}$ acts by $\tau(g)$ with 
character 
\begin{align*}
& \delta_{P^{4mn}_{2(m-1)n}}^{\frac{1}{2}}
\lvert \det(g) \rvert^{-\frac{m+1}{2}} \lvert \det(g) \rvert^{-n}
\gamma_{\psi}(\det(g)) \lvert \det(g) \rvert^{\frac{1}{2}}\\
= & \gamma_{\psi}(\det(g)) \delta_{P^{4mn-2n}_{2(m-1)n}}^{\frac{1}{2}} 
\lvert \det(g) \rvert^{-\frac{m}{2}}.
\end{align*}

Therefore, as a function on $GL_{2(m-1)n}(\BA) \times \wt{Sp}_{2n}(\BA)$,
\begin{align}\label{irreequ6}
\begin{split}
& \CC_{N^{4mn-2n}_{2(m-1)n}}(FJ^{\phi}_{\psi^{-1}_{n-1}}({\xi}))\\
\in &
\gamma_{\psi} \delta_{P^{4mn-2n}_{2(m-1)n}}^{\frac{1}{2}} 
\lvert \det \rvert^{-\frac{m}{2}} \Delta(\tau,m-1) \otimes \CD^{4n}_{2n,\psi^{-1}}(\CE_{\Delta(\tau,1)}).
\end{split}
\end{align}
By Theorem 2.3 \cite{GJS12}, we know that 
$\CD^{4n}_{2n,\psi^{-1}}(\CE_{\Delta(\tau,1)})$ is an irreducible, genuine, $\psi$-generic, cuspidal automorphic representation of $\wt{Sp}_{2n}(\BA)$, which lifts to $\tau$ with respect to $\psi$.
Hence, as a representation of $GL_{2(m-1)n}(\BA) \times \wt{Sp}_{2n}(\BA)$,
\begin{align}\label{irreequ7}
\begin{split}
& \CC_{N^{4mn-2n}_{2(m-1)n}}(\CD^{4mn}_{2n, \psi^{-1}}(\CE_{\Delta(\tau, m)}))\\
= &
\gamma_{\psi} \delta_{P^{4mn-2n}_{2(m-1)n}}^{\frac{1}{2}} 
\lvert \det \rvert^{-\frac{m}{2}} \Delta(\tau,m-1) \otimes \CD^{4n}_{2n,\psi^{-1}}(\CE_{\Delta(\tau,1)}).
\end{split}
\end{align}

Since, the cuspidal exponent of $\Delta(\tau,m-1)$ is 
$\{(\frac{2-m}{2}, \frac{4-m}{2}, \ldots, \frac{m-2}{2})\}$,
the cuspidal exponent of 
$\CC_{N^{4mn-2n}_{2(m-1)n}}(FJ^{\phi}_{\psi^{-1}_{n-1}}({\xi}))$
is 
$\{(\frac{2-2m}{2}, \frac{4-2m}{2}, \ldots, -1)\}$.
Hence, by Langlands square-integrability criterion \cite[Lemma \Rmnum{1}.4.11]{MW95},
the automorphic 
representation $\CD^{4mn}_{2n, \psi^{-1}}(\CE_{\Delta(\tau, m)})$ is square-integrable
and is in the discrete spectrum.

This completes the proof of Part (1).

\textbf{Proof of Part (2)}.

The proof of irreducibility of $\CE_{\Delta(\tau, m)}$ is similar to that of $\CE_{\Delta(\tau, 1)}$ which is given on Page 982 of \cite{GJS12} and in the proof of Theorem 2.1 of \cite{GRS11}. To show the square-integrable residual representation $\CE_{\Delta(\tau, m)}$ is irreducible, it suffices to show that at each local place $v$,
\begin{equation} \label{equirr1}
\Ind_{P_{2mn}(F_v)}^{Sp_{2mn}(F_v)} \Delta(\tau_v, m) \lvert \cdot \rvert^{\frac{m}{2}}
\end{equation}
has a unique quotient.
Since $\Delta(\tau_v, m)$ is the unique quotient of the following induced representation
\begin{equation*}
\Ind_{Q_{(2n)^m}(F_v)}^{GL_{2mn}(F_v)}\tau_v \lvert \cdot \rvert^{\frac{m-1}{2}} \otimes \tau_v \lvert \cdot \rvert^{\frac{m-3}{2}} \otimes \cdots \otimes \tau_v \lvert \cdot \rvert^{\frac{1-m}{2}},
\end{equation*}
where $Q_{(2n)^m}$ is the parabolic subgroup of $GL_{2mn}$ with Levi subgroup isomorphic to $GL_{2n}^m$.
We just have to show that the following induced representation has a unique quotient
\begin{equation}\label{equirr2}
\Ind_{P_{(2n)^m}(F_v)}^{Sp_{4mn}(F_v)} \tau_v \lvert \cdot \rvert^{\frac{2m-1}{2}} \otimes \tau_v \lvert \cdot \rvert^{\frac{2m-3}{2}} \otimes \cdots \otimes \tau_v \lvert \cdot \rvert^{\frac{1}{2}},
\end{equation}
where $P_{(2n)^m}$ is the parabolic subgroup of $Sp_{4mn}$ with Levi subgroup isomorphic to $GL_{2n}^m$.

Since $\tau_v$ is generic and unitary, by \cite{T86} and \cite{V86}, $\tau_v$ is full parabolic induction from its Langlands data with exponents in the open interval $(-\frac{1}{2}, \frac{1}{2})$. Explicitly, we can assume that 
$$\tau_v \cong \rho_1 \lvert \cdot \rvert^{\alpha_1} \times \rho_2 \lvert \cdot \rvert^{\alpha_2} \times \cdots \times \rho_r \lvert \cdot \rvert^{\alpha_r},$$
where $\rho_i$'s are tempered representations, $\alpha_i \in \BR$ and $\frac{1}{2} > \alpha_1 > \alpha_2 > \cdots > \alpha_r > -\frac{1}{2}$.
Therefore, the induced representation in \eqref{equirr2} can be written as follows:
\begin{align*}
& \rho_1 \lvert \cdot \rvert^{\frac{2m-1}{2}+ \alpha_1} \times \rho_2 \lvert \cdot \rvert^{\frac{2m-1}{2}+\alpha_2} \times \cdots \times \rho_r \lvert \cdot \rvert^{\frac{2m-1}{2}+\alpha_r}\\
\times & \rho_1 \lvert \cdot \rvert^{\frac{2m-3}{2}+ \alpha_1} \times \rho_2 \lvert \cdot \rvert^{\frac{2m-3}{2}+\alpha_2} \times \cdots \times \rho_r \lvert \cdot \rvert^{\frac{2m-3}{2}+\alpha_r}\\
\times & \cdots\\
\times & \rho_1 \lvert \cdot \rvert^{\frac{1}{2}+ \alpha_1} \times \rho_2 \lvert \cdot \rvert^{\frac{1}{2}+\alpha_2} \times \cdots \times \rho_r \lvert \cdot \rvert^{\frac{1}{2}+\alpha_r} \rtimes 1_{Sp_0}.
\end{align*}
Since $\alpha_i \in \BR$ and $\frac{1}{2} > \alpha_1 > \alpha_2 > \cdots > \alpha_r > -\frac{1}{2}$, we can easily see that 
the exponents satisfy
\begin{align*}
& \frac{2m-1}{2}+ \alpha_1 > \frac{2m-1}{2}+\alpha_2 > \cdots > \frac{2m-1}{2}+\alpha_r \\
> & \frac{2m-3}{2}+ \alpha_1 > \frac{2m-3}{2}+\alpha_2 > \cdots > \frac{2m-3}{2}+\alpha_r \\
> & \cdots \\
> & \frac{1}{2}+ \alpha_1 > \frac{1}{2}+\alpha_2 > \cdots > \frac{1}{2}+\alpha_r > 0.
\end{align*}
By Langlands classification, it is easy to see that the induced representation in \eqref{equirr2} has a unique quotient which is the Langlands quotient. 

This completes the proof of irreducibility of $\CE_{\Delta(\tau, m)}$.

The proof of irreducibility of $\CD^{4mn}_{2n, \psi^{-1}}(\CE_{\Delta(\tau, m)})$ is similar to that in Theorem \ref{thm6}. We just sketch all the steps needed.

Recall that $P^{4mn+4n}_{2n} = M^{4mn+4n}_{2n}N^{4mn+4n}_{2n}$ is the parabolic subgroup of $Sp_{4mn+4n}$ with Levi subgroup $M^{4mn+4n}_{2n}$ isomorphic to
$GL_{2n} \times Sp_{4mn}$. 
For any $\phi \in A(N^{4mn+4n}_{2n}(BA)M^{4mn+4n}_{2n}(F) \bs Sp_{4mn+4n}(\BA)_{\tau \otimes \CE_{\Delta(\tau, m)}}$,
the corresponding Eisenstein series defined as follows has a pole at $s=\frac{m+1}{2}$:
$$
E(\phi,s)(g)=\sum_{\gamma\in P^{4mn+4n}_{2n}(F)\bks Sp_{4mn+4n}(F)}\lambda_s \phi(\gamma g).
$$
The resulting residual representation generated by all the residues is actually $\CE_{\Delta(\tau, m+1)}$.

Then, by a similar argument as in the proof of Theorem \ref{thm7}, we get that 
\begin{align}\label{irreequ1}
\begin{split}
\CD_{2n, \psi^{1}}^{4mn+2n} \circ \CD_{2n, \psi^{-1}}^{4(m+1)n}
(\CE_{\Delta(\tau, m+1)}) & \neq 0,\\
\CD_{2n, \psi^{1}}^{4mn+2n} \circ \CD_{2n, \psi^{-1}}^{4(m+1)n}
(\CE_{\Delta(\tau, m+1)})& =\CE_{\Delta(\tau, m)}.
\end{split}
\end{align}

Note that, as indicated at the end of the proof of Theorem \ref{thm7}, the irreducibility of $\CE_{\Delta(\tau, m)}$ plays an essential role in proving the equality in \eqref{irreequ1}.

From Part (1), we see that $\CD^{4mn}_{2n, \psi^1}(\CE_{\Delta(\tau, m)})$ is square-integrable and is in the 
discrete spectrum. For any irreducible component $\pi$ of $\CD^{4mn}_{2n, \psi^1}(\CE_{\Delta(\tau, m)})$,
for any $\phi \in A(N^{4mn+2n}_{2n}(\BA)\wt{M}^{4mn+2n}_{2n}(F) \bs \wt{Sp}_{4mn+2n}(\BA))_{\mu_{\psi}\tau \otimes \pi}$,
the corresponding Eisenstein series defined as follows has a pole at $s=m$. Denote the residual representation generated by all the residues by $\wt{\CE}_{\tau, \pi}$.

Since $\pi$ is irreducible,
also by a similar argument as in the proof of Theorem \ref{thm7}, we get that 
\begin{align}\label{irreequ2}
\begin{split}
\CD_{2n, \psi^{-1}}^{4mn} \circ \CD_{2n, \psi^{1}}^{4mn+2n}
(\wt{\CE}_{\tau, \pi}) & \neq 0,\\
\CD_{2n, \psi^{-1}}^{4mn} \circ \CD_{2n, \psi^{1}}^{4mn+2n}
(\wt{\CE}_{\tau, \pi})& =\pi.
\end{split}
\end{align}

Then, using a similar argument as in the proof of Theorem \ref{thm9}, we have that
there is an irreducible component $\pi$ of
$\CD^{4mn}_{2n, \psi^1}(\CE_{\Delta(\tau, m)})$,
such that the representation space of $\CD^{4mn+4n}_{2n, \psi^1}(\CE_{\Delta(\tau, m+1)})$ has a non-trivial intersection 
with the representation space of $\wt{\CE}_{\tau, \pi}$.
Let $\pi'$ be an irreducible subrepresentation of $\wt{\CE}_{\tau,\pi}$ which is in this intersection.

Since $\CE_{\Delta(\tau, m)}$ is irreducible, by the identity in \eqref{irreequ1} we have
$\CE_{\Delta(\tau, m)} = \CD_{2n, \psi^{1}}^{4mn+2n} (\pi') \subseteq \CD_{2n, \psi^{1}}^{4mn+2n} (\wt{\CE}_{\tau, \pi})$.
Therefore, 
$$\CD_{2n, \psi^{-1}}^{4mn}(\CE_{\Delta(\tau, m)}) \subseteq
\CD_{2n, \psi^{-1}}^{4mn} \circ \CD_{2n, \psi^{1}}^{4mn+2n}
(\wt{\CE}_{\tau, \pi})=\pi,$$
by \eqref{irreequ2}.
Hence, $\CD_{2n, \psi^{-1}}^{4mn}(\CE_{\Delta(\tau, m)}) = \pi$, irreducible.

This completes the proof of the theorem.
\end{proof}

\begin{rmk}\label{rmk4}
Write $\wt{\pi} = \CD^{4n}_{2n,\psi^{-1}}(\CE_{\Delta(\tau,1)})$.
For 
$$\wt{\phi} \in A(N^{4mn-2n}_{2(m-1)n}(\BA) \wt{M}^{4mn-2n}_{2(m-1)n}\bs \wt{Sp}_{4mn-2n}(\BA))_{\mu_{\psi} \Delta(\tau,m-1) \otimes \wt{\pi}},$$
it is easy to see that
the corresponding Eisenstein series has a simple pole at $\frac{m}{2}$. 
Denote the residual representation by 
$\wt{\CE}_{\Delta(\tau,m-1), \wt{\pi}}$.

From the proof of Part (1) of Theorem \ref{irre}
(in particular, \eqref{irreequ7}), it is easy to see that if 
the residual representation 
$\wt{\CE}_{\Delta(\tau,m-1), \wt{\pi}}$ is 
irreducible, then actually we have proved that
$\CD^{4mn}_{2n, \psi^1}(\CE_{\Delta(\tau, m)}) = \wt{\CE}_{\Delta(\tau,m-1), \wt{\pi}}$.
And, with the assumption that $\wt{\CE}_{\Delta(\tau,m-1), \wt{\pi}}$ is 
irreducible, using similar argument as that in Theorem \ref{irre}, we can also prove that 
$\CD^{4mn-2n}_{2n, \psi^{-1}}(\wt{\CE}_{\Delta(\tau,m-1), \wt{\pi}})$ is irreducible, square-integrable and is in the discrete spectrum. Furthermore, since $\CE_{\Delta(\tau, m-1)}$
is also irreducible by Theorem \ref{irre}, we actually have
$$\CD^{4mn-2n}_{2n, \psi^{-1}}(\wt{\CE}_{\Delta(\tau,m-1), \wt{\pi}}) = \CE_{\Delta(\tau, m-1)}.$$
\end{rmk}

\end{document}